\documentclass{amsart}[12pt]
\usepackage{amsmath}
\usepackage{amsfonts}
\usepackage{amssymb}
\usepackage{color}
\usepackage{graphicx}
\usepackage{blkarray}
\usepackage{subfigure}
\usepackage[margin=1.2in]{geometry}

\usepackage[numbers]{natbib} \setcitestyle{open={},close={}}

\allowdisplaybreaks[4]
\usepackage[all]{xy}
\newtheorem{theorem}{Theorem}[section]
\newtheorem{proposition}[theorem]{Proposition}
\newtheorem{lemma}[theorem]{Lemma}
\newtheorem{remark}[theorem]{Remark}

\newtheorem{definition}[theorem]{Definition}
\newtheorem{corollary}[theorem]{Corollary}

\newcommand{\be}{\begin{equation}}
\newcommand{\ee}{\end{equation}}
\newcommand{\bea}{\begin{eqnarray}}
\newcommand{\eea}{\end{eqnarray}}
\newcommand{\ben}{\begin{eqnarray*}}
	\newcommand{\een}{\end{eqnarray*}}

\begin{document}

\title{Moduli spaces of semistable pairs on projective Deligne-Mumford stacks}

\author[Yijie Lin]{Yijie Lin}

\address{School of Mathematics and Statistics\\Fujian Normal University\\Fuzhou 350117, P.R. China}
\email{yjlin@fjnu.edu.cn}

\address{School of Mathematics (Zhuhai)\\Sun Yat-Sen University\\Zhuhai 519082, P.R. China}
\email{yjlin12@163.com}
\thanks{}
\maketitle

\begin{abstract}
We generalize the construction of a moduli space of semistable pairs parametrizing isomorphism classes of morphisms from a fixed coherent sheaf to any sheaf with fixed Hilbert polynomial under a notion of stability to the case of projective Deligne-Mumford stacks. We  study the deformation and obstruction theories of stable pairs,  and then prove the existence of virtual fundamental classes for some cases of dimension two and three. This leads to a definition of Pandharipande-Thomas invariants on three-dimensional smooth projective Deligne-Mumford stacks.
\end{abstract}




\tableofcontents

\section{Introduction}
Gromov-Witten theory, Donaldson-Thomas theory and Pandharipande-Thomas theory  are three important approaches to curve counting  in enumerative geometry via  intersection theories on moduli spaces of  stable maps, ideal sheaves and stable pairs respectively. It is interesting that  these theories are conjectured to be related to each other.
The conjecture of Gromov-Witten/Donaldson-Thomas correspondence was  proposed by MNOP in [\cite{MNOP1,MNOP2}], which predicted that the partition function of GW theory can be equated with the one of DT theory by a change of variables for smooth projective 3-folds. The GW/DT correspondence has been proved in several cases (e.g., [\cite{MNOP1,MNOP2,BP,OP1,MO,MOOP}]) with a formal reduced theory for DT side. In [\cite{PT1}], the authors developed a PT theory   providing a geometric interpretation of the reduced DT theory, and  the  DT/PT correspondence for Calabi-Yau 3-folds was proposed  to be viewed as a wall-crossing formula in the derived category. This DT/PT correspondence has been proved in several approaches  [\cite{Toda,Bri2,ST}]. 
On the other hand, the  GW/PT correspondence  has also been treated in  [\cite{MPP,PP1,PP2,OOP1}].

 As a generalization of the above mainfold cases, the orbifold GW theory and orbifold DT theory for   smooth projective Deligne-Mumford stacks have been investigated  in [\cite{CR1,AGV}] and [\cite{BCY,GT,Zhou1}] respectively. The orbifold GW/DT correspondence was 
proved in some cases (e.g. [\cite{Zong1,RZ1,RZ2,Ros,ZZ}]).  It is natural and expected to have an orbifold  PT theory together with the orbifold GW/PT or DT/PT correspondence. In [\cite{BCR}], the authors follow Toda's method [\cite{Toda}] of applying the notion of a torsion pair to obtain a stacky version of PT stable pairs, and then combine the motivic Hall algebra (cf. [\cite{Bri3}]) and   Behrend's constructible function [\cite{Beh}] to define  orbifold PT invariants for  smooth projective Calabi-Yau 3-orbifolds. The orbifold DT/PT correspondence are  proved there for some special cases. However, without  Calabi-Yau condition, it is well known that one may construct a perfect obstruction theory on some moduli space to define a virtual fundamental class [\cite{BF,LT}], which can be integrated to obtain  orbifold PT invariants.  For this purpose, we shall first aim to  provide an alternative construction of moduli spaces of  PT stable pairs in the stacky sense.

To achieve this goal, there is a need for another way to  generalize the  notion of  PT stable pairs to the stacky case. Now we begin with the notion of stable pairs for the case of a smooth projective variety $X$.
Based on the work of Le Potier [\cite{LeP}], Pandharipande and Thomas [\cite{PT1}] define a stable pair $(F,s)$ where $F$ is a pure sheaf of dimension one and the section $s:\mathcal{O}_{X}\to F$ has cokernel with 0-dimensional support. In a more general setting, the author in [\cite{Wan}] defines a $\delta$-semistable  pair $(F,\varphi)$ consisting of a coherent sheaf $F$ on $X$ and a morphism $\varphi:F_{0}\to F$ under a  notion of stability depending on some choice of a parameter $\delta\in\mathbb{Q}[m]$  where $F_{0}$ is any fixed coherent sheaf.  It is shown in [\cite{Lin18}] that the notion of  $\delta$-stable pairs $(F,\varphi)$ defined  in [\cite{Wan}] actually generalizes the one of PT stable pairs  for  the case when $\deg\,\delta\geq\deg\,P=1$ where $P$ is the Hilbert polynomial of $F$. Therefore, it is natural  to generalize the notion of stability [\cite{Wan,Lin18}] to the case of projective Deligne-Mumford stacks to obtain the stacky PT stable pairs.


Observe that the  stability condition used in [\cite{Wan,Lin18}] is different from the one defined in [\cite{HL1,HL2}] due to the  different framing, where the latter  has been naturally generalized to the stacky case  [\cite{BS}] using the modified Hilbert polynomial defined in [\cite{Nir1}].  In order to construct moduli spaces of  the stacky  PT stable pairs, we generalize the construction of moduli spaces of $\delta$-(semi)stable pairs in [\cite{Wan,Lin18}] to the case of projective Deligne-Mumford stacks by means of techniques developed in [\cite{HL2,Nir1,BS}].
We state this result explicitly as follows. Let $\mathcal{X}$ be a projective Deligne-Mumford stack of dimension $d$ over an algebraically closed field $k$ of characteristic zero with a moduli scheme $\pi:\mathcal{X}\to X$ and a polarization $(\mathcal{E},\mathcal{O}_{X}(1))$, see Definition \ref{polarization}. Assume  $\mathcal{F}_{0}$ is any fixed coherent sheaf  on $\mathcal{X}$. Let $\delta$ be any given stability parameter  which is zero or a rational polynomial with positive leading coefficient and $P$  any given polynomial of degree $\deg\,P\leq d$. We have a contravariant functor
\ben
\mathcal{M}^{(s)s}_{\mathcal{X}/k}(\mathcal{F}_{0},P,\delta): (\mathrm{Sch}/k)^\circ\to(\mathrm{Sets})
\een
where if $S$ is a $k$-scheme of finite type,  $\mathcal{M}^{(s)s}_{\mathcal{X}/k}(\mathcal{F}_{0},P,\delta)(S)$ is the set of isomorphism classes of flat families of $\delta$-(semi)stable pairs  $(\mathcal{F},\varphi)$ on $\mathcal{X}$ with  
modified Hilbert polynomial $P$ parametrized by  $S$, see Definition \ref{moduli-functor}. A central result  for the existence of (fine)  moduli space for this moduli functor is obtained in the following
\begin{theorem} [see Theorem \ref{GIT-quo1} and Theorem \ref{main-result}]\label{moduli-spaces}
	There is a projective scheme $M^{ss}:=M^{ss}_{\mathcal{X}/k}(\mathcal{F}_{0},P,\delta)$ which is a moduli space for the moduli functor $\mathcal{M}^{ss}_{\mathcal{X}/k}(\mathcal{F}_{0},P,\delta)$. Moreover, there is an open subscheme $M^s:=M^s_{\mathcal{X}/k}(\mathcal{F}_{0},P,\delta)$ of $M^{ss}$ which is a fine moduli space for the moduli functor $\mathcal{M}^s_{\mathcal{X}/k}(\mathcal{F}_{0},P,\delta)$. 
\end{theorem}

This theorem provides a moduli space $M^{(s)s}$ of $\delta$-(semi)stable pairs  on $\mathcal{X}$ with  
modified Hilbert polynomial $P$ and implies that the quasi-projective scheme $M^s$ has equipped with a universal family.
Given a coherent sheaf $\mathcal{F}_{0}$ and a polynomial $P$, moduli spaces $M^{(s)s}$ of $\delta$-(semi)stable pairs depend on the stability parameter $\delta$. As in [\cite{Wan}, Section 5] for the case of smooth projective varieties, we have a chamber structure of the stability parameter for the variation of moduli spaces $M^{(s)s}$.
\begin{theorem}[see Theorem \ref{Chamber}]
There are finitely many critical values $\delta^1,\cdots,\delta^t\in\mathbb{Q}[m]$ satisfying 
\ben
\delta^0:=0<\delta^1<\cdots<\delta^t<\delta^{t+1}:=+\infty
\een
such that we have  the chamber structure of the stability parameter as follows:
	\ben
	\xymatrixcolsep{1pc}\xymatrix{
		M^{0}_{\mathcal{X}/k}(\mathcal{F}_{0},P,\delta)\ar[d] \ar[rd] & \cdots &    M^{t}_{\mathcal{X}/k}(\mathcal{F}_{0},P,\delta)\ar[d] \ar[rd] & &\\
		M^{ss}_{\mathcal{X}/k}(\mathcal{F}_{0},P,\delta^0)  &M^{ss}_{\mathcal{X}/k}(\mathcal{F}_{0},P,\delta^1) \cdots& \cdots M^{ss}_{\mathcal{X}/k}(\mathcal{F}_{0},P,\delta^t) &M^{ss}_{\mathcal{X}/k}(\mathcal{F}_{0},P,\delta^{t+1})
	}
	\een
	where $M^{i}_{\mathcal{X}/k}(\mathcal{F}_{0},P,\delta):=M^{ss}_{\mathcal{X}/k}(\mathcal{F}_{0},P,\delta)=M^{s}_{\mathcal{X}/k}(\mathcal{F}_{0},P,\delta)$ for some $\delta\in(\delta^i,\delta^{i+1})$, $i=0,\cdots,t$.
\end{theorem}

Since the moduli space $M^s$ contains as a special case the moduli space of stacky PT stable pairs mentioned above, in order to construct a perfect obstruction theory, we next consider the deformation and obstruction theory of stable pairs in  $M^s$. Suppose that $\mathcal{X}$ is a smooth projective Deligne-Mumford stack. Let $[(\mathcal{F},\varphi)]$ be a point in  $M^s_{\mathcal{X}/k}(\mathcal{F}_{0},P,\delta)$ where $\varphi:\mathcal{F}_{0} \to \mathcal{F}$. Suppose that $\mathcal{A}rt_{k}$ is the category of Artinian local $k$-algebras with residue field $k$.
For $A, A^\prime\in\mathrm{Ob}\mathcal{A}rt_{k}$ and  let the short exact sequence 
\ben
0\to I \to A^\prime\to A\to0
\een
be a small extension, that is, $\mathtt{m}_{A^\prime}I=0$. Assume   $\check{\varphi}_{A}:\mathcal{F}_{0}\otimes_{k} A\to \mathcal{F}_{A}$ is a morphism over $\mathcal{X}_{A}:=\mathcal{X}\times_{\mathrm{Spec}\,k} \mathrm{Spec}\,A$ extending  $\varphi$, where $\mathcal{F}_{A}$ is a coherent sheaf flat over $A$. Let $\mathbf{I}^{\bullet}:=\{\mathcal{F}_{0} \xrightarrow{\varphi} \mathcal{F}\}$ and  $\mathbf{I}_{A}^{\bullet}:=\{\mathcal{F}_{0}\otimes_{k} A\xrightarrow{\check{\varphi}_{A}} \mathcal{F}_{A}\}$ be the complexes concentrated in degree 0 and 1. Then we have $\mathbf{I}_{A}^{\bullet}\otimes k=\mathbf{I}^{\bullet}$.
As obtained in [\cite{Lin18}, Theorem 1.2] for the deformation and obstruction theory of stable pairs on smooth projective varieties,  we generalize this result to our stacky case.
\begin{theorem}[see Theorem \ref{def-ob1}]
	For a given small extension $0\to I \to A^\prime\xrightarrow{\sigma} A\to0$, there is a class
	\ben
	\mathrm{ob}(\check{\varphi}_{A},\sigma)\in\mathrm{Ext}^1(\mathbf{I}^\bullet,\mathcal{F}\otimes I)
	\een
	such that there exists a flat extension of $\check{\varphi}_{A}$ over $\mathcal{X}_{A^\prime}$ if and only if $\mathrm{ob}(\check{\varphi}_{A},\sigma)=0$. If $\mathrm{ob}(\check{\varphi}_{A},\sigma)=0$, the space of extensions  is a torsor under $\mathrm{Hom}(\mathbf{I}^\bullet,\mathcal{F}\otimes I)$.
\end{theorem}

With this result, one can show the existence of virtual fundamental classes for the 2-dimensional case  which is a  generalization of [\cite{Lin18}, Theorem 1.3]. 
\begin{theorem}[see Theorem \ref{vir-exi-2}]
	Let $\mathcal{X}$ be a smooth projective Deligne-Mumford stack of dimension 2 over $\mathbb{C}$. Let $\mathcal{F}_{0}$ be a torsion free sheaf and $\deg\,\delta\geq\deg \,P=1$. Then the moduli space $M^s_{\mathcal{X}/\mathbb{C}}(\mathcal{F}_{0},P,\delta)$ of $\delta$-stable pairs admits a virtual fundamental class.
\end{theorem}
While for the  case of dimension three, it is unsuitable to apply Theorem \ref{def-ob1} to obtain some two-term complex for the obstruction theory (cf. Lemma \ref{def-ob-cmg1} and Remark \ref{3-term}). 
Alternatively, we follow Inaba's approach to the deformation and obstruction theory in [\cite{Inaba}, Section 2].
Explicitly, we  derive a stacky version of [\cite{Inaba}, Proposition 2.3] as follows.
\begin{theorem}[see Theorem \ref{def-ob2}]
	For a given small extension $0\to I \to A^\prime\xrightarrow{\sigma} A\to0$, there is an element
	\ben
	\omega(\mathbf{I}_{A}^{\bullet})\in\mathrm{Ext}^2(\mathbf{I}^\bullet,\mathbf{I}^\bullet\otimes I)
	\een
	such that there exists a flat extension of $\check{\varphi}_{A}$ over $\mathcal{X}_{A^\prime}$ if and only if $\omega(\mathbf{I}_{A}^{\bullet})=0$. If $\omega(\mathbf{I}_{A}^{\bullet})=0$, then  the space of extensions form a torsor under $\mathrm{Ext}^1(\mathbf{I}^\bullet,\mathbf{I}^\bullet\otimes I)$.
\end{theorem}
The similar result as above also holds in the case of a square zero extension  (cf.  Theorem \ref{def-ob3}).
As an application, we study the  case when $\dim\mathcal{X}=3$, $\mathcal{F}_{0}=\mathcal{O}_{\mathcal{X}}$ and polynomials $\delta, P$ satisfying $\deg\,\delta\geq\deg\,P=1$. The corresponding moduli space $\overline{M}^s:=M^s_{\mathcal{X}/\mathbb{C}}(\mathcal{O}_{\mathcal{X}},P,\delta)$ parametrizes orbifold PT stable pairs (cf. Lemma \ref{orbi-PT} and Remark \ref{stacky-PT}). Furtherly, we consider the subfunctor $\mathcal{M}^{s}_{\mathcal{X}/\mathbb{C}}(\mathcal{O}_{\mathcal{X}},\beta,\delta)$ of  $\mathcal{M}^{s}_{\mathcal{X}/\mathbb{C}}(\mathcal{O}_{\mathcal{X}},P,\delta)$, which associates for any $\mathbb{C}$-scheme $S$ of finite type with the set of isomorphism classes of flat families of  $\delta$-(semi)stable pairs $(\mathcal{F},\varphi)$ with the fixed numerical class $\beta\in {\color{red}N_{\leq1}}(\mathcal{X})$  parametrized by $S$ where $P_{\mathcal{E}}(\beta)=P$, see Section 5.3 for more details. Similarly, there is a projective scheme $\overline{M}^{s}_{\beta}:=M^{s}_{\mathcal{X}/\mathbb{C}}(\mathcal{O}_{\mathcal{X}},\beta,\delta)$ $(\subseteq \overline{M}^s)$, which is a fine moduli space for $\mathcal{M}^{s}_{\mathcal{X}/\mathbb{C}}(\mathcal{O}_{\mathcal{X}},\beta,\delta)$, see Theorem \ref{moduli-spaces2}. Assume $\bar{\mathbb{I}}^{\bullet}$ and $\bar{\mathbb{I}}^{\bullet}_{\beta}$ are the universal complexes on $\mathcal{X}\times\overline{M}^{s}$ and $\mathcal{X}\times\overline{M}^{s}_{\beta}$ respectively. Let $\overline{\pi}_{\overline{M}^s}:\mathcal{X}\times \overline{M}^s\to \overline{M}^s$, $\overline{\pi}_{\mathcal{X}}:\mathcal{X}\times \overline{M}^s\to\mathcal{X}$, $\tilde{\pi}_{\overline{M}^s_{\beta}}:\mathcal{X}\times \overline{M}^s_{\beta}\to \overline{M}^s_{\beta}$, and $\tilde{\pi}_{\mathcal{X}}:\mathcal{X}\times \overline{M}^s_{\beta}\to\mathcal{X}$ be the  projections.
To assert the existence of virtual fundamental classes of  $\overline{M}^s$ and $\overline{M}^{s}_{\beta}$, we  construct the perfect obstruction theories  with fixed determinant for these moduli spaces in the sense of [\cite{BF}].
\begin{theorem}[see Theorem \ref{perf-ob1}, Corollary \ref{vir-exi-3} and Theorem \ref{perf-ob2}]
Let $\mathcal{X}$ be a 3-dimensional smooth projective Deligne-Mumford stack  over $\mathbb{C}$. Assume that  polynomials $\delta$ and  $P$ satisfy $\deg\,\delta\geq\deg \,P=1$. 
The maps 
\ben
\Phi: E^\bullet:=R\overline{\pi}_{\overline{M}^{s}*}(R\mathcal{H}om(\bar{\mathbb{I}}^{\bullet},\bar{\mathbb{I}}^{\bullet})_{0}\otimes\overline{\pi}_{\mathcal{X}}^*\omega_{\mathcal{X}})[2]\to\mathbb{L}_{\overline{M}^s}.
\een
and 
	\ben
	\Phi_{\beta}:E^\bullet_{\beta}:=R\tilde{\pi}_{\overline{M}^s_{\beta}*}(R\mathcal{H}om(\bar{\mathbb{I}}^{\bullet}_{\beta},\bar{\mathbb{I}}^{\bullet}_{\beta})_{0}\otimes\tilde{\pi}_{\mathcal{X}}^*\omega_{\mathcal{X}})[2]\to\mathbb{L}_{\overline{M}^s_{\beta}}.
	\een
are perfect obstruction theories for $\overline{M}^s$ and $\overline{M}^{s}_{\beta}$ respectively. And there exist  virtual fundamental classes $[\overline{M}^s]^{\mathrm{vir}}\in A_{\mathrm{vdim}_{1}}(\overline{M}^s)$  and $[ \overline{M}^{s}_{\beta}]^{\mathrm{vir}}\in A_{\mathrm{vdim}_{2}}( \overline{M}^{s}_{\beta})$ where $\mathrm{vdim}_{1}=\mathrm{rk}(E^\bullet)$ and $\mathrm{vdim}_{2}=\mathrm{rk}(E^\bullet_{\beta})$. Here, $\mathbb{L}_{\overline{M}^s}$ and $\mathbb{L}_{\overline{M}^s_{\beta}}$ are truncated cotangent complexes of $\overline{M}^s$ and $\overline{M}^s_{\beta}$ respectively.
\end{theorem}

By integrations against virtual fundamental classes of  $\overline{M}^{s}_{\beta}$, one can define the orbifold descendent PT theory  (cf. Definition \ref{descen-PT}) for 3-dimensional smooth projective Deligne-Mumford stacks. In particular, when $\mathcal{X}$ is a  Calabi-Yau 3-orbifold, our definition of orbifold PT invariants of $\mathcal{X}$ is corresponding to the one in [\cite{BCR}] (cf. Definition \ref{CY-PT1} and Remark \ref{CY-PT2}) by applying [\cite{Beh}, Theorem 4.18] for some moduli space with a symmetric perfect obstruction theory.

It is also natural to study  the stacky version of higher rank PT stable pairs  [\cite{She}] if $\dim\mathcal{X}=3$, $\mathcal{F}_{0}=\mathcal{O}_{\mathcal{X}}(-n)^{\oplus r}$ and polynomials $\delta, P$ satisfying $\deg\,\delta\geq\deg\,P=1$.
As  the orbifold PT invariants defined here is  absolute, we will  explore the relative orbifold PT theory and its degeneration formula elsewhere  based on the work  of [\cite{LW,MPP}]  as in the case of relative DT theory [\cite{Zhou1}]. One may study the issue of rationality for the partition function of orbifold PT invariants (cf. Definition \ref{descen-PT}) as in [\cite{Toda,Bri2,PP3,PP4,PP5}],  and furtherly investigate the conjecture of orbifold GW/PT or DT/PT correspondence as in [\cite{Zhou2,ZZ,BCR}], or  more generally  the  stacky version of GW/PT or DT/PT correspondence proved in [\cite{MPP,PP1,OOP1,Toda,Bri2}]. 

This paper is organized as follows. In Section 2, we recall the definitions of projective Deligne-Mumford stacks with polarizations and the modified Hilbert polynomial together with some relevant facts, and then generalize the notion of semistable pairs and some related results in [\cite{Wan,Lin18}] to the stacky case.  And we recollect the definition of set-theoretic families of coherent sheaves and  several boundedness results in [\cite{Nir1}] and give the notion of flat families of  pairs. We devote Section 3 to prove the boundedness of  families of $\delta$-(semi)stable pairs on projective Deligne-Mumford stacks. In Section 4, we prove Theorem \ref{moduli-spaces}  generalizing the construction of moduli spaces of $\delta$-(semi)stable pairs in [\cite{Wan,Lin18}] and then describe the variation of moduli spaces along of the change of the stability parameter $\delta$. In Section 5, we present two deformation and obstruction theories of stable pairs which are useful for proving the existence of virtual fundamental classes in some 2-dimensional and 3-dimensional cases respectively, and then provide a definition of orbifold Pandharipande-Thomas invariants.

{\bf Acknowledgements.} {The author would like to thank an anonymous referee for his helpful  comments. The author was partially  supported by Chinese Universities Scientific Fund (74120-31610010)
 and Guangdong Basic and Applied Basic Research Foundation (2019A1515110255).

\section{Semistable pairs on projective Deligne-Mumford stacks}
In this section, we first recollect some preliminaries of projective Deligne-Mumford stacks and properties of coherent sheaves on them, and  then recall the definition of the modified Hilbert polynomial and generalize the notion of stability in [\cite{Wan,Lin18}] to give the definition of $\delta$-(semi)stable pairs and some related properties. We also recollect some boundedness results for set-theoretic families of coherent sheaves and introduce the definition of flat families of pairs used in the rest sections.  
For the theory of stacks, one may refer to [\cite{LMB,DM,Vis}]  for more details. 

\subsection{Preliminaries on projective Deligne-Mumford stacks}
In this subsection, we recall the definition of projective Deligne-Mumford stacks and some relevant properties of coherent sheaves on them (see [\cite{AOV,AV,KM,Kre,Nir1,OS03}]). We first make some convention as in [\cite{BS,Nir1}] here as follows. Let $k$ be an algebraically closed field of characteristic zero. All schemes are assumed to be noetherian over $k$, and every variety  is a reduced separated scheme of finite type over $k$. Denote by $S$ the generic base scheme of finite type over $k$. For a Deligne-Mumford $S$-stack and its related properties, one may refer to [\cite{DM}] or
[\cite{Vis}, Appendix]. Every Deligne-Mumford $S$-stack  is assumed to be a separated noetherian Deligne-Mumford stack  of finite type over $S$, and when $S=\mathrm{Spec}\,k$, it is called a Deligne-Mumford stack (over $k$). Let $p:\mathcal{X}\to S$ be the structure morphism of $\mathcal{X}$. By an orbifold we mean a smooth Deligne-Mumford stack of finite type over $k$ with generically trivial stabilizer. 
By [\cite{KM}, Corollary 1.3 (1)], for a Deligne-Mumford $S$-stack $\mathcal{X}$,  we have a separated algebraic space $X$ and a morphism $\pi:\mathcal{X}\to X$. More explicitly, the following properties are stated  in [\cite{AV}, Theorem 2.2.1]: 

(i) The morphism $\pi:\mathcal{X}\to X$ is proper and quasi-finite.

(ii) If $\mathfrak{f}$ is an algebraically closed field, the map $\mathcal{X}(\mathfrak{f})/\mathrm{Isom}\to X(\mathfrak{f})$ is a bijection.

(iii) If $\widehat{S}\to S$ is a flat morphism of schemes, and suppose that $Y\to\widehat{S}$ 
is an algebraic space and $\mathcal{X}\times_{S} \widehat{S}\to Y$ is a morphism,  then $\mathcal{X}\times_{S} \widehat{S}\to Y$  factors uniquely as $\mathcal{X}\times_{S} \widehat{S}\to X\times_{S} \widehat{S}\to Y$.

(iv) $\pi_{*}\mathcal{O}_{\mathcal{X}}=\mathcal{O}_{X}$.\\
Here, we call $X$ a coarse moduli space (or moduli space) of $\mathcal{X}$ , and when $X$ is a scheme, we call it a moduli scheme. 

It is shown in [\cite{AOV}] that Deligne-Mumford $S$-stacks with the above convention are tame. This implies three useful properties: (i) the functor $\pi_{*}:\mathrm{QCoh}(\mathcal{X})\to\mathrm{QCoh}(X)$ is exact and carries coherent sheaves to coherent sheaves by [\cite{AV}, Lemma 2.3.4]; (ii) $\pi_{*}\mathcal{F}$ is flat over $S$ if $\mathcal{F}\in\mathrm{QCoh}(\mathcal{X})$ is a flat sheaf over S by [\cite{Nir1}, Corollary 1.3 (3)]; (iii) for any quasi-coherent sheaf $\mathcal{F}$, we have $H^\bullet(\mathcal{X},\mathcal{F})\cong H^\bullet(X,\pi_{*}\mathcal{F})$ by [\cite{Nir1}, Remark 1.4 (3)]. Here, $\mathrm{QCoh}$ denotes the category of quasi-coherent sheaves.

To begin with, we introduce  the following 
\begin{definition}([\cite{Nir1}, Definition 2.4 and 2.6])
Let $\mathcal{X}$ be a Deligne-Mumford $S$-stack with moduli space $\pi:\mathcal{X}\to X$.
Let $\mathcal{E}$ be a locally free sheaf on $\mathcal{X}$. Define two  functors
\ben
&&F_{\mathcal{E}}:\mathrm{QCoh}(\mathcal{X})\to\mathrm{QCoh}(X),\;\;\;\;\mathcal{F}\mapsto \pi_{*}\mathcal{H}om_{\mathcal{O}_{\mathcal{X}}}(\mathcal{E},\mathcal{F});\\
&&G_{\mathcal{E}}:\mathrm{QCoh}(X)\to\mathrm{QCoh}(\mathcal{X}),\;\;\;\;
F\mapsto\pi^*F\otimes\mathcal{E}.
\een
A locally free sheaf $\mathcal{E}$ is said to be a generator for the quasi-coherent sheaf $\mathcal{F}$ if the left adjoint of the identity $\pi_{*}(\mathcal{F}\otimes\mathcal{E}^\vee)\xrightarrow{id}\pi_{*}(\mathcal{F}\otimes\mathcal{E}^\vee)$, i.e.,
\ben
\theta_{\mathcal{E}}(\mathcal{F}): G_{\mathcal{E}}\circ F_{\mathcal{E}} (\mathcal{F}) =\pi^{*}\pi_{*}\mathcal{H}om_{\mathcal{O}_{\mathcal{X}}}(\mathcal{E},\mathcal{F})\otimes\mathcal{E}\to\mathcal{F}
\een
is surjective. It is a generating sheaf for $\mathcal{X}$ if it is a generator for every quasicoherent sheaf on $\mathcal{X}$.
\end{definition}
Obviously the functor $F_{\mathcal{E}}$ is exact since two functors $\mathcal{H}om(\mathcal{E},\cdot)$ and $\pi_{*}$ are exact. It is shown in [\cite{Nir1}, Remark 2.5 (1)] that $G_{\mathcal{E}}$ is exact if the morphism $\pi$ is flat. And this holds when $\mathcal{X}$ is a flat gerbe over a scheme or a root stack.  Compared with $\theta_{\mathcal{E}}$, one can also define the right adjoint of the identity $\pi^*F\otimes\mathcal{E}\xrightarrow{id}\pi^*F\otimes\mathcal{E}$ as 
\ben
\varphi_{\mathcal{E}}(F): F\to\pi_{*}\mathcal{H}om(\mathcal{E},\pi^*F\otimes\mathcal{E})=F_{\mathcal{E}}(G_{\mathcal{E}}(F))=F\otimes\pi_{*}\mathcal{E}nd_{\mathcal{O}_{\mathcal{X}}}(\mathcal{E})
\een
where the last equality is obtained by  the projection formula  in the following lemma. Thus  $\varphi_{\mathcal{E}}$ is the map  given by tensoring a section with the identity endomorphism, which implies that $\varphi_{\mathcal{E}}$ is injective.
\begin{lemma}
Let $\mathcal{F}$ be a quasi-coherent sheaf on $\mathcal{X}$ and $M$ a quasi-coherent sheaf on $X$. Then we have a projection formula
\ben
\pi_{*}(\pi^*M\otimes\mathcal{F})=M\otimes\pi_{*}\mathcal{F}.
\een
Moreover, it is functorial that if $\phi:\mathcal{F}\to\mathcal{F}^\prime$ is a morphism of quasi-coherent sheaves on $\mathcal{X}$ and $\psi: M\to M^\prime$ is a morphism of quasi-coherent sheaves on $X$, then 
\ben
\pi_{*}(\pi^*\psi\otimes\phi)=\psi\otimes\pi_{*}\phi.
\een
\end{lemma}
\begin{proof}
See the proofs in [\cite{OS03}, Corollary 5.4] and [\cite{Nir1}, Lemma 2.8].
\end{proof}

Next, we recall some more results in [\cite{OS03}, Section 5]. We start with the following
\begin{proposition}([\cite{OS03}, Theorem 5.2])
A locally free sheaf $\mathcal{E}$ on $\mathcal{X}$ is a generating sheaf if and only if for each geometric point of $\mathcal{X}$, the representation of  the stabilizer group at that point on the fiber contains every irreducible representation.
\end{proposition}

\begin{definition}([\cite{EHKV}, Definition 2.9])
Let $\mathcal{X}$ be a stack of finite type over a base scheme $S$. We say $\mathcal{X}$ is a global quotient stack over $S$ if it is isomorphic to a stack $[Z/G]$ where $Z$ is  an algebraic space of finite type over $S$ and $G$ is a flat group scheme over $S$ which is a subgroup scheme of  the general linear group scheme $\mathrm{GL}_{n,S}$ for some $n$.
\end{definition}
\begin{theorem}\label{base-chag-proj}
Let $\mathcal{X}$ be a Deligne-Mumford $S$-stack which is a global quotient stack over $S$, we have

$(i)$  there is a locally free sheaf $\mathcal{E}$ over $\mathcal{X}$ which is a generating sheaf for $\mathcal{X}$;

$(ii)$  let $\pi:\mathcal{X}\to X$ be the moduli space of $\mathcal{X}$ and $f:X^\prime\to X$ a morphism of algebraic spaces, then $X^\prime$ is the moduli space of $\mathcal{X}^\prime:=\mathcal{X}\times_{X}X^\prime$. Moreover,  let $p:\mathcal{X}^\prime\to\mathcal{X}$ be the natural projection, then $p^*\mathcal{E}$ is a generating sheaf for $\mathcal{X}^\prime$.
\end{theorem}
\begin{proof}
See [\cite{OS03}, Proposition 5.3] for the proof of (i), and use [\cite{AOV}, Corollary 3.3 (a)] and  [\cite{OS03}, Lemma 1.3 and Theorem 5.5] for the second part.
\end{proof}
Now, we present the definitions of projective Deligne-Mumford stacks and a family of projective stacks as follows.
\begin{definition}(see [\cite{Kre}, Definition 5.5] and [\cite{Nir1}, Definition 2.23])
A Deligne-Mumford stack $\mathcal{X}$ over $k$ is a (quasi-)projective stack if $\mathcal{X}$ admits a (locally) closed embedding into a smooth proper Deligne-Mumford stack and has a projective moduli scheme.

Let $p:\mathcal{X}\to S$ be a Deligne-Mumford $S$-stack which is a global quotient stack over $S$ with a moduli scheme $X$. We call $p:\mathcal{X}\to S$ a family of projective stacks if $p$ factorizes as $\pi:\mathcal{X}\to X$  followed by a projective morphism $\rho:X\to S$. 
\end{definition}
The notion of projective Deligne-Mumford stacks has some equivalent definitions.

\begin{proposition}([\cite{Kre}, Theorem 5.3])\label{equi-proj-stack}
Let $\mathcal{X}$ be a Deligne-Mumford stack over $k$. The following are equivalent:

$(i)$ $\mathcal{X}$ is (quasi-)projective.

$(ii)$ $\mathcal{X}$ has a (quasi-)projective moduli scheme and possesses a generating sheaf.

$(iii)$ $\mathcal{X}$ is global quotient stack over $k$ with a moduli space which is a 
(quasi-)projective scheme.
\end{proposition}
\begin{remark}
Proposition \ref{equi-proj-stack} shows that Theorem \ref{base-chag-proj} holds for any projective Deligne-Mumford stack and every family of projective stacks. In particular, if $p:\mathcal{X}:=[Z/G]\to S$ is a family of projective stacks, then for any geometric point $s\in S$, we have 

$(i)$ the fibre $\mathcal{X}_{s}=[Z_{s}/G_{s}]$ is a projective Deligne-Mumford stack with a moduli scheme $X_{s}$ which is the fibre of $\rho: X\to S$ over $s$;

$(ii)$ there is a generating sheaf $\mathcal{E}$ for the family of projective stacks $p:\mathcal{X}\to S$, and hence $\mathcal{E}_{s}$ which is the fibre of $\mathcal{E}$ over $s$ is a generating sheaf for $\mathcal{X}_{s}$.
\end{remark}
To conclude this subsection, we introduce the  notion of polarizations for  projective Deligne-Mumford stacks  and the family of projective stacks [\cite{Nir1,BS}].
\begin{definition}\label{polarization}
Let $\mathcal{X}$ be a projective Deligne-Mumford stack over $k$ with a moduli scheme $X$. A polarization of $\mathcal{X}$ is a pair $(\mathcal{E}, \mathcal{O}_{X}(1))$ where $\mathcal{E}$ is a generating sheaf for $\mathcal{X}$ and $\mathcal{O}_{X}(1)$ is a very ample invertible sheaf on $X$ relative to $\mathrm{Spec}\,k$. 

Let $p: \mathcal{X}\to S$ be a family of projective stacks. A relative polarization of $p:\mathcal{X}\to S$ is a pair $(\mathcal{E}, \mathcal{O}_{X}(1))$ where $\mathcal{E}$ is a generating sheaf for $\mathcal{X}$ and $\mathcal{O}_{X}(1)$ is a very ample invertible sheaf on $X$ relative to $S$.
\end{definition}

\subsection{The modified Hilbert polynomial} 
In this subsection, we  recall some properties of coherent sheaves on a projective Deligne-Mumford stack, and then present the definition of the modified Hilbert polynomial and its related stability  [\cite{Nir1}]. Let $\mathcal{X}$ be a projective Deligne-Mumford stack over $k$ with a moduli scheme $\pi:\mathcal{X}\to X$ and a polarization $(\mathcal{E},\mathcal{O}_{X}(1))$. 
\begin{definition}([\cite{HL3}, Definition 1.1.1 and 1.1.2])
Let $\mathcal{F}$ be a coherent sheaf on $\mathcal{X}$. The support of $\mathcal{F}$, denoted by $\mathrm{Supp}(\mathcal{F})$, is defined by the closed substack associated to the ideal $\mathcal{I}$:
\ben
0\to\mathcal{I}\to \mathcal{O}_{\mathcal{X}}\to \mathcal{E}nd_{\mathcal{O}_{\mathcal{X}}}(\mathcal{F}).
\een	
The dimension of $\mathcal{F}$ is defined as the dimension of its support.
$\mathcal{F}$ is called pure of dimension $d$ if $\mathrm{dim}(\mathcal{G})=d$ for every nonzero subsheaf $\mathcal{G}\subset \mathcal{F}$.
\end{definition}
As  in [\cite{HL3}, Definition 1.1.4], there is a unique torsion filtration of any coherent sheaf $\mathcal{F}$:
\ben
0\subseteq T_{0}(\mathcal{F})\subseteq\cdots\subseteq T_{d}(\mathcal{F})=\mathcal{F}
\een
where $d=\mathrm{dim}(\mathcal{F})$ and $T_{i}(\mathcal{F})$ is the maximal subsheaf of $\mathcal{F}$ of dimension $\leq i$. Hence $T_{i}(\mathcal{F})/T_{i-1}(\mathcal{F})$ is zero or pure of dimension $i$. And $\mathcal{F}$ is pure if and only if $T_{d-1}(\mathcal{F})=0$.
\begin{definition}([\cite{HL3}, Definition 1.1.5]) The saturation of a subsheaf $\mathcal{F}^\prime\subset\mathcal{F}$ is the minimal subsheaf $\overline{\mathcal{F}^\prime}$ containing $\mathcal{F}^\prime$ such that $\mathcal{F}/\overline{\mathcal{F}^\prime}$ is pure of dimension $\dim(\mathcal{F})$
or zero.
\end{definition}
\begin{remark}\label{saturation}
As in [\cite{HL3}], the saturation of $\mathcal{F}^\prime$ is also defined to be the kernel of the  surjection
\ben
\mathcal{F}\to\mathcal{F}/\mathcal{F}^\prime\to\left(\mathcal{F}/\mathcal{F}^\prime\right)/T_{\dim\mathcal{F}-1}\left(\mathcal{F}/\mathcal{F}^\prime\right).
\een
\end{remark}
Now we have the following properties of coherent sheaves on $\mathcal{X}$.
\begin{lemma}(see [\cite{Nir1}, Lemma 3.4 and Proposition 3.6] and [\cite{BS}, Proposition 2.22])\label{puretopure}
Let $\mathcal{F}$ be a coherent sheaf on $\mathcal{X}$. Then we have

$(i)$ $\pi(\mathrm{Supp}(\mathcal{F}))=\pi(\mathrm{Supp}(\mathcal{F}\otimes\mathcal{E}^\vee))\supseteq\mathrm{Supp}(F_{\mathcal{E}}(\mathcal{F}))$.

$(ii)$ $F_{\mathcal{E}}(\mathcal{F})$ is zero if and only if $\mathcal{F}$ is zero.

$(iii)$
If  $\mathcal{F}$ is pure of dimension $d$, then $F_{\mathcal{E}}(\mathcal{F})$ is pure of the same dimension. Moreover, $\mathcal{F}$ is pure if and only if $F_{\mathcal{E}}(\mathcal{F})$ is pure.
\end{lemma}

\begin{lemma}([\cite{Nir1}, Corollary 3.7 and 3.8])\label{tor-fil}
Let  $\mathcal{F}$ be a  coherent sheaf on $\mathcal{X}$ of dimension $d$. Then the functor $F_{\mathcal{E}}$ sends the torsion filtration $0\subseteq T_{0}(\mathcal{F})\subseteq\cdots\subseteq T_{d}(\mathcal{F})=\mathcal{F}$ of $\mathcal{F}$ to the torsion filtration of $F_{\mathcal{E}}(\mathcal{F})$, i.e., $F_{\mathcal{E}}(T_{i}(\mathcal{F}))=T_{i}(F_{\mathcal{E}}(\mathcal{F}))$ for $0\leq i\leq d$. Moreover,  if $\mathcal{F}$ is pure, then we have $\pi(\mathrm{Supp}(\mathcal{F}))=\mathrm{Supp}(F_{\mathcal{E}}(\mathcal{F}))$.
\end{lemma}

Next, we have the following definition of the modified Hilbert polynomial.

\begin{definition}([\cite{Nir1}, Definition 3.10]) 
Let $\mathcal{F}$ be a coherent sheaf  of dimension $d$ on $\mathcal{X}$.  The modified Hilbert polynomial of $\mathcal{F}$ is defined as
\ben
P_{\mathcal{E}}(\mathcal{F})(m):=\chi(\mathcal{X},\mathcal{F}\otimes\mathcal{E}^\vee\otimes\pi^*\mathcal{O}_{X}(m))=\chi(X,F_{\mathcal{E}}(\mathcal{F})(m))=P(F_{\mathcal{E}}(\mathcal{F})(m)).
\een
\end{definition}
Since $\dim F_{\mathcal{E}}(\mathcal{F})=\dim\mathcal{F}=d$ by Lemma \ref{puretopure},  the modified Hilbert polynomial can be written as 
\ben
P_{\mathcal{E}}(\mathcal{F})(m)=\sum_{i=0}^{d}\alpha_{\mathcal{E},i}(\mathcal{F})\frac{m^i}{i!}.
\een
The multiplicity of $F_{\mathcal{E}}(\mathcal{F})$ is defined by 
\ben
r(F_{\mathcal{E}}(\mathcal{F}))=\alpha_{\mathcal{E},d}(\mathcal{F}).
\een
As in [\cite{HL3}, Definition 1.2.2], if $\dim\mathcal{F}=d=\dim X$,  the rank of $F_{\mathcal{E}}(\mathcal{F})$ is defined by 
\ben
\mathrm{rk}(F_{\mathcal{E}}(\mathcal{F}))=\frac{\alpha_{\mathcal{E},d}(\mathcal{F})}{\alpha_{d}(\mathcal{O}_{X})},
\een
where $\alpha_{d}(\mathcal{O}_{X})$ is the leading coefficient of  Hilbert polynomial $\chi(\mathcal{O}_{X}(m))=\sum_{i=0}^{d}\alpha_{i}(\mathcal{O}_{X})\frac{m^i}{i!}$.

\begin{definition}([\cite{Nir1}, Definition 3.12 and 3.13])
Define the reduced Hilbert polynomial of a coherent sheaf $\mathcal{F}$ of dimension $d$ as
\ben
p_{\mathcal{E}}(\mathcal{F})=\frac{P_{\mathcal{E}}(\mathcal{F})}{\alpha_{\mathcal{E},d}(\mathcal{F})},
\een
and its slope as
\ben
\hat{\mu}_{\mathcal{E}}(\mathcal{F})=\frac{\alpha_{\mathcal{E},d-1}(\mathcal{F})}{\alpha_{\mathcal{E},d}(\mathcal{F})}.
\een
\end{definition}
Then we have the following definition of stabilities with respect to $p_{\mathcal{E}}$ and  $\hat{\mu}_{\mathcal{E}}$.
\begin{definition}([\cite{HL3}, Definition 1.2.4])
A coherent sheaf $\mathcal{F}$  is $p_{\mathcal{E}}$-semistable $($respectively $\hat{\mu}_{\mathcal{E}}$-semistable$)$ if $\mathcal{F}$ is pure and  for every proper subsheaf $\mathcal{F}^\prime\subset\mathcal{F}$ one has $p_{\mathcal{E}}(\mathcal{F}^\prime)\leq p_{\mathcal{E}}(\mathcal{F})$ $($respectively $\hat{\mu}_{\mathcal{E}}(\mathcal{F}^\prime)\leq\hat{\mu}_{\mathcal{E}}(\mathcal{F})$$)$, and it is $p_{\mathcal{E}}$-stable  $($respectively $\hat{\mu}_{\mathcal{E}}$-stable$)$ if it is $p_{\mathcal{E}}$-semistable $($respectively $\hat{\mu}_{\mathcal{E}}$-semistable$)$ and the corresponding inequality is strict.
\end{definition}

As in [\cite{HL3}, Theorem 1.6.7], for a pure sheaf $\mathcal{F}$, there is a unique Harder-Narasimhan filtration with respect to the $\hat{\mu}_{\mathcal{E}}$-stability:
\ben
0=\mathcal{F}_{0}\subsetneqq\mathcal{F}_{1}\subsetneqq\mathcal{F}_{2}\subsetneqq\cdots\subsetneqq\mathcal{F}_{l}=\mathcal{F}
\een
such that the factors $\mathcal{F}_{i}/\mathcal{F}_{i-1}$  for $i=1,\cdots,l$ are $\hat{\mu}_{\mathcal{E}}$-semistable and 
\ben
\hat{\mu}_{\mathcal{E},\mathrm{max}}(\mathcal{F}):=\hat{\mu}_{\mathcal{E}}(\mathcal{F}_{1})>\hat{\mu}_{\mathcal{E}}(\mathcal{F}_{2}/\mathcal{F}_{1})>\cdots>\hat{\mu}_{\mathcal{E}}(\mathcal{F}_{l}/\mathcal{F}_{l-1}):=\hat{\mu}_{\mathcal{E},\mathrm{min}}(\mathcal{F}).
\een
Similarly, one also has a unique Harder-Narasimhan filtration with respect to the $p_{\mathcal{E}}$-stability.
\subsection{Semistable pairs}
We  generalize the notion of semistable pairs  on smooth projective varieties [\cite{Wan,Lin18}] to the case of projective Deligne-Mumford stacks in this subsection. Let $\mathcal{X}$ be a projective Deligne-Mumford stack over $k$ with a moduli scheme $\pi:\mathcal{X}\to X$ and a polarization $(\mathcal{E},\mathcal{O}_{X}(1))$. 
Let $\mathcal{F}_{0}$ be a fixed coherent sheaf on $\mathcal{X}$ and $\delta\in\mathbb{Q}[m]$ be $0$ or a polynomial with positive leading coefficient.

\begin{definition}\label{def-pair}
A pair $(\mathcal{F},\varphi)$  on $\mathcal{X}$ consists of  a coherent sheaf $\mathcal{F}$  and a morphism $\varphi:\mathcal{F}_{0}\to\mathcal{F}$.  A morphism of pairs $\phi:(\mathcal{F},\varphi)\to(\mathcal{G},\psi)$ is a morphism of sheaves $\phi:\mathcal{F}\to\mathcal{G}$ such that
there is an element $\lambda\in k$ making the following  diagram commute
\ben
\xymatrix{
	\mathcal{F}_{0}\ar[d]_{\varphi} \ar[r]^{\lambda \cdot \mathrm{id}} &  \mathcal{F}_{0} \ar[d]^{\psi}  \\
	\mathcal{F}\ar[r]_{\phi}&\mathcal{G}  &
}
\een
A subpair $(\mathcal{F}^\prime,\varphi^\prime)$ of $(\mathcal{F},\varphi)$ consists of a coherent subsheaf $\mathcal{F}^\prime\subset\mathcal{F}$ and a morphism $\varphi^\prime:\mathcal{F}_{0}\to\mathcal{F}^\prime$ satisfying
$i\circ\varphi^\prime=\varphi$ if $\mathrm{im}\varphi\subset\mathcal{F}^\prime$, and $\varphi^\prime=0$ otherwise, where $i$ denotes the inclusion $\mathcal{F}^\prime\hookrightarrow\mathcal{F}$. A quotient pair $(\mathcal{F}^{\prime\prime},\varphi^{\prime\prime})$ consists of a coherent quotient sheaf $q:\mathcal{F}\to\mathcal{F}^{\prime\prime}$ and a morpism
$\varphi^{\prime\prime}=q\circ\varphi:\mathcal{F}_{0}\to\mathcal{F}^{\prime\prime}$.
\end{definition}
A pair $(\mathcal{F},\varphi)$ is said to be of dimension $d$ if $\dim\mathcal{F}=d$. We say a pair $(\mathcal{F},\varphi)$ is  pure if $\mathcal{F}$ is pure.
Let $P$ be a polynomial of degree $d$, we call  a pair $(\mathcal{F},\varphi)$ of type $P$ if  $P_{\mathcal{E}}(\mathcal{F})=P$.
\begin{lemma}([\cite{HL2}, Lemma 1.5])\label{iso-pair}
The set  $\mathrm{Hom}((\mathcal{F},\varphi), (\mathcal{G},\psi))$ of morphisms of pairs is a linear subspace of $\mathrm{Hom}(\mathcal{F},\mathcal{G})$. If $\phi:(\mathcal{F},\varphi)\to(\mathcal{G},\psi)$ is an isomorphism, the factor $\lambda$ in the definition satisfies $\lambda\in k^*$. In particular, the isomorphism
$\phi_{0}=\lambda^{-1}\phi$ satisfies $\phi_{0}\circ\varphi=\psi$.
\end{lemma}	
\begin{remark}\label{Hom-set}
As in [\cite{BS}, Remark 3.5], we have the cartesian diagram
\ben
\xymatrix{
	W\ar[d] \ar[r] & k \ar[d]^{\cdot \psi}  \\
\mathrm{Hom}(\mathcal{F},\mathcal{G})\ar[r]_-{\circ\varphi}&\mathrm{Hom}(\mathcal{F}_{0},\mathcal{G}) &
}
\een
 and 
 \[
W\cong
 \left\{
 \begin{aligned}
 &\mathrm{Hom}((\mathcal{F},\varphi), (\mathcal{G},\psi)), \;\;\;\;\;\;\qquad\mathrm{if} \;\; \psi\neq0; \\
 & \mathrm{Hom}((\mathcal{F},\varphi), (\mathcal{G},\psi))\times k, \qquad\mathrm{ otherwise}.
 \end{aligned}
 \right.\]
\end{remark}
Define the Hilbert polynomial of a pair $(\mathcal{F},\varphi)$ as
\ben
P_{(\mathcal{F},\varphi)}=P_{\mathcal{E}}(\mathcal{F})+\epsilon(\varphi)\delta
\een
and the reduced Hilbert polynomial of this pair by
\ben
p_{(\mathcal{F},\varphi)}=p_{\mathcal{E}}(\mathcal{F})+\frac{\epsilon(\varphi)\delta}{r(F_{\mathcal{E}}(\mathcal{F}))}
\een
where
\[
\epsilon(\varphi)=
\left\{
\begin{aligned}
&1, \;\;\mbox{if} \;\; \varphi\neq0; \\
& 0, \;\;\mbox{otherwise}.
\end{aligned}
\right.\]
\begin{remark}
As in [\cite{Lin18}], a short exact sequence of pairs,
\ben
0\to(\mathcal{F}^\prime,\varphi^\prime)\xrightarrow{i}(\mathcal{F},\varphi)\xrightarrow{q}(\mathcal{F}^{\prime\prime},\varphi^{\prime\prime})\to0
\een 
consists of a short exact sequence of sheaves $0\to\mathcal{F}^\prime\to\mathcal{F}\to\mathcal{F}^{\prime\prime}\to0$ such that $(\mathcal{F}^\prime,\varphi^\prime)$ is a subpair and $(\mathcal{F}^{\prime\prime},\varphi^{\prime\prime})$ is the corresponding quotient pair. Here, $\mathcal{F}^{\prime\prime}=\mathcal{F}/\mathcal{F}^\prime$,
$\varphi^{\prime\prime}=q\circ\varphi$ if $\varphi^{\prime}=0$, and $\varphi^{\prime\prime}=0$ if $\mathrm{im}\varphi\subset\mathcal{F}^\prime$.
Since the modified Hilbert polynomial is additive in a short exact sequence of coherent sheaves [\cite{Nir1}, Remark 3.11 (2)] and  $\epsilon(\varphi)=\epsilon(\varphi^{\prime})+\epsilon(\varphi^{\prime\prime})$, it is obviously that the Hilbert polynomial is also additive in a short exact sequence of pairs.
\end{remark}	
We present a definition of $\delta$-(semi)stable pairs on a projective Deligne-Mumford stack.
\begin{definition}\label{semi-sub}
A pair $(\mathcal{F},\varphi)$ is $\delta$-semistable if $\mathcal{F}$ is pure and $p_{(\mathcal{F}^\prime,\varphi^\prime)}\leq p_{(\mathcal{F},\varphi)}$ for every proper subpair $(\mathcal{F}^\prime,\varphi^\prime)$. We call $(\mathcal{F},\varphi)$ a $\delta$-stable pair if it is $\delta$-semistable and the inequality is strict.
\end{definition}
\begin{remark}\label{semi-quo}
As in [\cite{HL3}, Proposition 1.2.6], Definition \ref{semi-sub} can be shown to be equivalent to the statement: a pair $(\mathcal{F},\varphi)$ is $\delta$-(semi)stable if and only if  $\mathcal{F}$ is pure and $p_{(\mathcal{F}^{\prime\prime},\varphi^{\prime\prime})}(\geq) p_{(\mathcal{F},\varphi)}$ for every proper purely quotient pair $(\mathcal{F}^{\prime\prime},\varphi^{\prime\prime})$ of dimension $\dim\mathcal{F}$.
\end{remark}
Notice that when $\varphi=0$ or $\delta=0$, the $\delta$-(semi)stablility of a pair $(\mathcal{F},\varphi)$ is equivalent to $p_{\mathcal{E}}$-(semi)stablility of the coherent sheaf $\mathcal{F}$.
We say a pair $(\mathcal{F},\varphi)$ is nondegenerate if $\varphi\neq0$ as in [\cite{Lin18}]. Since most of the following results in this subsection are straightforward generalizations of those in [\cite{Lin18}, Section 2], some proofs are omitted.
\begin{lemma}\label{no-str-semi}
When $\mathrm{deg}\,\delta\geq\mathrm{deg}\,P$, every nondegenerate $\delta$-semistable pair of type $P$ is $\delta$-stable.
\end{lemma}
\begin{lemma}\label{equ-sta}
Let $\mathcal{F}$ be a pure coherent sheaf of dimension $d$ with the modified Hilbert polynomial $P_{\mathcal{E}}(\mathcal{F})=P$, and assume $\mathrm{deg}\,\delta\geq\mathrm{deg}\,P=d$. Then a pair $(\mathcal{F},\varphi)$ is $\delta$-stable if and only if for every proper subpair $(\mathcal{F}^\prime,\varphi^\prime)$,
\ben
\frac{P_{\mathcal{E}}(\mathcal{F}^{\prime})}{2r(F_{\mathcal{E}}(\mathcal{F}^{\prime}))-\epsilon(\varphi^{\prime})}<\frac{P_{\mathcal{E}}(\mathcal{F})}{2r(F_{\mathcal{E}}(\mathcal{F}))-\epsilon(\varphi)}.
\een
\end{lemma}

\begin{lemma}\label{orbi-PT}
If $\mathrm{deg}\,\delta\geq\mathrm{deg}\,P$, then a nondegenerate pair $(\mathcal{F},\varphi)$ of type $P$ is $\delta$-stable if and only if $\mathcal{F}$ is pure and $\mathrm{dim}\, \mathrm{coker}\varphi<\mathrm{deg}\,P$.
\end{lemma}
\begin{remark}\label{stacky-PT}
Lemma \ref{orbi-PT} implies that when $\mathcal{F}_{0}=\mathcal{O}_{\mathcal{X}}$ and $\deg\,P=1\leq \deg\,\delta$,    a nondegenerate $\delta$-sable pair $(\mathcal{F},\varphi)$ of type $P$ is a stable pair in the sense of Pandharipande-Thomas  [\cite{PT1}, Lemma 1.3].
\end{remark}

\begin{lemma}\label{stable-iso}
Suppose $\phi:(\mathcal{F},\varphi)\to(\mathcal{G},\psi)$ is a nonzero morphism of pairs. If $(\mathcal{F},\varphi)$ and $(\mathcal{G},\psi)$ are $\delta$-semistable pairs of dimension $d$, then $p_{(\mathcal{F},\varphi)}\leq p_{(\mathcal{G},\psi)}$. Suppose $(\mathcal{F},\varphi)$ and $(\mathcal{G},\psi)$ are $\delta$-stable with the same reduced Hilbert polynomial, then $\phi$ induces an isomorphism between $\mathcal{F}$ and $\mathcal{G}$. In particular, we have $\mathrm{End}((\mathcal{F},\varphi))\cong k$ for a $\delta$-stable pair $(\mathcal{F},\varphi)$.
\end{lemma}

\begin{proposition}[Harder-Narasimhan filtration]\label{HN-pair}
Let $(\mathcal{F},\varphi)$ be a pair and $\mathcal{F}$ be pure of dimension $d$. Then there exists a unique filtration by subpairs
\ben
0\subsetneqq(\mathcal{G}_{1},\varphi_{1})\subsetneqq(\mathcal{G}_{2},\varphi_{2})
\subsetneqq\cdots\subsetneqq(\mathcal{G}_{l},\varphi_{l})=(\mathcal{F},\varphi)
\een
such that the factors $\mathrm{gr}_{i}=(\mathcal{G}_{i},\varphi_{i})/(\mathcal{G}_{i-1},\varphi_{i-1})$ for $i=1,\cdots,l$, are $\delta$-semistable pair of dimension $d$ with the reduced Hilbert polynomials $p_{\mathrm{gr}_{i}}$ satisfying
\ben
p_{\mathrm{gr}_{1}}>\cdots>p_{\mathrm{gr}_{l}}.
\een
\end{proposition}
\begin{proof}
Given a pair $(\mathcal{F},\varphi)$, any subpair $(\mathcal{F}^\prime,\varphi^\prime)$ is actually determined by the subsheaf $\mathcal{F}^\prime\subseteq\mathcal{F}$ since $\varphi^\prime$ is determined by the given morphism $\varphi$ and comparing  $\mathcal{F}^\prime$ with  $\mathrm{im}\varphi$ by definition of a subpair.
With this point of view, the proof is completed by using the similar  argument in the proof of  [\cite{HL3}, Theorem 1.3.4]  with the (reduced) Hilbert polynomial of pairs and  the  multiplicity of $F_{\mathcal{E}}(\mathcal{G})$  for any subsheaf $\mathcal{G}$ of $\mathcal{F}$. Alternatively, one may refer to  the similar argument in the proof of [\cite{Sha}, Theorem 1] for a pure sheaf as in [\cite{Lin18},  Proposition 2.12].
\end{proof}
\begin{remark}\label{HN-factor}
For a nondegenerate pair $(\mathcal{F},\varphi)$, let $\mathrm{gr}_{i}:=(\mathrm{gr}_{i}\mathcal{F},\mathrm{gr}_{i}\varphi)$ for $i=1,\cdots,l$, it is obviously that there is only one nonzero $\mathrm{gr}_{i}\varphi$ by the definition of subpairs and quotient pairs. When
$\deg\delta\geq d$, only $\mathrm{gr}_{1}\varphi=\varphi_{1}$ is nonzero.
\end{remark}
\begin{proposition} [Jordan-H$\ddot{\mathrm{o}}$lder filtration]
Let $(\mathcal{F},\varphi)$ be a $\delta$-semistable pair. There is a filtration
\ben
0\subsetneqq(\mathcal{F}_{1},\varphi_{1})\subsetneqq(\mathcal{F}_{2},\varphi_{2})
\subsetneqq\cdots\subsetneqq(\mathcal{F}_{l},\varphi_{l})=(\mathcal{F},\varphi)
\een
such that all factors $\mathrm{gr}_{i}=(\mathcal{F}_{i},\varphi_{i})/(\mathcal{F}_{i-1},\varphi_{i-1})$ for $i=1,\cdots, l$,  are $\delta$-stable with the same reduced Hilbert polynomial $p_{(\mathcal{F},\varphi)}$. 
Here, $\mathrm{gr}(\mathcal{F},\varphi)=\oplus_{i}\mathrm{gr}_{i}$ does not depend on the choice of filtration.
\end{proposition}
\begin{remark}
As in [\cite{Wan}], Jordan-H$\ddot{o}$lder filtration induces a homomorphism $\mathrm{gr}(\varphi):\mathcal{F}_{0}\to\mathrm{gr}(\mathcal{F},\varphi)$ which is nontrivial for a nondegenerate pair $(\mathcal{F},\varphi)$ and its image is contained in exactly one summand of $\mathrm{gr}(\mathcal{F},\varphi)$.
\end{remark}	
\begin{definition}
Two $\delta$-semistable pair $(\mathcal{F}_{1},\varphi_{1})$ and $(\mathcal{F}_{2},\varphi_{2})$ with the same reduced Hilbert polynomial are called $\mathrm{S}$-equivalent if $\mathrm{gr}(\mathcal{F}_{1},\varphi_{1})\cong\mathrm{gr}(\mathcal{F}_{2},\varphi_{2})$. 
\end{definition}
\subsection{Families of coherent sheaves and  pairs}
We first recall the definition of a set-theoretic family of coherent sheaves as in [\cite{Kle}, Section 1.12] and its relevant boundedness results in [\cite{Nir1}, Section 4], and then give the notion of flat families of  pairs. Let $p:\mathcal{X}\to S$ be a family of projective stacks with a relative polarization $(\mathcal{E},\mathcal{O}_{X}(1))$. Suppose $s\in S$ and $K$ is a field extension of residue field $k(s)$, a coherent sheaf on a fiber of $p$ is defined to be a coherent sheaf $\mathcal{F}_{K}$ on $\mathcal{X}_{K}:=\mathcal{X}\times_{S} \mathrm{Spec}\,K$. Given two extensions $K_{1}$ and $K_{2}$,  two coherent sheaves $\mathcal{F}_{K_{1}}$ and $\mathcal{F}_{K_{2}}$ are equivalent if there are $k(s)$-homomorphisms of $K_{1}$, $K_{2}$ to a third extension $K_{3}$ such that $\mathcal{F}_{K_{1}}\otimes_{k(s)}K_{3}$ and $\mathcal{F}_{K_{2}}\otimes_{k(s)}K_{3}$ are isomorphic.

\begin{definition}([\cite{Kle}, Section 1.12] or [\cite{Nir1}, Definition 4.9 and 4.10])
A set-theoretic family of coherent sheaves on $p:\mathcal{X}\to S$ is a set of coherent sheaves defined on the fibers of $p$. A set-theoretic family $\mathfrak{F}$ of coherent sheaves on $\mathcal{X}$ is bounded if there is an $S$-scheme $T$ of finite type and a coherent sheaf $\mathcal{G}$ on $\mathcal{X}_{T}:=\mathcal{X}\times_{S}T$ such that every sheaf in $\mathfrak{F}$ is contained in $\{\mathcal{G}|_{\mathcal{X}\times_{S}\mathrm{Spec}\,k(t)}|t\in T\}$.
\end{definition}

\begin{definition} (see [\cite{HL3}, Definition 1.7.1 and 1.7.3] and [\cite{Nir1}, Definition 4.2]) Let $X$ be a projective scheme over $k$ with a very ample line bundle $\mathcal{O}_{X}(1)$.
	A coherent sheaf $F$ on  $X$ is said to be $m$-regular if for every $i>0$ we have 
	$H^i(X,F(m-i))=0$. The Mumford-Castelnuovo regularity of $F$ is defined to be the least  integer $m$ such that $F$ is $m$-regular. Let $\mathcal{X}$ be a projective Deligne-Mumford stack over $k$  with a polarization $(\mathcal{E},\mathcal{O}_{X}(1))$. A coherent sheaf $\mathcal{F}$ on $\mathcal{X}$ is defined to be $m$-regular if $F_{\mathcal{E}}(\mathcal{F})$ on $X$ is  $m$-regular. Denote by $\mathrm{reg}_{\mathcal{E}}(\mathcal{F})$ the Mumford-Castelnuovo regularity of $F_{\mathcal{E}}(\mathcal{F})$.
\end{definition}
A criterion for   boundedness of a set-theoretic family of coherent sheaves is Kleiman criterion for stacks  [\cite{Nir1}, Theorem 4.12]. It implies that if a set-theoretic family $\mathfrak{F}$ is bounded, then the set of modified Hilbert polynomials $P_{\mathcal{E}_{K}}(\mathcal{F}_{K})$ for $\mathcal{F}_{K}\in\mathfrak{F}$ is finite, and there exists an integer $m\geq0$ such that every coherent sheaf $\mathcal{F}_{K}$ is $m$-regular. Kleiman criterion for stacks is also used in proving the following stacky version of Grothendieck lemma.
\begin{lemma}([\cite{Nir1}, Lemma 4.13 and Remark 4.14])\label{Gro-lem}
Let $\mathcal{X}$ be a projective Deligne-Mumford stack over $k$ with a moduli scheme $\pi:\mathcal{X}\to X$ and a polarization $(\mathcal{E},\mathcal{O}_{X}(1))$.  Let $P$ be an  polynomial of degree $d\in[0,\mathrm{dim}(X)]$ and $\bar{\rho}$ an integer. There exists a constant
$C=C(P,\bar{\rho})$ such that if $\mathcal{F}$ is a coherent sheaf of dimension $d$ on $\mathcal{X}$ with $P_{\mathcal{E}}(\mathcal{F})=P$ and $\mathrm{reg}_{\mathcal{E}}(\mathcal{F})\leq \bar{\rho}$, then for every  purely $d$-dimensional quotient $\mathcal{F}^\prime$, we have $\hat{\mu}_{\mathcal{E}}(\mathcal{F}^\prime)\geq C$. Moreover, the family of purely $d$-dimensional quotients $\mathcal{F}_{i}^\prime$, $i\in I$ $($for some set of indices $I$$)$ with $\hat{\mu}_{\mathcal{E}}(\mathcal{F}_{i}^\prime)$ bounded from above is bounded. 

The similar statement as above is true, that is, for every  purely saturated subsheaf
$\mathcal{F}^\prime$, the slope $\hat{\mu}_{\mathcal{E}}(\mathcal{F}^\prime)$ is bounded from above, and the family of  pure subsheaves $\mathcal{F}^\prime_{i}\subseteq\mathcal{F}$, $i\in I$  with $\hat{\mu}_{\mathcal{E}}(\mathcal{F}_{i}^\prime)$ bounded from below such that the quotient $\mathcal{F}/\mathcal{F}^\prime_{i}$ is pure  of dimension $d$, is bounded.
\end{lemma}
Boundedness of a family of coherent sheaves on projective Deligne-Mumford stacks is proved to be equivalent to boundedness of the one on projective schemes under some conditions in the following
\begin{proposition}([\cite{Nir1}, Corollary 4.17])
Let $p:\mathcal{X}\to S$ be a family of projective stacks with a relative polarization $(\mathcal{E},\mathcal{O}_{X}(1))$. Let $\mathfrak{F}$ be a set-theoretic family of coherent sheaves on the fibers of $p$. Then the family $\mathfrak{F}$ is bounded if and only if $F_{\mathcal{E}}(\mathfrak{F})$ is bounded.	
\end{proposition}
Here is a  useful boundedness result as follows.
\begin{theorem}([\cite{Nir1}, Theorem 4.27 (1)])\label{fam-bound1}
Let $p:\mathcal{X}\to S$ be a family of projective stacks with a relative polarization $(\mathcal{E},\mathcal{O}_{X}(1))$. Let $P$ be a  polynomial of degree $d$ and $\mu_{0}$ a real number. Every set-theoretic family $\mathcal{F}_{i}$, $i\in I$ $($$I$ is a set$)$ of purely $d$-dimensional sheaves on the fiber of $p$ with fixed modified Hilbert polynomial $P$ such that $\hat{\mu}_{\mathrm{max}}(F_{\mathcal{E}}(\mathcal{F}_{i}))\leq\mu_{0}$ is bounded.
\end{theorem}
The following bound of the number of global sections will be used in Section 3.
\begin{lemma} ([\cite{Nir1}, Proposition 4.24 and Corollary 4.30])\label{glo-sec-bound}
Let $\mathcal{X}$ be a projective Deligne-Mumford stack over $k$ with a moduli scheme $\pi:\mathcal{X}\to X$ and a polarization $(\mathcal{E},\mathcal{O}_{X}(1))$.  For any  $\hat{\mu}_{\mathcal{E}}$-semistable sheaf $\mathcal{F}$ of dimension $d$ on $\mathcal{X}$, we have
\[
h^0(\mathcal{X},\mathcal{F}\otimes\mathcal{E}^\vee)\leq
\left\{
\begin{aligned}
&r\binom{\hat{\mu}_{\mathcal{E}}(\mathcal{F})+\widetilde{m}\deg(\mathcal{O}_{X}(1))+r^2+f(r)+\frac{d-1}{2}}{d}, \;\;\mathrm{if}\;\; \hat{\mu}_{\mathrm{max}}(F_{\mathcal{E}}(\mathcal{F}))\geq \frac{d+1}{2}-r^2\\
& 0,  \qquad  \qquad \qquad\qquad\qquad\qquad\qquad\qquad\qquad\qquad \mathrm{otherwise}
\end{aligned}
\right.\]
where $\hat{\mu}_{\mathrm{max}}(F_{\mathcal{E}}(\mathcal{F}))\leq\hat{\mu}_{\mathcal{E}}(\mathcal{F})+\widetilde{m}\deg(\mathcal{O}_{X}(1))$, $r=r(F_{\mathcal{E}}(\mathcal{F}))$, $f(r)=-1+\sum_{i=1}^r\frac{1}{i}$, and $\widetilde{m}$ is the integer making that $\pi_{*}\mathcal{E}nd_{\mathcal{O}_{\mathcal{X}}}(\mathcal{E})(\widetilde{m})$ is generated by global sections.
\end{lemma}
Now, as in [\cite{Wan,Lin18}], we introduce the notion of a flat family of pairs.
\begin{definition}\label{fam-pair}
A flat family $(\mathcal{F},\varphi)$ of pairs parametrized by a scheme $S$ consists of a coherent sheaf $\mathcal{F}$ on $\mathcal{X}\times S$ which is flat over $S$ and a morphism $\varphi: \pi_{\mathcal{X}}^*\mathcal{F}_{0}\to\mathcal{F}$, where $\pi_{\mathcal{X}}:\mathcal{X}\times S\to \mathcal{X}$ is the natural projection. Two families $(\mathcal{F},\varphi)$ and $(\mathcal{G},\psi)$ are isomorphic if there is an isomorphism $\Phi:\mathcal{F}\to\mathcal{G}$ such that $\Phi\circ\varphi=\psi$.
\end{definition}
\begin{remark}\label{comp-fam}
In Definition \ref{fam-pair}, the implicit $\lambda$-scaling of $\pi_{\mathcal{X}}^*\mathcal{F}_{0}$ as in Definition \ref{def-pair} for the  isomorphism of two families has been absorbed in  $\Phi$ as in Lemma \ref{iso-pair}.
Compared with the notion of a flat family  of framed sheaves defined in [\cite{BS}, Definition 3.16],  one may alternatively define a flat family of pairs as follows. A flat family $(\mathcal{F},L_{\mathcal{F}},\phi_{\mathcal{F}})$ of pairs parameterized by a scheme $S$ consists of a coherent sheaf $\mathcal{F}$ on $\mathcal{X}\times S$  which is  flat over $S$, a line bundle $L_{\mathcal{F}}$ on $S$, and a morphism $\phi_{\mathcal{F}}: L_{\mathcal{F}}\to\pi_{S*}\mathcal{H}om(\pi_{\mathcal{X}}^*\mathcal{F}_{0},\mathcal{F})$, where $\pi_{S}:\mathcal{X}\times S\to S$ is the projection.
Two families $(\mathcal{F},L_{\mathcal{F}},\phi_{\mathcal{F}})$ and $(\mathcal{G},L_{\mathcal{G}},\psi_{\mathcal{G}})$ are isomorphic if there are isomorphisms $\Phi:\mathcal{F}\to\mathcal{G}$ and $\Psi: L_{\mathcal{F}}\to L_{\mathcal{G}}$ such that
\ben
\psi_{\mathcal{G}}\circ\Psi=\widehat{\Phi}\circ\phi_{\mathcal{F}}
\een
where 
\ben	\widehat{\Phi}:\pi_{S*}\mathcal{H}om(\pi_{\mathcal{X}}^*\mathcal{F}_{0},\mathcal{F})\to\pi_{S*}\mathcal{H}om(\pi_{\mathcal{X}}^*\mathcal{F}_{0},\mathcal{G})
\een	
is the isomorphism induced by $\Phi$.  For this definition, one can impose the $\lambda$-scaling on $\pi_{\mathcal{X}}^*\mathcal{F}_{0}$ for two isomorpic families which is also absorbed in $\Phi$. As in [\cite{BS}, Remark 3.17], the morphism $\phi_{\mathcal{F}}: L_{\mathcal{F}}\to\pi_{S*}\mathcal{H}om(\pi_{\mathcal{X}}^*\mathcal{F}_{0},\mathcal{F})$ may be taken as a nowhere vanishing morphism, i.e., the composition as follows
\ben
\pi_{S}^*L_{\mathcal{F}}\otimes \pi_{\mathcal{X}}^*\mathcal{F}_{0}\to\pi_{S}^*\pi_{S*}\mathcal{H}om(\pi_{\mathcal{X}}^*\mathcal{F}_{0},\mathcal{F})\otimes \pi_{\mathcal{X}}^*\mathcal{F}_{0}\to\mathcal{H}om(\pi_{\mathcal{X}}^*\mathcal{F}_{0},\mathcal{F})\otimes \pi_{\mathcal{X}}^*\mathcal{F}_{0}\to\mathcal{F}.
\een
When $L_{\mathcal{F}}$ is trivial, the morphism $\phi_{\mathcal{F}}$ may be viewed as $\varphi$  in Definition \ref{fam-pair}. 
However, we will  adopt the notion of a flat family of pairs in Definition \ref{fam-pair} (see  Remark \ref{no-asump} for the reason).
\end{remark}

\section{Boundedness of the family of semistable pairs}
One important step in constructing moduli spaces of semistable pairs is to prove the  boundedness of the family of semistable pairs. Given a polynomial $P\in\mathbb{Q}[m]$ and a stability parameter $\delta\in\mathbb{Q}[m]$ which is zero or a polynomial with positive leading coefficient, we will prove  in this section that the family of $\delta$-semistable pairs of type $P$ is bounded. In fact,  we will generalize the boundedness results on smooth projective varieties in [\cite{Wan}, Section 3] when $\mathrm{deg}\,\delta<\mathrm{deg}\,P$ and those in [\cite{Lin18}, Section 3] when $\mathrm{deg}\,\delta\geq\mathrm{deg}\,P$ to the case of projective Deligne-Mumford stacks. Let $\mathcal{X}$ be a projective Deligne-Mumford stack over $k$ with a moduli scheme $\pi:\mathcal{X}\to X$ and a polarization $(\mathcal{E},\mathcal{O}_{X}(1))$. We  start with the case when $\mathrm{deg}\,\delta<\mathrm{deg}\,P$.

\begin{lemma}\label{case1-bound}
When $\mathrm{deg}\,\delta<\mathrm{deg}\,P$. Suppose that $(\mathcal{F},\varphi)$ is a nondegenerate $\delta$-semistable pair with the modified Hilbert polynomial $P_{\mathcal{E}}(\mathcal{F})=P$. Then $\hat{\mu}_{\mathrm{max}}(F_{\mathcal{E}}(\mathcal{F}))$ is bounded above by a constant depending on $P$, $\mathcal{F}_{0}$ and $X$.
\end{lemma}
\begin{proof}
We combine the  arguments in the proofs of [\cite{Wan}, Proposition 2.1] and  [\cite{Nir1}, Proposition 4.24].
Suppose that $\mathcal{F}$ is pure of dimension $d$, and let $\frac{\delta_{1}}{(d-1)!}$ be the coefficient of $\delta$ in degree $d-1$.	The assumption of $\delta$ implies that $\delta_{1}\geq0$. Let $(\mathcal{F}^\prime,\varphi^{\prime})$ be a subpair of $(\mathcal{F},\varphi)$ satisfying $\mathrm{im}\varphi\subset\mathcal{F}^\prime$. By  assumption we have
\ben
\hat{\mu}_{\mathcal{E}}(\mathcal{F}^{\prime})+\frac{\delta_{1}}{r(F_{\mathcal{E}}(\mathcal{F}^\prime))}\leq\hat{\mu}_{\mathcal{E}}(\mathcal{F})+\frac{\delta_{1}}{r(F_{\mathcal{E}}(\mathcal{F}))}
\een
Then $\hat{\mu}_{\mathcal{E}}(\mathcal{F}^{\prime})\leq\hat{\mu}_{\mathcal{E}}(\mathcal{F})$ since the exactness of the functor $F_{\mathcal{E}}$ implies $r(F_{\mathcal{E}}(\mathcal{F}^\prime))\leq r(F_{\mathcal{E}}(\mathcal{F}))$.
Now, let $\mathcal{F}^\prime\subseteq\mathcal{F}$ be any subsheaf. Set  $\mathcal{H}=\mathcal{F}^\prime+\mathrm{im}\varphi$ and $\mathcal{G}=\mathrm{im}\varphi/(\mathcal{F}\cap\mathrm{im}\varphi)$. Then $\mathcal{G}$ is a quotient of $\mathcal{F}_{0}$ and we have
a short exact sequence
\ben
0\to\mathcal{F}^\prime \to\mathcal{H}\to\mathcal{G}\to0.
\een
Note that $\mathcal{H}\subseteq\mathcal{F}$ and $\mathrm{im}\varphi\subseteq\mathcal{H}$, then $(\mathcal{H},i\circ\varphi)$ is a subpair of $(\mathcal{F},\varphi)$, where $i$ denotes the inclusion $\mathcal{H}\hookrightarrow\mathcal{F}$. As above, we have
$\hat{\mu}_{\mathcal{E}}(\mathcal{H})\leq\hat{\mu}_{\mathcal{E}}(\mathcal{F})$. Since $F_{\mathcal{E}}(\mathcal{G})$ is a quotient of $F_{\mathcal{E}}(\mathcal{F}_{0})$, we have $\hat{\mu}_{\mathrm{min}}(F_{\mathcal{E}}(\mathcal{F}_{0}))\leq\hat{\mu}_{\mathcal{E}}(\mathcal{G})$. Since the modified Hilbert polynomial $P_{\mathcal{E}}(\cdot)$ is additive in a short exact sequence, we have
\ben
\hat{\mu}_{\mathcal{E}}(\mathcal{F}^\prime)&=&\frac{\hat{\mu}_{\mathcal{E}}(\mathcal{H})r(F_{\mathcal{E}}(\mathcal{H}))-\hat{\mu}_{\mathcal{E}}(\mathcal{G})r(F_{\mathcal{E}}(\mathcal{G}))}{r(F_{\mathcal{E}}(\mathcal{F}^\prime))}\\
&\leq&\frac{\hat{\mu}_{\mathcal{E}}(\mathcal{F})r(F_{\mathcal{E}}(\mathcal{H}))-\hat{\mu}_{\mathrm{min}}(F_{\mathcal{E}}(\mathcal{F}_{0}))r(F_{\mathcal{E}}(\mathcal{G}))}{r(F_{\mathcal{E}}(\mathcal{F}^\prime))}\\
&=&\hat{\mu}_{\mathcal{E}}(\mathcal{F})+(\hat{\mu}_{\mathcal{E}}(\mathcal{F})-\hat{\mu}_{\mathrm{min}}(F_{\mathcal{E}}(\mathcal{F}_{0})))\frac{r(F_{\mathcal{E}}(\mathcal{G}))}{r(F_{\mathcal{E}}(\mathcal{F}^\prime))}.
\een
where $\mathcal{F}^\prime$ and $\mathcal{H}$ are of dimension $d$, and $r(F_{\mathcal{E}}(\mathcal{G}))=\alpha_{\mathcal{E},d}(\mathcal{G})\geq0$ (it is zero if $\mathrm{dim}(F_{\mathcal{E}}(\mathcal{G}))<d$).
Set 
\ben
\widetilde{C}:=\max\{\hat{\mu}_{\mathcal{E}}(\mathcal{F}),\hat{\mu}_{\mathcal{E}}(\mathcal{F})+(\hat{\mu}_{\mathcal{E}}(\mathcal{F})-\hat{\mu}_{\mathrm{min}}(F_{\mathcal{E}}(\mathcal{F}_{0})))\cdot r(F_{\mathcal{E}}(\mathcal{F}))\}
\een
which is a constant depending on $P$ and $\mathcal{F}_{0}$. Then 
$\hat{\mu}_{\mathcal{E}}(\mathcal{F}^\prime)\leq \widetilde{C}$ for any subsheaf $\mathcal{F}^\prime\subseteq\mathcal{F}$.

By Serre's vanishing theorem, one can choose an integer $\widetilde{m}$ large enough such that $\pi_{*}\mathcal{E}nd_{\mathcal{O}_{\mathcal{X}}}(\mathcal{E})(\widetilde{m})$
is generated by global sections. Set $N=h^0(X,\pi_{*}\mathcal{E}nd_{\mathcal{O}_{\mathcal{X}}}(\mathcal{E})(\widetilde{m}))$. Let $\overline{F}$ be the maximal destabilizing sheaf of $F_{\mathcal{E}}(\mathcal{F})$ with respect to ordinary $\hat{\mu}$-stability. Then $\hat{\mu}(\overline{F})=\hat{\mu}_{\mathrm{max}}(F_{\mathcal{E}}(\mathcal{F}))$.
As in the argument of [\cite{Nir1}, Proposition 4.24], one has a surjection $\overline{F}\otimes\mathcal{O}_{X}(-\widetilde{m})^{\oplus N}\to F_{\mathcal{E}}(\overline{\mathcal{F}})$, where $\overline{\mathcal{F}}$ is a subsheaf of $\mathcal{F}$  associated to $\overline{F}$ by some transformation. Since $\overline{F}\otimes\mathcal{O}_{X}(-\widetilde{m})^{\oplus N}$ is also $\hat{\mu}$-semistable, then 
\ben
\hat{\mu}(\overline{F}(-\widetilde{m}))=\hat{\mu}(\overline{F}\otimes\mathcal{O}_{X}(-\widetilde{m})^{\oplus N})\leq\hat{\mu}(F_{\mathcal{E}}(\overline{\mathcal{F}}))=\hat{\mu}_{\mathcal{E}}(\overline{\mathcal{F}})\leq\widetilde{C}.
\een
The above inequality implies that $\hat{\mu}_{\mathrm{max}}(F_{\mathcal{E}}(\mathcal{F}))\leq C$, where $C:=\widetilde{C}+\widetilde{m} \deg(\mathcal{O}_{X}(1))$  is a constant depending on $P$, $\mathcal{F}_{0}$ and $X$.
\end{proof}
The remaining case is 
\begin{lemma}\label{case2-bound}
When $\mathrm{deg}\,\delta\geq\mathrm{deg}\,P$. Suppose that $(\mathcal{F},\varphi)$ is a  nondegenerate $\delta$-semistable pair with the modified Hilbert polynomial $P_{\mathcal{E}}(\mathcal{F})=P$. Then $\hat{\mu}_{\mathrm{min}}(F_{\mathcal{E}}(\mathcal{F}))$ is bounded below by a constant depending on $P$, $\mathcal{F}_{0}$ and $X$. 
\end{lemma}
\begin{proof}
Since the functor $F_{\mathcal{E}}$ is exact, from the short exact sequence
\ben
0\to\mathrm{im}\varphi\to \mathcal{F}\to\mathrm{coker}\varphi\to0
\een
we have $F_{\mathcal{E}}(\mathrm{coker}\varphi)\cong F_{\mathcal{E}}(\mathcal{F})/F_{\mathcal{E}}(\mathrm{im}\varphi)$. By Lemma \ref{no-str-semi} and Lemma \ref{orbi-PT},  $\mathrm{dim}\, \mathrm{coker}\varphi<\mathrm{deg}\,P=\mathrm{dim}\mathcal{F}$. Then
$\mathrm{dim}F_{\mathcal{E}}(\mathrm{coker}\varphi)<\mathrm{deg}P=\mathrm{dim}F_{\mathcal{E}}(\mathcal{F})$ by Lemma \ref{puretopure} and Lemma \ref{tor-fil}. The exactness of  $F_{\mathcal{E}}$ implies $F_{\mathcal{E}}(\varphi):F_{\mathcal{E}}(\mathcal{F}_{0})\twoheadrightarrow F_{\mathcal{E}}(\mathrm{im}\varphi)\hookrightarrow F_{\mathcal{E}}(\mathcal{F})$ and $\mathrm{im}\,F_{\mathcal{E}}(\varphi)=F_{\mathcal{E}}(\mathrm{im}\varphi)$.
The proof is completed by applying the argument about $(E_{0}, E)$ in the proof of [\cite{Lin18}, Lemma 3.1] to $(F_{\mathcal{E}}(\mathcal{F}_{0}), F_{\mathcal{E}}(\mathcal{F}))$.
\end{proof}

As the modified Hilbert polynomial is additive in a short exact sequence and $P_{\mathcal{E}}(\mathcal{F})=P$ is fixed, bounding $\hat{\mu}_{\mathrm{min}}$ from below is equivalent to bounding $\hat{\mu}_{\mathrm{max}}$ from above. Hence
the constant $C$ can be chosen to be independent of $\delta$ such that $\hat{\mu}_{\mathrm{max}}$ is bounded above by $C$ by Lemma \ref{case1-bound} and  Lemma \ref{case2-bound}. Using these two lemmas and  Theorem \ref{fam-bound1} with $S=\mathrm{Spec}\,k$,  we have 
\begin{proposition}\label{bound1}
	Fix a modified Hilbert polynomial $P$ and some $\delta$. The set-theoretic family
	\ben
	\{\mathcal{F}\arrowvert (\mathcal{F},\varphi)\mbox{ is a nondegenerate $\delta$-semistable pair of type $P$}\}
	\een
	of coherent sheaves on $\mathcal{X}$ is bounded.
\end{proposition}

Next,  in order to apply GIT machinery, one may relate the semistability condition to the number of global sections of subsheaves as in [\cite{HL3}, Theorem 4.4.1]. We need the following estimate for the number of global sections.
\begin{lemma}\label{lan-esti}
Let $\mathcal{X}$ be a projective Deligne-Mumford stack over $k$ with  a polarization $(\mathcal{E},\mathcal{O}_{X}(1))$.
	Let $\mathcal{F}$ be a pure coherent sheaf of dimension $d$ on $\mathcal{X}$, then
	\ben
	\frac{h^0(X,F_{\mathcal{E}}(\mathcal{F})(m))}{r}&=&\frac{h^0(\mathcal{X},\mathcal{F}\otimes\pi^*\mathcal{O}_{X}(m)\otimes\mathcal{E}^\vee)}{r}\\
	&\leq&\frac{r-1}{r}\left[\binom{\hat{\mu}_{\mathrm{max}}(F_{\mathcal{E}}(\mathcal{F}))+m+C}{d}\right]_{+}+\frac{1}{r}\left[\binom{\hat{\mu}(F_{\mathcal{E}}(\mathcal{F}))+m+C}{d}\right]_{+}
	\een
	where $r=r(F_{\mathcal{E}}(\mathcal{F}))$ and $C:=\widetilde{m}\deg(\mathcal{O}_{X}(1))+r^2+f(r)+\frac{d-1}{2}$. Here,  $[x]_{+}:=\max\{0,x\}$ for any $x\in\mathbb{R}$, $f(r)=-1+\sum_{i=1}^r\frac{1}{i}$ and the integer
	$\widetilde{m}$ is the same as in Lemma \ref{glo-sec-bound}.
\end{lemma}
\begin{proof}
	For a pure coherent sheaf $\mathcal{F}$ of dimension $d$, we have the following Harder-Narasimhan filtration with repect to $\hat{\mu}_{\mathcal{E}}$-stability
	\ben
	0=\mathcal{F}_{0}\subsetneqq\mathcal{F}_{1}\subsetneqq\mathcal{F}_{2}\subsetneqq\cdots\subsetneqq\mathcal{F}_{l}=\mathcal{F}
	\een
	such that the factors $\mathcal{F}_{i}/\mathcal{F}_{i-1}$ for $i=1,\cdots,l$ are $\hat{\mu}_{\mathcal{E}}$-semistable of dimension $d$ and 
	\be\label{HN-inequ}
	\hat{\mu}(F_{\mathcal{E}}(\mathcal{F}_{1}))>\hat{\mu}(F_{\mathcal{E}}(\mathcal{F}_{2}/\mathcal{F}_{1}))>\cdots>\hat{\mu}(F_{\mathcal{E}}(\mathcal{F}_{l}/\mathcal{F}_{l-1})).
	\ee
	For any $i=1,\cdots,l$, we
	have a short exact sequence
	\ben
	0\to\mathcal{F}_{i-1}\to\mathcal{F}_{i}\to\mathcal{F}_{i}/\mathcal{F}_{i-1}\to0.
	\een
	Tensor with $\pi^*\mathcal{O}_{X}(m)$ and apply the exact functor $F_{\mathcal{E}}$, we obtain the following exact sequence
	\ben
	0\to F_{\mathcal{E}}(\mathcal{F}_{i-1})(m)\to F_{\mathcal{E}}(\mathcal{F}_{i})(m)\to  F_{\mathcal{E}}(\mathcal{F}_{i}/\mathcal{F}_{i-1})(m)\to0.
	\een
	Then for any $i=1,\cdots,l$, we have 
	\ben
	&&r(F_{\mathcal{E}}(\mathcal{F}_{i}/\mathcal{F}_{i-1})(m))=r(F_{\mathcal{E}}(\mathcal{F}_{i})(m))-r(F_{\mathcal{E}}(\mathcal{F}_{i-1})(m)),\\
	&&h^0(X,F_{\mathcal{E}}(\mathcal{F}_{i})(m))-h^0(X,F_{\mathcal{E}}(\mathcal{F}_{i-1})(m))\leq h^0(X,F_{\mathcal{E}}(\mathcal{F}_{i}/\mathcal{F}_{i-1})(m)).
	\een
	This implies that
	\ben
	&&r(F_{\mathcal{E}}(\mathcal{F})(m))=\sum_{i=1}^{l}r(F_{\mathcal{E}}(\mathcal{F}_{i}/\mathcal{F}_{i-1})(m)),\\
	&&h^0(X,F_{\mathcal{E}}(\mathcal{F})(m))\leq \sum_{i=1}^{l}h^0(X,F_{\mathcal{E}}(\mathcal{F}_{i}/\mathcal{F}_{i-1})(m)).
	\een
	By a simple computation, we have for any pure coherent sheaf $\mathcal{G}$,
	\ben
	r(F_{\mathcal{E}}(\mathcal{G})(m))=r(F_{\mathcal{E}}(\mathcal{G}));\;\;\; \hat{\mu}_{\mathcal{E}}(\mathcal{G}\otimes\pi^*\mathcal{O}_{X}(m))=\hat{\mu}(F_{\mathcal{E}}(\mathcal{G})(m))=\hat{\mu}(F_{\mathcal{E}}(\mathcal{G}))+m.
	\een
	Let
	$r_{i}=:r(F_{\mathcal{E}}(\mathcal{F}_{i}/\mathcal{F}_{i-1}))$ for $i=1,\cdots,l$.
	Since for each $1\leq i\leq l$, the sheaf $\mathcal{F}_{i}/\mathcal{F}_{i-1}$ is $\hat{\mu}_{\mathcal{E}}$-semistable of dimension $d$, $(\mathcal{F}_{i}/\mathcal{F}_{i-1})\otimes\pi^*\mathcal{O}_{X}(m)$ is also $\hat{\mu}_{\mathcal{E}}$-semistable of dimension $d$. By Lemma \ref{glo-sec-bound}, we have for each $i=1,\cdots,l$,
	\ben
	\frac{h^0(X,F_{\mathcal{E}}(\mathcal{F}_{i}/\mathcal{F}_{i-1})(m))}{r_{i}}&=&\frac{h^0(\mathcal{X},(\mathcal{F}_{i}/\mathcal{F}_{i-1})\otimes\pi^*\mathcal{O}_{X}(m)\otimes\mathcal{E}^\vee)}{r_{i}}\\
	&\leq&\left[\binom{\hat{\mu}(F_{\mathcal{E}}(\mathcal{F}_{i}/\mathcal{F}_{i-1}))+m+\widetilde{m}\deg(\mathcal{O}_{X}(1))+r_{i}^2+f(r_{i})+\frac{d-1}{2}}{d}\right]_{+}
	\een
	Since $F_{\mathcal{E}}(\mathcal{F}_{1})\hookrightarrow F_{\mathcal{E}}(\mathcal{F})$, we have $\hat{\mu}(F_{\mathcal{E}}(\mathcal{F}_{1}))\leq\hat{\mu}_{\mathrm{max}}(F_{\mathcal{E}}(\mathcal{F}_{1}))\leq\hat{\mu}_{\mathrm{max}}(F_{\mathcal{E}}(\mathcal{F}))$. Combining with \eqref{HN-inequ}, we have 
	$\hat{\mu}(F_{\mathcal{E}}(\mathcal{F}_{i}/\mathcal{F}_{i-1}))\leq\hat{\mu}_{\mathrm{max}}(F_{\mathcal{E}}(\mathcal{F}))$ for $i=1,\cdots,l-1$ and $\hat{\mu}(F_{\mathcal{E}}(\mathcal{F}_{l}/\mathcal{F}_{l-1}))\leq\hat{\mu}(F_{\mathcal{E}}(\mathcal{F}))$.
	Then 
	\ben
	\frac{h^0(X,F_{\mathcal{E}}(\mathcal{F})(m))}{r}&\leq&\sum_{i=1}^{l}\frac{r_{i}}{r}\frac{h^0(X,F_{\mathcal{E}}(\mathcal{F}_{i}/\mathcal{F}_{i-1})(m))}{r_{i}}\\
	&\leq&\frac{r-1}{r}\left[\binom{\hat{\mu}_{\mathrm{max}}(F_{\mathcal{E}}(\mathcal{F}))+m+C}{d}\right]_{+}+\frac{1}{r}\left[\binom{\hat{\mu}(F_{\mathcal{E}}(\mathcal{F}))+m+C}{d}\right]_{+}
	\een
	where $C:=\widetilde{m}\deg(\mathcal{O}_{X}(1))+r^2+f(r)+\frac{d-1}{2}$.
	
\end{proof}	
Now, we begin with the first case when $\mathrm{deg}\,\delta<\mathrm{deg}\,P$.
\begin{lemma}\label{bound2}
Let $\mathcal{X}$ be a projective Deligne-Mumford stack over $k$ with a polarization $(\mathcal{E},\mathcal{O}_{X}(1))$.	
Assume that $\mathrm{deg}\,\delta<\mathrm{deg}\,P$. Then there is an integer $m_{0}>0$, such that for any integer $m\geq m_{0}$ and any nondegenerate pair $(\mathcal{F},\varphi)$ satisfying that $\mathcal{F}$ is a  pure coherent sheaf of dimension $d$ on  $\mathcal{X}$ with
$P_{\mathcal{E}}(\mathcal{F})=P$ and $r=r(F_{\mathcal{E}}(\mathcal{F}))$, the following properties are equivalent.\\
$(i)$ The pair $(\mathcal{F},\varphi)$ is $\delta$-(semi)stable.\\
$(ii)$  $P(m)\leq h^0(F_{\mathcal{E}}(\mathcal{F})(m))$ and for any  subpair $(\mathcal{F}^\prime,\varphi^\prime)$ with  $r(F_{\mathcal{E}}(\mathcal{F}^\prime))=r^\prime$ satisfying $0<r^\prime<r$,
\ben
h^0(F_{\mathcal{E}}(\mathcal{F}^\prime)(m))+\epsilon(\varphi^\prime)\delta(m) (\leq)\frac{r^\prime}{r}(P(m)+\epsilon(\varphi)\delta(m)).
\een
$(iii)$ For any quotient pair $(\mathcal{G},\varphi^{\prime\prime})$ with  $r(F_{\mathcal{E}}(\mathcal{G}))=r^{\prime\prime}$ satisfying $0<r^{\prime\prime}<r$,
\ben
\frac{r^{\prime\prime}}{r}(P(m)+\epsilon(\varphi)\delta(m)) (\leq) h^0(F_{\mathcal{E}}(\mathcal{G})(m))+\epsilon(\varphi^{\prime\prime})\delta(m).
\een
\end{lemma}
\begin{proof}
We will deal with the  $\delta$-stable case, the $\delta$-semistable case can be proved similarly.\\
$(i)\Rightarrow(ii)$: By Proposition \ref{bound1}, the set-theoretic family of coherent sheaves on $\mathcal{X}$ underlying $\delta$-semistable pairs with  fixed modified Hilbert polynomial $P$ is bounded. By Kleiman criterion for stacks,  there is an integer $m_{0}>0$ such that  for any $\mathcal{F}$ underlying a $\delta$-stable pair $(\mathcal{F},\varphi)$, we have $H^i(F_{\mathcal{E}}(\mathcal{F})(m))=0$ for all $i>0$, and hence $P(m)=h^0(F_{\mathcal{E}}(\mathcal{F})(m))$. Since the pair 
$(\mathcal{F},\varphi)$ is $\delta$-stable, by Lemma \ref{case1-bound}, there is a constant $\mu_{0}$ depending on  $P$, $\mathcal{F}_{0}$ and $X$ such that $\hat{\mu}_{\mathrm{max}}(F_{\mathcal{E}}(\mathcal{F}))\leq \mu_{0}$.
As $(\mathcal{F}^\prime,\varphi^\prime)$ is a subpair of $(\mathcal{F},\varphi)$, we have
$\hat{\mu}_{\mathrm{max}}(F_{\mathcal{E}}(\mathcal{F}^\prime))\leq\hat{\mu}_{\mathrm{max}}(F_{\mathcal{E}}(\mathcal{F}))\leq \mu_{0}$. By Lemma \ref{lan-esti} and the inequality $0<r^\prime<r$,  there is a constant $C$ depending on $r$ and $d$ such that
\be\label{glo-sec-bou1}
\frac{h^0(F_{\mathcal{E}}(\mathcal{F}^\prime)(m))}{r^\prime}
\leq\frac{r-1}{r}\left[\binom{\mu_{0}+m+C}{d}\right]_{+}+\frac{1}{r}\left[\binom{\hat{\mu}(F_{\mathcal{E}}(\mathcal{F}^\prime))+m+C}{d}\right]_{+}.
\ee
We distinguish two cases:\\
(a) $\hat{\mu}(F_{\mathcal{E}}(\mathcal{F}^\prime))\geq r\cdot\hat{\mu}(F_{\mathcal{E}}(\mathcal{F}))-(r-1)\cdot\mu_{0}-r\cdot(C-\frac{d-3}{2}+\delta_{1})$,\\
(b) $\hat{\mu}(F_{\mathcal{E}}(\mathcal{F}^\prime))< r\cdot\hat{\mu}(F_{\mathcal{E}}(\mathcal{F}))-(r-1)\cdot\mu_{0}-r\cdot(C-\frac{d-3}{2}+\delta_{1})$,\\
where $\frac{\delta_{1}}{(d-1)!}$ is the coefficient of $\delta$ in degree $d-1$. Then $\delta_{1}\geq0$. In order to show $(ii)$, we assume that $\mathcal{F}^\prime$ is saturated in $\mathcal{F}$ since by Remark \ref{saturation} the saturation $\overline{\mathcal{F}^\prime}\supseteq\mathcal{F}^\prime$ implies that 
$r(F_{\mathcal{E}}(\mathcal{F}^\prime))=r(F_{\mathcal{E}}(\overline{\mathcal{F}^\prime}))$ and $h^0(F_{\mathcal{E}}(\mathcal{F}^\prime)(m))\leq h^0(F_{\mathcal{E}}(\overline{\mathcal{F}^\prime})(m))$. The set-theoretic family
of purely saturated subsheaves $\mathcal{F}^\prime$ of type (a) is bounded by Grothendieck's Lemma \ref{Gro-lem}. By Kleiman criterion for stacks, the set of modified Hilbert polynomials are finite and enlarging $m_{0}$ if necessary, for $m\geq m_{0}$, we have $h^0(F_{\mathcal{E}}(\mathcal{F}^\prime)(m))=P_{\mathcal{E}}(\mathcal{F}^\prime)(m)$ and by $\delta$-stability of $(\mathcal{F},\varphi)$,
\ben
P_{\mathcal{E}}(\mathcal{F}^\prime)+\epsilon(\varphi^\prime)\delta <\frac{r^\prime}{r}(P+\epsilon(\varphi)\delta)\iff P_{\mathcal{E}}(\mathcal{F}^\prime)(m)+\epsilon(\varphi^\prime)\delta(m) <\frac{r^\prime}{r}(P(m)+\epsilon(\varphi)\delta(m)).
\een
For the subsheaves $\mathcal{F}^\prime$ of type $(b)$, enlarging $m_{0}$ if necessary, it follows from \eqref{glo-sec-bou1}
\ben
\frac{h^0(F_{\mathcal{E}}(\mathcal{F}^\prime)(m))}{r^\prime}
&\leq&\frac{r-1}{r}\left[\binom{\mu_{0}+m+C}{d}\right]_{+}\\
&&+\frac{1}{r}\left[\binom{r\cdot\hat{\mu}(F_{\mathcal{E}}(\mathcal{F}))-(r-1)\cdot\mu_{0}-r\cdot(C-\frac{d-3}{2}+\delta_{1})+m+C}{d}\right]_{+}\\
&=&\frac{m^d}{d!}+\frac{m^{d-1}}{(d-1)!}\bigg(\frac{r-1}{r}\left(\mu_{0}+C-\frac{d-1}{2}\right)+\frac{1}{r}\bigg(r\cdot\hat{\mu}(F_{\mathcal{E}}(\mathcal{F}))\\
&&-(r-1)\cdot\mu_{0}-r\cdot(C-\frac{d-3}{2}+\delta_{1})+C-\frac{d-1}{2}\bigg)\bigg)+\cdots\\
&=&\frac{m^d}{d!}+\frac{m^{d-1}}{(d-1)!}\left(\hat{\mu}(F_{\mathcal{E}}(\mathcal{F}))-\delta_{1}-1\right)+\cdots
\een
where $\cdots$ denotes for some polynomial in $m$ of degree $\leq d-2$ with coefficients independent of $\mathcal{F}^{\prime}$.
Since 
\ben
\hat{\mu}(F_{\mathcal{E}}(\mathcal{F}))-\delta_{1}-1+\frac{\epsilon(\varphi^\prime)}{r^\prime}\delta_{1}<\hat{\mu}(F_{\mathcal{E}}(\mathcal{F}))+\frac{\epsilon(\varphi)}{r}\delta_{1},
\een
by enlarging $m_{0}$ if necessary, we have
\ben
\frac{h^0(F_{\mathcal{E}}(\mathcal{F}^\prime)(m))+\epsilon(\varphi^\prime)\delta(m)}{r^\prime}<\frac{P(m)+\epsilon(\varphi)\delta(m)}{r}
\een
for any $m\geq m_{0}$.\\
$(ii)\Rightarrow(iii)$: For any quotient pair $(\mathcal{G},\varphi^{\prime\prime})$ with multiplicity  $r(F_{\mathcal{E}}(\mathcal{G}))=r^{\prime\prime}$ satisfying $0<r^{\prime\prime}<r$, we have a short exact sequence 
\ben
0\to(\mathcal{F}^\prime,\varphi^\prime)\to(\mathcal{F},\varphi)\to(\mathcal{G},\varphi^{\prime\prime})\to0.
\een
Since the functor $F_{\mathcal{E}}$ is exact, we have 
$h^0(F_{\mathcal{E}}(\mathcal{G})(m))\geq h^0(F_{\mathcal{E}}(\mathcal{F})(m))-h^0(F_{\mathcal{E}}(\mathcal{F}^\prime)(m))$.
Set $r^\prime=r(F_{\mathcal{E}}(\mathcal{F}^\prime))$, we have $r=r^\prime+r^{\prime\prime}$. Also, we have $\epsilon(\varphi)=\epsilon(\varphi^\prime)+\epsilon(\varphi^{\prime\prime})$. By $(ii)$, we have
\ben
h^0(F_{\mathcal{E}}(\mathcal{G})(m))+\epsilon(\varphi^{\prime\prime})\delta(m)
&\geq& h^0(F_{\mathcal{E}}(\mathcal{F})(m))+\epsilon(\varphi)\delta(m)-h^0(F_{\mathcal{E}}(\mathcal{F}^\prime)(m))-\epsilon(\varphi^\prime)\delta(m)\\
&>&P(m)+\epsilon(\varphi)\delta(m)-\frac{r^\prime}{r}(P(m)+\epsilon(\varphi)\delta(m))\\&=&\frac{r^{\prime\prime}}{r}(P(m)+\epsilon(\varphi)\delta(m)).
\een
$(iii)\Rightarrow(i)$: Fix a pair $(\mathcal{F},\varphi)$ with  modified Hilbert polynomial $P_{\mathcal{E}}(\mathcal{F})=P$. Let $(\mathcal{G},\varphi^{\prime\prime})$  be any purely quotient pair of $(\mathcal{F},\varphi)$ with  $0<r(F_{\mathcal{E}}(\mathcal{G}))=r^{\prime\prime}<r$, we distinguish all  sheaves $\mathcal{G}$ into two cases:\\
$(\tilde{a})$ $\hat{\mu}(F_{\mathcal{E}}(\mathcal{G}))>\hat{\mu}(F_{\mathcal{E}}(\mathcal{F}))+\frac{\delta_{1}}{r}$,\\
$(\tilde{b})$ $\hat{\mu}(F_{\mathcal{E}}(\mathcal{G}))\leq\hat{\mu}(F_{\mathcal{E}}(\mathcal{F}))+\frac{\delta_{1}}{r}$,\\
where $\frac{\delta_{1}}{(d-1)!}$ is the coefficient of $\delta$ in degree $d-1$.
For the case $(\tilde{a})$, we have $p_{(\mathcal{F},\varphi)}<p_{(\mathcal{G},\varphi^{\prime\prime})}$. For the case $(\tilde{b})$,  the set-theoretic family of purely $d$-dimensional quotient $\mathcal{G}$ with $\hat{\mu}_{\mathcal{E}}(\mathcal{G})$ bounded from above  is bounded by Grothendieck's Lemma \ref{Gro-lem}. By Kleiman criterion for stacks, for large $m$ we have $h^0(F_{\mathcal{E}}(\mathcal{G})(m))=P_{\mathcal{E}}(\mathcal{G})(m)$, and by $(iii)$ we have
\ben
\frac{r^{\prime\prime}}{r}(P(m)+\epsilon(\varphi)\delta(m)) < P_{\mathcal{E}}(\mathcal{G})(m)+\epsilon(\varphi^{\prime\prime})\delta(m)\Leftrightarrow p_{(\mathcal{F},\varphi)}<p_{(\mathcal{G},\varphi^{\prime\prime})}
\een
Then the pair $(\mathcal{F},\varphi)$ is $\delta$-stable  by Remark \ref{semi-quo}.  Thus all pairs $(\mathcal{F},\varphi)$ satisfying $(iii)$ for large $m$ are $\delta$-stable pairs with fixed modified Hilbert polynomial $P_{\mathcal{E}}(\mathcal{F})=P$,  and  hence the set-theoretic family of  sheaves $\mathcal{F}$ underlying pairs $(\mathcal{F},\varphi)$ satisfying $(iii)$ for large $m$ is bounded by Proposition \ref{bound1}.
\end{proof}

The remaining case when $\mathrm{deg}\,\delta\geq\mathrm{deg}\,P$ is presented in the following
\begin{lemma}\label{bound3}
Let $\mathcal{X}$ be a projective Deligne-Mumford stack over $k$ with a polarization $(\mathcal{E},\mathcal{O}_{X}(1))$.	
Assume that $\mathrm{deg}\,\delta\geq\mathrm{deg}\,P$. Then there is an integer $m_{0}>0$, such that for any integer $m\geq m_{0}$ and any nondegenerate pair $(\mathcal{F},\varphi)$ satisfying that $\mathcal{F}$ is a  pure coherent sheaf of dimension $d$ on  $\mathcal{X}$ with
$P_{\mathcal{E}}(\mathcal{F})=P$ and $r=r(F_{\mathcal{E}}(\mathcal{F}))$, the following properties are equivalent.\\
$(i)$ The pair $(\mathcal{F},\varphi)$ is $\delta$-stable.\\
$(ii)$  $P(m)\leq h^0(F_{\mathcal{E}}(\mathcal{F})(m))$ and for any  subpair $(\mathcal{F}^\prime,\varphi^\prime)$ with  $r(F_{\mathcal{E}}(\mathcal{F}^\prime))=r^\prime$ satisfying $0<r^\prime<r$,
\ben
\frac{h^0(F_{\mathcal{E}}(\mathcal{F}^\prime)(m))}{2r^\prime-\epsilon(\varphi^\prime)} <\frac{h^0(F_{\mathcal{E}}(\mathcal{F})(m))}{2r-\epsilon(\varphi)}.
\een
$(iii)$ For any quotient pair $(\mathcal{G},\varphi^{\prime\prime})$ with  $r(F_{\mathcal{E}}(\mathcal{G}))=r^{\prime\prime}$ satisfying $0<r^{\prime\prime}<r$,
\ben
\frac{P(m)}{2r-\epsilon(\varphi)} <\frac{h^0(F_{\mathcal{E}}(\mathcal{G})(m))}{2r^{\prime\prime}-\epsilon(\varphi^{\prime\prime})}.
\een
\end{lemma}
\begin{proof} 
$(i)\Rightarrow(ii)$:  Using the same argument in the proof of Lemma \ref{bound2} with the same notation $\mu_{0}$ and $C$, we have $P(m)= h^0(F_{\mathcal{E}}(\mathcal{F})(m))$ for $m\geq m_{0}$ and $\hat{\mu}_{\mathrm{max}}(F_{\mathcal{E}}(\mathcal{F}^\prime))\leq\hat{\mu}_{\mathrm{max}}(F_{\mathcal{E}}(\mathcal{F}))\leq \mu_{0}$ by Lemma \ref{case2-bound}.
Then we  also have the inequality \eqref{glo-sec-bou1}. And all sheaves $\mathcal{F}^\prime$ underlying  subpairs $(\mathcal{F}^\prime,\varphi^\prime)$ are divided into the following two cases:\\
(a) $\hat{\mu}(F_{\mathcal{E}}(\mathcal{F}^\prime))\geq r\cdot\hat{\mu}(F_{\mathcal{E}}(\mathcal{F}))-(r-1)\cdot\mu_{0}-r\cdot(C-\frac{d-3}{2})$,\\
(b) $\hat{\mu}(F_{\mathcal{E}}(\mathcal{F}^\prime))< r\cdot\hat{\mu}(F_{\mathcal{E}}(\mathcal{F}))-(r-1)\cdot\mu_{0}-r\cdot(C-\frac{d-3}{2})$.\\
For type (a), one uses Grothendieck's Lemma \ref{Gro-lem}, Kleiman criterion for stacks and Lemma \ref{equ-sta} to derive the desired inequality. For type $(b)$, it is easy to show 
$
\frac{h^0(F_{\mathcal{E}}(\mathcal{F}^\prime)(m))}{r^\prime}<\frac{h^0(F_{\mathcal{E}}(\mathcal{F})(m))}{r}
$ and then use the inequality $\frac{\epsilon(\varphi^\prime)}{r^\prime}\leq\frac{\epsilon(\varphi)}{r}$ (see [\cite{Lin18}, Lemma 2.9 (2-1)]) to complete this part of proof.\\
$(ii)\Rightarrow(iii)$: Follow the similar arguemnt in [\cite{Lin18}, Lemma 3.5]. \\
$(iii)\Rightarrow(i)$:  Fix a nondegenerate pair  $(\mathcal{F},\varphi)$. As in [\cite{Lin18}, Lemma 3.5], let $\mathrm{gr}_{l}:=(\mathrm{gr}_{l}\mathcal{F},\mathrm{gr}_{l}\varphi)$
be the last factor of Harder-Narasimhan filtration for $(\mathcal{F},\varphi)$ by Proposition \ref{HN-pair}, by $(iii)$ we have
\ben
\frac{h^0(F_{\mathcal{E}}(\mathrm{gr}_{l}\mathcal{F})(m))}{2r(F_{\mathcal{E}}(\mathrm{gr}_{l}\mathcal{F}))-\epsilon(\mathrm{gr}_{l}\varphi)}>\frac{P(m)}{2r-1}.
\een
For large $m$, we have $h^0(F_{\mathcal{E}}(\mathrm{gr}_{l}\mathcal{F})(m))=P_{\mathcal{E}}(\mathrm{gr}_{l}\mathcal{F})(m)$ and hence
\ben
\frac{P_{\mathcal{E}}(\mathrm{gr}_{l}\mathcal{F})(m)}{2r(F_{\mathcal{E}}(\mathrm{gr}_{l}\mathcal{F}))-\epsilon(\mathrm{gr}_{l}\varphi)}>\frac{P(m)}{2r-1}\Longleftrightarrow\frac{P_{\mathcal{E}}(\mathrm{gr}_{l}\mathcal{F})}{2r(F_{\mathcal{E}}(\mathrm{gr}_{l}\mathcal{F}))-\epsilon(\mathrm{gr}_{l}\varphi)}>\frac{P}{2r-1}. 
\een
Since both polynomials $\frac{P_{\mathcal{E}}(\mathrm{gr}_{l}\mathcal{F})}{r(F_{\mathcal{E}}(\mathrm{gr}_{l}\mathcal{F}))}$ and $\frac{P}{r}$ have the same leading coefficient, we have $\frac{\epsilon(\mathrm{gr}_{l}\varphi)}{r(F_{\mathcal{E}}(\mathrm{gr}_{l}\mathcal{F}))}\geq\frac{1}{r}$. Then $\epsilon(\mathrm{gr}_{l}\varphi)=1$, which implies  $l=1$ by Remark \ref{HN-factor}. Thus $(\mathcal{F},\varphi)$ is a nondegenerate $\delta$-semistable pair, and hence it is $\delta$-stable by Lemma \ref{no-str-semi}. Thus all pairs $(\mathcal{F},\varphi)$ satisfying $(iii)$ for large $m$ are $\delta$-stable pairs with fixed modified Hilbert polynomial $P_{\mathcal{E}}(\mathcal{F})=P$,  and hence the set-theoretic family of  sheaves $\mathcal{F}$ underlying pairs $(\mathcal{F},\varphi)$ satisfying $(iii)$ for large $m$ is bounded by Proposition \ref{bound1}.
\end{proof}

\section{Construction of  moduli spaces of semistable pairs}
In this section, given a polynomial $P\in\mathbb{Q}[m]$ and a stability parameter $\delta\in\mathbb{Q}[m]$ which is zero or a polynomial with positive leading coefficient, we will construct a moduli space of $\delta$-(semi)stable pairs of type $P$ on a projective Deligne-Mumford stack $\mathcal{X}$ over $k$ with a moduli scheme $\pi:\mathcal{X}\to X$ and a polarization $(\mathcal{E},\mathcal{O}_{X}(1))$. We will only consider the case of nondegenerate $\delta$-(semi)stable pairs since otherwise these are the moduli spaces of (semi)stable
sheaves which have been constructed in [\cite{Nir1}, Section 6]. Actually, we generalize the construction of moduli spaces on smooth projective varieties in [\cite{Wan,Lin18}] to the case of projective Deligne-Mumford stacks. We also give a description of variation of  moduli spaces when the stability parameter $\delta$ changes.
\subsection{The parameter space, group actions  and  linearizations}
We  first recall some notation and results in [\cite{HL3,OS03,Nir1}]. If $p:\mathcal{X}\to S$ is a family of projective stacks with a moduli scheme $\mathcal{X}\xrightarrow{\pi}X\xrightarrow{\hat{p}}S$ and a relative polarization $(\mathcal{E},\mathcal{O}_{X}(1))$. Let $\mathcal{H}$ be a coherent sheaf on $\mathcal{X}$.
Denote by $\underline{\mathrm{Quot}}_{\mathcal{X}/S}(\mathcal{H},P)$
the functor of quotients of $\mathcal{H}$ with modified Hilbert polynomial $P$. It is shown in [\cite{Nir1}, Proposition 4.20] or [\cite{OS03}, Proposition 6.2], the natural transformation $F_{\mathcal{E}}$ which maps $\underline{\mathrm{Quot}}_{\mathcal{X}/S}(\mathcal{H},P)$ to the ordinary Quot functor $\underline{\mathrm{Quot}}_{X/S}(F_{\mathcal{E}}(\mathcal{H}),P)$ is relatively representable by schemes and is a closed immersion. And this implies that $\widehat{\mathcal{Q}}:=\mathrm{Quot}_{\mathcal{X}/S}(\mathcal{H},P)$ which represents $\underline{\mathrm{Quot}}_{\mathcal{X}/S}(\mathcal{H},P)$ is a closed immersion of the ordinary Quot scheme $\mathcal{Q}:=\mathrm{Quot}_{X/S}(F_{\mathcal{E}}(\mathcal{H}),P)$ which represents $\underline{\mathrm{Quot}}_{X/S}(F_{\mathcal{E}}(\mathcal{H}),P)$, and hence  $\widehat{\mathcal{Q}}$ is a  projective scheme.  Let $i: \widehat{\mathcal{Q}}\to \mathcal{Q}$ be the closed immersion.  Let $\widehat{\mathcal{U}}$ be the universal quotient sheaf of $\widehat{\mathcal{Q}}$, that is, the morphism $\mathcal{O}_{\widehat{\mathcal{Q}}}\otimes\mathcal{H}\to\widehat{\mathcal{U}}$ is the universal quotient parameterized by $\widehat{\mathcal{Q}}$. And let $\mathcal{U}$ be the universal quotient sheaf of $\mathcal{Q}$. As in the proof of [\cite{HL3}, Proposition 2.2.5], for sufficiently large $l$, we have a closed immersion
\ben
\mathrm{Quot}_{X/S}(F_{\mathcal{E}}(\mathcal{H}),P)\xrightarrow{\zeta_{l}}\mathrm{Grass}_{S}(\hat{p}_{*}F_{\mathcal{E}}(\mathcal{H})(l),P(l))
\een
and a Pl$\mathrm{\ddot{u}}$cker embedding of the Grassmannian
\ben
\mathrm{Grass}_{S}(\hat{p}_{*}F_{\mathcal{E}}(\mathcal{H})(l),P(l))\xrightarrow{\xi_{l}}
\widetilde{\mathbb{P}}:=\mathbb{P}(\Lambda^{P(l)}(\hat{p}_{*}F_{\mathcal{E}}(\mathcal{H})(l))).
\een
Then we have a class of very ample line bundles on $\mathcal{Q}$:
\ben
\zeta_{l}^*\xi_{l}^*\mathcal{O}_{\widetilde{\mathbb{P}}}(1)=\det(\hat{p}_{\mathcal{Q}*}(\mathcal{U}(l))).
\een
where $\hat{p}_{\mathcal{Q}}: X_{\mathcal{Q}}:=X\times_{S}\mathcal{Q}\to\mathcal{Q}$ is the natural projection. Then by [\cite{Nir1}, Proposition 6.2], a class of very ample line bundles on $\widehat{\mathcal{Q}}$ is constructed
as follows
\ben
\mathcal{L}_{l}:=\det(\hat{p}_{\widehat{\mathcal{Q}}*}(F_{\hat{p}_{\mathcal{X}}^*\mathcal{E}}(\widehat{\mathcal{U}})(l)))=i^*\zeta_{l}^*\xi_{l}^*\mathcal{O}_{\widetilde{\mathbb{P}}}(1)
\een
where $\hat{p}_{\widehat{\mathcal{Q}}}:X_{\widehat{\mathcal{Q}}}=X\times_{S}\widehat{\mathcal{Q}}\to\widehat{\mathcal{Q}}$ and $\hat{p}_{\mathcal{X}}:\mathcal{X}\times_{S}\widehat{\mathcal{Q}}\to\mathcal{X}$ are the  natural projections.
In this section, we will consider the case when $S=\mathrm{Spec}\,k$ and $\mathcal{H}:=V\otimes\mathcal{E}\otimes\pi^*\mathcal{O}_{X}(-m)$.

Using the boundedness results in Section 3,
there exists an integer $\widehat{m}>0$ such that for every integer $m\geq\widehat{m}$,
we have \\
$(a)$ $F_{\mathcal{E}}(\mathcal{F}_{0})$ is $m$-regular (by Serre vanishing theorem) and $\delta(m)>0$ unless $\delta=0$.\\
$(b)$ For any $\delta$-semistable pair
$(\mathcal{F},\varphi)$ of type $P$, the coherent sheaf $\mathcal{F}$ is $m$-regular (by Proposition \ref{bound1}).\\
$(c)$  Equivalent properties in both Lemma \ref{bound2} and Lemma \ref{bound3} hold.

Now, we fix such an integer $m$, then for any $\delta$-semistable pair
$(\mathcal{F},\varphi)$ of type $P$, the sheaf $F_{\mathcal{E}}(\mathcal{F})$  is globally generated and $h^0(F_{\mathcal{E}}(\mathcal{F})(m))=P(m)$.
Let $V=k^{\oplus P(m)}$. Then we have an isomorphism $\rho:V\to H^0(F_{\mathcal{E}}(\mathcal{F})(m))$, which  induces the following quotient $\widetilde{\rho}=\mathrm{ev}\circ(\rho\otimes \mathrm{id})$:
\ben
V\otimes \mathcal{O}_{X}(-m)\xrightarrow{\rho\otimes \mathrm{id}}H^0(F_{\mathcal{E}}(\mathcal{F})(m))\otimes \mathcal{O}_{X}(-m)\xrightarrow{\mathrm{ev}}F_{\mathcal{E}}(\mathcal{F}).
\een
Since $G_{\mathcal{E}}$ is a right exact functor and the morphism $\theta_{\mathcal{E}}(\mathcal{F})$ is surjective, we have the following quotient $q:=\theta_{\mathcal{E}}(\mathcal{F})\circ G_{\mathcal{E}}(\widetilde{\rho})$:
\ben
V\otimes\mathcal{E}\otimes\pi^*\mathcal{O}_{X}(-m)=G_{\mathcal{E}}(V\otimes \mathcal{O}_{X}(-m))\xrightarrow{G_{\mathcal{E}}(\widetilde{\rho})}G_{\mathcal{E}}\circ F_{\mathcal{E}}(\mathcal{F})\xrightarrow{\theta_{\mathcal{E}}(\mathcal{F})}\mathcal{F}.
\een
This defines a point $[q]$ in $\mathrm{Quot}_{\mathcal{X}/k}(V\otimes\mathcal{E}\otimes\pi^*\mathcal{O}_{X}(-m),P)$. On the other hand, given a point in $\mathrm{Quot}_{\mathcal{X}/k}(V\otimes\mathcal{E}\otimes\pi^*\mathcal{O}_{X}(-m),P)$ which is represented by a quotient $V\otimes\mathcal{E}\otimes\pi^*\mathcal{O}_{X}(-m)\xrightarrow{\sigma}\mathcal{F}$, we have the composition
$\widetilde{\sigma}:=F_{\mathcal{E}}(\sigma)\circ\varphi_{\mathcal{E}}(V\otimes \mathcal{O}_{X}(-m))$ as follows:
\ben
V\otimes \mathcal{O}_{X}(-m)\xrightarrow{\varphi_{\mathcal{E}}(V\otimes \mathcal{O}_{X}(-m))} F_{\mathcal{E}}\circ G_{\mathcal{E}}(V\otimes \mathcal{O}_{X}(-m))\xrightarrow{F_{\mathcal{E}}(\sigma)}F_{\mathcal{E}}(\mathcal{F})
\een
which induces a morphism in cohomology $H^0(\widetilde{\sigma}(m)):$ $V\rightarrow H^0(X,F_{\mathcal{E}}(\mathcal{F})(m))$. Define $\mathbf{Q}$ to be the set of points $[\sigma]$ of $\mathrm{Quot}_{\mathcal{X}/k}(V\otimes\mathcal{E}\otimes\pi^*\mathcal{O}_{X}(-m),P)$ such that the induced morphism $H^0(\widetilde{\sigma}(m))$ is  an isomorphism where $\widetilde{\sigma}:=F_{\mathcal{E}}(\sigma)\circ\varphi_{\mathcal{E}}(V\otimes \mathcal{O}_{X}(-m))$. The proof in [\cite{Nir1}, Theorem 5.1] shows $\mathbf{Q}$  is an open subscheme of the projective scheme $\mathrm{Quot}_{\mathcal{X}/k}(V\otimes\mathcal{E}\otimes\pi^*\mathcal{O}_{X}(-m),P)$.

Given a $\delta$-semistable pair $(\mathcal{F},\varphi)$, consider the morphism $\varphi:\mathcal{F}_{0}\to\mathcal{F}$, by applying the functor $F_{\mathcal{E}}$, tensoring with $\mathcal{O}_{X}(m)$, and then applying the global section functor, we get
\ben
H^0(F_{\mathcal{E}}(\varphi)(m)): H^0(F_{\mathcal{E}}(\mathcal{F}_{0})(m))\to H^0(F_{\mathcal{E}}(\mathcal{F})(m))
\een
Composing with $\rho^{-1}:H^0(F_{\mathcal{E}}(\mathcal{F})(m))\to V$, we have
\ben
a:=\rho^{-1}\circ H^0(F_{\mathcal{E}}(\varphi)(m)):H^0(F_{\mathcal{E}}(\mathcal{F}_{0})(m))\to V
\een
Then we have the commutative diagram
\ben
\xymatrix{
	H^0(F_{\mathcal{E}}(\mathcal{F}_{0})(m))\otimes\mathcal{O}_{X}(-m)\ar[d]_{a} \ar[r]^-{\mathrm{ev}} &F_{\mathcal{E}}(\mathcal{F}_{0}) \ar[d]^{F_{\mathcal{E}}(\varphi)}   \\
	V\otimes\mathcal{O}_{X}(-m)\ar[r]_-{\widetilde{\rho}}& F_{\mathcal{E}}(\mathcal{F}) &   
}
\een
where $\mathrm{ev}$ is the evaluation map, $a$ on the left column denotes $a\otimes\mathrm{id}$  and $\widetilde{\rho}:=\mathrm{ev}\circ(\rho\otimes \mathrm{id})$.
By applying the functor $G_{\mathcal{E}}$ and then using the natural transformation $\theta_{\mathcal{E}}$, we have the commutative diagram
\ben
\xymatrix{
	H^0(F_{\mathcal{E}}(\mathcal{F}_{0})(m))\otimes\mathcal{E}\otimes\pi^*\mathcal{O}_{X}(-m)\ar[d]_{a} \ar[r]^-{\widetilde{\mathrm{ev}}} &   \mathcal{F}_{0}\ar[d]^{\varphi}\\
	V\otimes\mathcal{E}\otimes\pi^*\mathcal{O}_{X}(-m)\ar[r]_-{q} &\mathcal{F}
}
\een
where $a$ again denotes $a\otimes\mathrm{id}$, $q:=\theta_{\mathcal{E}}(\mathcal{F})\circ G_{\mathcal{E}}(\widetilde{\rho})$ and $\widetilde{\mathrm{ev}}:=\theta_{\mathcal{E}}(\mathcal{F}_{0})\circ G_{\mathcal{E}}(\mathrm{ev})$. Notice that $a=0$ if and only if $\varphi=0$.
Then $a\neq0$ since the pair $(\mathcal{F},\varphi)$ is nondegenerate.

Define \ben
&&\widetilde{\mathcal{Q}}:=\mathrm{Quot}_{\mathcal{X}/k}(V\otimes\mathcal{E}\otimes\pi^*\mathcal{O}_{X}(-m),P),\\
&&\mathbb{P}:=\mathbb{P}(\mathrm{Hom}(H^0(F_{\mathcal{E}}(\mathcal{F}_{0})(m)),V)).
\een
Since the morphism $\widetilde{\mathrm{ev}}$ is surjective, assume that $\mathcal{K}_{0}$ is its kernel, then we have a short exact sequence:
\ben
0\to\mathcal{K}_{0}\xrightarrow{\iota} H^0(F_{\mathcal{E}}(\mathcal{F}_{0})(m))\otimes\mathcal{E}\otimes\pi^*\mathcal{O}_{X}(-m)\xrightarrow{\widetilde{\mathrm{ev}}}\mathcal{F}_{0}\to0.
\een
As in [\cite{Wan}, Proposition 3.4], the composition $q\circ a$ factor through $\widetilde{\mathrm{ev}}$ if and only if $q\circ a \circ \iota=0$, and there is a closed subscheme $\mathcal{Z}\subseteq\mathbb{P}\times\widetilde{\mathcal{Q}}$ such that  
$([a],[q])\in\mathcal{Z}$ if and only if  $q\circ a \circ \iota=0$. In particular,  a point $([a],[q])\in\mathcal{Z}$ corresponds to a pair $(\mathcal{F},\varphi)$ induced by  $q\circ a$ factoring through $\widetilde{\mathrm{ev}}$.

\begin{remark}\label{univ-fam}
As in [\cite{BS}, Section 4.1], one has a flat family of pairs parameterized by $\mathcal{Z}\subseteq\mathbb{P}\times\widetilde{\mathcal{Q}}$ as follows.
From now on, we make the convention that the notation $p_{\bullet,\star}$ denotes the natural projection from $\bullet$ to $\star$.  Let  $\widetilde{i}:\mathcal{Z}\to\mathbb{P}\times\widetilde{\mathcal{Q}}$ be the inclusion.
Let $\breve{q}:p_{\mathcal{X}\times\widetilde{\mathcal{Q}},\mathcal{X}}^*(V\otimes\mathcal{E}\otimes\pi^*\mathcal{O}_{X}(-m))\to\widetilde{\mathcal{F}}$ be the universal quotient family parameterized by $\widetilde{\mathcal{Q}}$. Then there is a quotient morphism
\ben
p_{\mathcal{X}\times\mathbb{{P}}\times\widetilde{\mathcal{Q}},\mathcal{X}\times\widetilde{\mathcal{Q}}}^*\breve{q}:p_{\mathcal{X}\times\mathbb{{P}}\times\widetilde{\mathcal{Q}},\mathcal{X}}^*(V\otimes\mathcal{E}\otimes\pi^*\mathcal{O}_{X}(-m))\to p_{\mathcal{X}\times\mathbb{{P}}\times\widetilde{\mathcal{Q}},\mathcal{X}\times\widetilde{\mathcal{Q}}}^*\widetilde{\mathcal{F}}.
\een
The universal quotient family parameterized by $\mathbb{P}$ is 
\ben
\mathcal{H}om(V\otimes\mathcal{O}_{\mathbb{P}},H^0(F_{\mathcal{E}}(\mathcal{F}_{0})(m))\otimes\mathcal{O}_{\mathbb{P}})\to\mathcal{O}_{\mathbb{P}}(1),
\een
or 
\ben
H^0(F_{\mathcal{E}}(\mathcal{F}_{0})(m))\otimes\mathcal{O}_{\mathbb{P}}\xrightarrow{\check{a}} V\otimes\mathcal{O}_{\mathbb{P}}(1).
\een
Then we have 
\ben
&&p_{\mathcal{X}\times\mathbb{{P}}\times\widetilde{\mathcal{Q}},\mathcal{X}}^*(H^0(F_{\mathcal{E}}(\mathcal{F}_{0})(m))\otimes\mathcal{E}\otimes\pi^*\mathcal{O}_{X}(-m))\otimes p_{\mathcal{X}\times\mathbb{P}\times\widetilde{\mathcal{Q}},\mathbb{P}}^*\mathcal{O}_{\mathbb{P}}\\
&&\xrightarrow{\mathrm{id}_{p_{\mathcal{X}\times\mathbb{{P}}\times\widetilde{\mathcal{Q}},\mathcal{X}}^*(\mathcal{E}\otimes\pi^*\mathcal{O}_{X}(-m))}\otimes p_{\mathcal{X}\times\mathbb{P}\times\widetilde{\mathcal{Q}},\mathbb{P}}^*\check{a}} p_{\mathcal{X}\times\mathbb{{P}}\times\widetilde{\mathcal{Q}},\mathcal{X}}^*(V\otimes\mathcal{E}\otimes\pi^*\mathcal{O}_{X}(-m))\otimes p_{\mathcal{X}\times\mathbb{P}\times\widetilde{\mathcal{Q}},\mathbb{P}}^*\mathcal{O}_{\mathbb{P}}(1)
\een	
and 
\ben
p_{\mathcal{X}\times\mathbb{{P}}\times\widetilde{\mathcal{Q}},\mathcal{X}}^*\widetilde{\mathrm{ev}}:p_{\mathcal{X}\times\mathbb{{P}}\times\widetilde{\mathcal{Q}},\mathcal{X}}^*(H^0(F_{\mathcal{E}}(\mathcal{F}_{0})(m))\otimes\mathcal{E}\otimes\pi^*\mathcal{O}_{X}(-m))\to p_{\mathcal{X}\times\mathbb{{P}}\times\widetilde{\mathcal{Q}},\mathcal{X}}^*\mathcal{F}_{0}.
\een
Combined with the definition of $\mathcal{Z}$, we have a morphism
\ben
p_{\mathcal{X}\times\mathcal{Z},\mathcal{X}}^*\mathcal{F}_{0}=(\mathrm{id}_{\mathcal{X}}\times\widetilde{i})^*p_{\mathcal{X}\times\mathbb{{P}}\times\widetilde{\mathcal{Q}},\mathcal{X}}^*\mathcal{F}_{0}\to (\mathrm{id}_{\mathcal{X}}\times\widetilde{i})^*p_{\mathcal{X}\times\mathbb{{P}}\times\widetilde{\mathcal{Q}},\mathcal{X}\times\widetilde{\mathcal{Q}}}^*\widetilde{\mathcal{F}}\otimes(\mathrm{id}_{\mathcal{X}}\times\widetilde{i})^* p_{\mathcal{X}\times\mathbb{P}\times\widetilde{\mathcal{Q}},\mathbb{P}}^*\mathcal{O}_{\mathbb{P}}(1)
\een
which is obviously a flat family of pairs parameterized by $\mathcal{Z}$ in the sense of Definition \ref{fam-pair}.   In the sense of  [\cite{BS}, Definition 3.16] (see also Remark \ref{comp-fam}), the  flat family of pairs is given as
\ben
\left((\mathrm{id}_{\mathcal{X}}\times\widetilde{i})^*p_{\mathcal{X}\times\mathbb{{P}}\times\widetilde{\mathcal{Q}},\mathcal{X}\times\widetilde{\mathcal{Q}}}^*\widetilde{\mathcal{F}},\;\; \widetilde{i}^* p_{\mathbb{P}\times\widetilde{\mathcal{Q}},\mathbb{P}}^*\mathcal{O}_{\mathbb{P}}(-1), \; \phi_{\widetilde{\mathcal{F}}}\right)
\een
where 
\ben
\phi_{\widetilde{\mathcal{F}}}:\widetilde{i}^* p_{\mathbb{P}\times\widetilde{\mathcal{Q}},\mathbb{P}}^*\mathcal{O}_{\mathbb{P}}(-1)\to p_{\mathcal{X}\times\mathcal{Z},\mathcal{Z}*}\mathcal{H}om\left(p_{\mathcal{X}\times\mathcal{Z},\mathcal{X}}^*\mathcal{F}_{0},(\mathrm{id}_{\mathcal{X}}\times\widetilde{i})^*p_{\mathcal{X}\times\mathbb{{P}}\times\widetilde{\mathcal{Q}},\mathcal{X}\times\widetilde{\mathcal{Q}}}^*\widetilde{\mathcal{F}}\right).
\een
\end{remark}

The subset of $\mathcal{Z}$ in which every  point $([a],[q])$  satisfying
$[q]\in\mathbf{Q}$ and  $\mathcal{F}=q(V\otimes\mathcal{E}\otimes\pi^*\mathcal{O}_{X}(-m))$ is pure forms an open subset. In fact, the first condition is an open condition mentioned above. With the notation in Remark \ref{univ-fam}, consider the  universal quotient family parameterized by $\mathcal{Z}$:
\ben
p_{\mathcal{X}\times\mathcal{Z},\mathcal{X}}^*(V\otimes\mathcal{E}\otimes\pi^*\mathcal{O}_{X}(-m))\to (\mathrm{id}_{\mathcal{X}}\times\widetilde{i})^*p_{\mathcal{X}\times\mathbb{{P}}\times\widetilde{\mathcal{Q}},\mathcal{X}\times\widetilde{\mathcal{Q}}}^*\widetilde{\mathcal{F}}
\een
The second condition
is equivalent to the requirement that $(\mathrm{id}_{\mathcal{X}}\times\widetilde{i})^*p_{\mathcal{X}\times\mathbb{{P}}\times\widetilde{\mathcal{Q}},\mathcal{X}\times\widetilde{\mathcal{Q}}}^*\widetilde{\mathcal{F}}$ is pure, which is   an open condition by using [\cite{HL3}, Proposition 2.3.1]  for the  sheaf $F_{p_{\mathcal{X}\times\mathcal{Z},\mathcal{X}}^*\mathcal{E}}\left((\mathrm{id}_{\mathcal{X}}\times\widetilde{i})^*p_{\mathcal{X}\times\mathbb{{P}}\times\widetilde{\mathcal{Q}},\mathcal{X}\times\widetilde{\mathcal{Q}}}^*\widetilde{\mathcal{F}}\right)$ being of pure dimension (if and only if $(\mathrm{id}_{\mathcal{X}}\times\widetilde{i})^*p_{\mathcal{X}\times\mathbb{{P}}\times\widetilde{\mathcal{Q}},\mathcal{X}\times\widetilde{\mathcal{Q}}}^*\widetilde{\mathcal{F}}$ is pure by Lemma \ref{puretopure}), together with Grothendieck's Lemma \ref{Gro-lem} and  [\cite{Nir1}, Proposition 4.20]. Let $\mathcal{Z}^\prime$ be  the closure of this open  subset  in $\mathcal{Z}$. We call $\mathcal{Z}^\prime$ the parameter space. Obviously, $\mathcal{Z}^\prime$ is a projective scheme.

Next, we will consider a $\mathrm{GL}(V)$-action on $\mathcal{Z}^\prime\subseteq\mathbb{P}\times\widetilde{\mathcal{Q}}$, which is induced by $\mathrm{GL}(V)$-actions on $\mathbb{P}$ and $\widetilde{\mathcal{Q}}$. We recall these actions  in [\cite{HL3}, Section 4.3] and [\cite{BS}, Section 4.1] as follows.  Let $\tau:V\otimes\mathcal{O}_{\mathrm{GL}(V)}\to V\otimes\mathcal{O}_{\mathrm{GL}(V)}$ be the universal automorphism of $V$ parameterized by $\mathrm{GL}(V)$. By the universal property of $\widetilde{\mathcal{Q}}$, we have a classifying morphism $\eta_{1}:\widetilde{\mathcal{Q}}\times \mathrm{GL}(V)\to\widetilde{\mathcal{Q}}$ which is the right action by composition $[q]\cdot g:=[q\circ g]$, 
such that the following diagram commute
\ben
\xymatrixcolsep{7pc}\xymatrix{
	p_{\mathcal{X}\times\widetilde{\mathcal{Q}}\times\mathrm{GL}(V),\mathcal{X}}^*\left(V\otimes\mathcal{E}\otimes\pi^*\mathcal{O}_{X}(-m)\right)\ar[r]^-{(\mathrm{id}_{\mathcal{X}}\times\eta_{1})^*\check{q}}\ar[d]_{p_{\mathcal{X}\times\widetilde{\mathcal{Q}}\times\mathrm{GL}(V),\mathrm{GL}(V)}^*\tau} &(\mathrm{id}_{\mathcal{X}}\times\eta_{1})^*\widetilde{\mathcal{F}}\ar[d]^{\Lambda_{1}}	\\
p_{\mathcal{X}\times\widetilde{\mathcal{Q}}\times\mathrm{GL}(V),\mathcal{X}}^*\left(V\otimes\mathcal{E}\otimes\pi^*\mathcal{O}_{X}(-m)\right) \ar[r]^-{p_{\mathcal{X}\times\widetilde{\mathcal{Q}}\times\mathrm{GL}(V),\mathcal{X}\times\widetilde{\mathcal{Q}}}^*\check{q}} &   p_{\mathcal{X}\times\widetilde{\mathcal{Q}}\times\mathrm{GL}(V),\mathcal{X}\times\widetilde{\mathcal{Q}}}^*\widetilde{\mathcal{F}}
}
\een
where $\check{q}$ is defined in Remark \ref{univ-fam} and  $\Lambda_{1}:(\mathrm{id}_{\mathcal{X}}\times\eta_{1})^*\widetilde{\mathcal{F}}\to p_{\mathcal{X}\times\widetilde{\mathcal{Q}}\times\mathrm{GL}(V),\mathcal{X}\times\widetilde{\mathcal{Q}}}^*\widetilde{\mathcal{F}}$ is an isomorphism. In Remark \ref{univ-fam}, the universal quotient family parameterized by $\mathbb{P}$ is given as
\ben
H^0(F_{\mathcal{E}}(\mathcal{F}_{0})(m))\otimes\mathcal{O}_{\mathbb{P}}\xrightarrow{\check{a}} V\otimes\mathcal{O}_{\mathbb{P}}(1).
\een
Then we have the following surjective composition
\ben
H^0(F_{\mathcal{E}}(\mathcal{F}_{0})(m))\otimes\mathcal{O}_{\mathbb{P}\times\mathrm{GL}(V)}\xrightarrow{p_{\mathbb{P}\times\mathrm{GL}(V),\mathbb{P}}^*\check{a}}V\otimes p_{\mathbb{P}\times\mathrm{GL}(V),\mathbb{P}}^*\mathcal{O}_{\mathbb{P}}(1)\xrightarrow{p_{\mathbb{P}\times\mathrm{GL}(V),\mathrm{GL}(V)}^*\tau}V\otimes p_{\mathbb{P}\times\mathrm{GL}(V),\mathbb{P}}^*\mathcal{O}_{\mathbb{P}}(1)
\een
which induces a classifying morphism $\eta_{2}:\mathbb{P}\times\mathrm{GL}(V)\to\mathbb{P}$, where $\eta_{2}$ is the  action by composition $[a]\cdot g=[g\circ a]$. Then we have  an isomorphism $\Lambda_{2}:\eta_{2}^*\mathcal{O}_{\mathbb{P}}(1)\to p_{\mathbb{P}\times\mathrm{GL}(V),\mathbb{P}}^*\mathcal{O}_{\mathbb{P}}(1)$ such that we have the following commutative diagram
\ben
\xymatrixcolsep{6pc}\xymatrix{
	H^0(F_{\mathcal{E}}(\mathcal{F}_{0})(m))\otimes\mathcal{O}_{\mathbb{P}\times\mathrm{GL}(V)}\ar[r]^-{\eta_{2}^*\check{a}}\ar[d]_{p_{\mathbb{P}\times\mathrm{GL}(V),\mathbb{P}}^*\check{a}} &  V\otimes\eta_{2}^*\mathcal{O}_{\mathbb{P}}(1) \ar[d]^{p_{\mathbb{P}\times\mathrm{GL}(V),\mathrm{GL}(V)}^*\tau\otimes \Lambda_{2}}	\\
	V\otimes p_{\mathbb{P}\times\mathrm{GL}(V),\mathbb{P}}^*\mathcal{O}_{\mathbb{P}}(1)\ar[r]^-{p_{\mathbb{P}\times\mathrm{GL}(V),\mathrm{GL}(V)}^*\tau} &   V\otimes p_{\mathbb{P}\times\mathrm{GL}(V),\mathbb{P}}^*\mathcal{O}_{\mathbb{P}}(1).
}
\een
Using the stacky version of the $G$-linearization of a coherent sheaf in  [\cite{Rom}, Example 4.3],  it can be verified that $\Lambda_{1}$ is a $\mathrm{GL}(V)$-linearization for $\widetilde{\mathcal{F}}$ and $\Lambda_{2}$ is a $\mathrm{GL}(V)$-linearization for $\mathcal{O}_{\mathbb{P}}(1)$. 

Two morphisms $\eta_{1}$ and $\eta_{2}$ induce a $\mathrm{GL}(V)$-action $\eta:\mathbb{P}\times\widetilde{\mathcal{Q}}\times\mathrm{GL}(V)\to\mathbb{P}\times\widetilde{\mathcal{Q}}$ and hence induces
a natural $\mathrm{SL}(V)$-action on $\mathbb{P}\times\widetilde{\mathcal{Q}}$ pointwise as follows
\ben
([a],[q])\cdot g=([g^{-1}\circ a], [q\circ g])
\een
where $g\in\mathrm{SL}(V)$ and $([a],[q])\in\mathbb{P}\times\widetilde{\mathcal{Q}}$. This is a right $\mathrm{SL}(V)$-action. Obviously, $\mathcal{Z}$ and  $\mathcal{Z}^\prime$ are invariant under this action. For $l$ sufficiently large, it is  shown in [\cite{Nir1}, Lemma 6.3], the class of very ample invertible sheaves $\mathcal{L}_{l}:=\det\left(p_{X\times\widetilde{\mathcal{Q}},\widetilde{\mathcal{Q}}*}\left(F_{p_{\mathcal{X}\times\widetilde{\mathcal{Q}},\mathcal{X}}^*\mathcal{E}}(\widetilde{\mathcal{F}})(l)\right)\right)$  carries a natural $\mathrm{GL}(V)$-linearization which induces a $\mathrm{SL}(V)$-linearization. Notice that $\Lambda_{2}$ induces a $\mathrm{SL}(V)$-linearization for $\mathcal{O}_{\mathbb{P}}(1)$. Then for any two  positive integers $n_{1}$ and $n_{2}$, we have the very ample line bundles 
\ben
\mathcal{O}_{\mathcal{Z}^\prime}(n_{1},n_{2}):=\mathcal{O}_{\mathbb{P}}(n_{1})\boxtimes\mathcal{L}_{l}^{\otimes n_{2}}:=p_{\mathbb{P}\times\widetilde{\mathcal{Q}},\mathbb{P}}^*\mathcal{O}_{\mathbb{P}}(n_{1})\otimes p_{\mathbb{P}\times\widetilde{\mathcal{Q}},\widetilde{\mathcal{Q}}}^*\mathcal{L}_{l}^{n_{2}}
\een
 which again carry  natural $\mathrm{SL}(V)$-linearizations.

\subsection{GIT-stability and $\delta$-stability} 
We have obtained $\mathrm{SL}(V)$-linearized very ample line bundles $\mathcal{O}_{\mathcal{Z}^\prime}(n_{1},n_{2})$ for $l$ sufficiently large and any two integers $n_{1}, n_{2}$, the next step is to consider the GIT (semi)stable points of $\mathcal{Z}^\prime$ and compare with $\delta$-stability condition for the corresponding pairs. 
In this subsection, we first generalize the results which relate GIT-stability condition with $\delta$-stability condition in [\cite{Wan}, Section 4] when 
$\mathrm{deg}\,\delta<\mathrm{deg}\,P$ and then those in [\cite{Lin18}, section 4] when $\mathrm{deg}\,\delta\geq\mathrm{deg}\,P$ to the case of projective Deligne-Mumford stacks.

We recall the relation between linear subspaces of $V$ and subsheaves of $\mathcal{F}$  in [\cite{Nir1}, Remark 6.14], which is useful for  the argument below. Given a quotient $q:V\otimes\mathcal{E}\otimes\pi^*\mathcal{O}_{X}(-m)\to\mathcal{F}$ and  any linear subspace $V^\prime\subseteq V$, we have the induced subsheaf $\mathcal{F}^\prime:=q(V^\prime\otimes\mathcal{E}\otimes\pi^*\mathcal{O}_{X}(-m))\subseteq\mathcal{F}$. On the other hand, given a subsheaf $\mathcal{F}^\prime\subseteq\mathcal{F}$, we have a injective morphism $F_{\mathcal{E}}(\mathcal{F}^\prime)(m)\hookrightarrow F_{\mathcal{E}}(\mathcal{F})(m)$ and hence an inclusion $H^0(F_{\mathcal{E}}(\mathcal{F}^\prime)(m))\hookrightarrow H^0(F_{\mathcal{E}}(\mathcal{F})(m))$. Since we have the following composition
$\widetilde{q}:=F_{\mathcal{E}}(q)(m)\circ\varphi_{\mathcal{E}}(V\otimes\mathcal{O}_{X})=\left(F_{\mathcal{E}}(q)\circ\varphi_{\mathcal{E}}(V\otimes \mathcal{O}_{X}(-m))\right)(m)$:
\ben V\otimes\mathcal{O}_{X}\xrightarrow{\varphi_{\mathcal{E}}(V\otimes\mathcal{O}_{X})}F_{\mathcal{E}}\circ G_{\mathcal{E}}(V\otimes\mathcal{O}_{X})\xrightarrow{F_{\mathcal{E}}(q)(m)}F_{\mathcal{E}}(\mathcal{F})(m),
\een
then we have the following cartesian diagram:
\ben
\xymatrix{
	V\cap H^0(F_{\mathcal{E}}(\mathcal{F}^\prime)(m))\ar[r]\ar[d] &H^0(F_{\mathcal{E}}(\mathcal{F}^\prime)(m))\ar[d]	\\
	V \ar[r]^-{H^0(\widetilde{q})} &   H^0(F_{\mathcal{E}}(\mathcal{F})(m))
}
\een
where we use the notation $V\cap H^0(F_{\mathcal{E}}(\mathcal{F}^\prime)(m))$ for $H^0(\widetilde{q})^{-1}(H^0(F_{\mathcal{E}}(\mathcal{F}^\prime)(m)))$ as in [\cite{HL3}, Lemma 4.4.6].
We call $V^\prime:=V\cap H^0(F_{\mathcal{E}}(\mathcal{F}^\prime)(m))$  the linear space induced by $\mathcal{F}^\prime$. If $\mathcal{F}^\prime$ is induced by $V^\prime\subseteq V$, then $V^\prime\subseteq V\cap H^0(F_{\mathcal{E}}(\mathcal{F}^\prime)(m))$. And if $\mathcal{F}^\prime\subseteq \mathcal{F}$ is an arbitrary subsheaf and $V^\prime=V\cap H^0(F_{\mathcal{E}}(\mathcal{F}^\prime)(m))$, then the subsheaf induced by $V^\prime$ is contained in $\mathcal{F}^\prime$.

Since the parameter space $\mathcal{Z}^\prime$ is a closed set, we have to consider the more general class of sheaves in the following
\begin{lemma} \label{deg-pure}
	If a pair $(\mathcal{F},\varphi)$ on $\mathcal{X}$ of dimension $d$ can be deformed to a pure pair, then there is a pure sheaf $\mathcal{H}$ with a morphism $\psi:\mathcal{F}\to\mathcal{H}$ satisfying $\mathrm{ker}\psi=T_{d-1}(\mathcal{F})$ and $P_{\mathcal{E}}(\mathcal{F})=P_{\mathcal{E}}(\mathcal{H})$.
\end{lemma}
\begin{proof}
Since 	$(\mathcal{F},\varphi)$ can be deformed to  a pure pair, there is  a smooth connected curve $C$ and a flat family $(\mathcal{F}_{C},\varphi_{C})$ on $\mathcal{X}_{C}:=\mathcal{X}\times C$ such that $(\mathcal{F}_{0},\varphi_{0})\cong(\mathcal{F},\varphi)$ for some closed point $0\in C$ and $(\mathcal{F}_{t},\varphi_{t})$ is pure for any point $t\neq0$. This implies that  $\mathcal{F}$ deforms into a pure sheaf. The proof is completed by [\cite{Nir1}, Lemma 6.10].
\end{proof}	
Now, we consider the case when $\mathrm{deg}\,\delta<\mathrm{deg}\,P$. In this case, as in [\cite{Wan}, section 3], two positive integers $n_{1}, n_{2}$ are chosen to satisfy
\be\label{ratio1}
\frac{n_{1}}{n_{2}}=\frac{P(l)\cdot\delta(m)-\delta(l)\cdot P(m)}{P(m)+\delta(m)}.
\ee

\begin{proposition}\label{Criterion1}
Assume  the equality \eqref{ratio1} holds. Let $([a],[q])\in\mathcal{Z}^\prime$ be a point with the corresponding pair $(\mathcal{F},\varphi)$. For $l$ sufficiently large, 
$([a],[q])$ is GIT-(semi)stable with respect to $\mathcal{O}_{\mathcal{Z}^\prime}(n_{1},n_{2})$ if and only if for any nonzero proper subsheaf $\mathcal{F}^\prime$ and the induced subspace $V^\prime:=V\cap H^0(F_{\mathcal{E}}(\mathcal{F}^\prime)(m))$, we have an inequality of polynomials in $l$:
\ben
P\cdot(\dim V^\prime+\epsilon(\varphi^\prime)\delta(m))+\delta\cdot(\dim V^\prime-\epsilon(\varphi^\prime)P(m)) (\leq) P_{\mathcal{E}}(\mathcal{F}^\prime)\cdot(P(m)+\delta(m)).
\een
where $(\mathcal{F}^\prime,\varphi^\prime)$ is the induced subpair of $(\mathcal{F},\varphi)$.
\end{proposition}
\begin{proof}
Let $([a],[q])\in\mathcal{Z}^\prime$ be represented by morphisms  $a:H^0(F_{\mathcal{E}}(\mathcal{F}_{0})(m))\to V$
and $q:V\otimes\mathcal{E}\otimes\pi^*\mathcal{O}_{X}(-m)\to\mathcal{F}$.  Let $\varphi:\mathcal{F}_{0}\to\mathcal{F}$ be the corresponding pair. For any $[q]\in\widetilde{\mathcal{Q}}$, we have a point $[F_{\mathcal{E}}(q)]$  in $\mathrm{Quot}_{X/k}(V\otimes\pi_{*}\mathcal{E}nd_{\mathcal{O}_{\mathcal{X}}}(\mathcal{E})(-m),P)$ and hence the set-theoretic family $\{F_{\mathcal{E}}(\mathcal{F})|[q]\in\widetilde{\mathcal{Q}}\}$ with fixed Hilbert polynomial $P$ is bounded. Then we can choose an integer $l(\geq m)$ sufficiently large such that  $F_{\mathcal{E}}(\mathcal{F})$ and $\pi_{*}\mathcal{E}nd_{\mathcal{O}_{\mathcal{X}}}(\mathcal{E})(-m)$ are both $l$-regular for any $q\in\widetilde{\mathcal{Q}}$. Using the same argument in the proof of   [\cite{Wan}, Proposition 4.1], the Hilbert-Mumford criterion shows the following:

$([a],[q])$ is GIT-(semi)stable with respect to $\mathcal{O}_{\mathcal{Z}^\prime}(n_{1},n_{2})$ if and only if for any nontrivial subspace $U\subseteq V$ we have
\be\label{GIT1}
\dim U\cdot(n_{2}P(l)-n_{1})(\leq)P(m)\cdot(\dim(q^\prime(U\otimes W))n_{2}-\epsilon(U)n_{1}).
\ee
where $W:=H^0(\pi_{*}\mathcal{E}nd_{\mathcal{O}_{\mathcal{X}}}(\mathcal{E})(l-m))$, $q^\prime: =H^0(F_{\mathcal{E}}(q)(l)):V\otimes W\to H^0(F_{\mathcal{E}}(\mathcal{F})(l))$ and $\epsilon(U)=1$ if $\mathrm{im}a\subseteq U$ and 0 otherwise.

It is shown in the proof  of [\cite{Nir1}, Lemma 6.15], for a fixed $q$, the family of subsheaves $\mathcal{F}_{U}\subseteq\mathcal{F}$ induced by a linear subspace $U\subseteq V$ is bounded since exact sequences of linear spaces split which implies that any subsheaf $\mathcal{F}_{U}$ has the same regularity as $\mathcal{F}$. By Kleiman criterion for stacks, for $l$ sufficiently large, all the subsheaves $\mathcal{F}_{U}$ are $l$-regular, that is, all $F_{\mathcal{E}}(\mathcal{F}_{U})$ are $l$-regular, and then we have $P_{\mathcal{E}}(\mathcal{F}_{U})(l)=h^0(F_{\mathcal{E}}(\mathcal{F}_{U})(l))=\dim(q^\prime(U\otimes W))$ for any $\mathcal{F}_{U}$. Using the similar argument in the proofs of  [\cite{Wan}, Proposition 4.3 and Proposition 4.4] and the equality \eqref{ratio1}, the criterion for GIT-(semi)stability above can be restated as:

$([a],[q])$  is GIT-(semi)stable with respect to $\mathcal{O}_{\mathcal{Z}^\prime}(n_{1},n_{2})$ if and only if for any nontrivial proper subspace $U\subseteq V$  and the induced sheaf $\mathcal{F}_{U}:=q(U\otimes\mathcal{E}\otimes\pi^*\mathcal{O}_{X}(-m))$, we have
an inequality of polynomials in $l$:
\ben
P\cdot(\dim U+\epsilon(\varphi|_{U})\delta(m))+\delta\cdot(\dim U-\epsilon(\varphi|_{U})P(m)) (\leq) P_{\mathcal{E}}(\mathcal{F}_{U})\cdot(P(m)+\delta(m)).
\een
where  $\varphi|_{U}=\varphi$ if $\mathrm{im}\varphi\subseteq\mathcal{F}_{U}$ and $\varphi|_{U}=0$ otherwise.

The proof is completed by using the similar argument in the proof of  [\cite{Wan}, Theorem 4.5] to drop the restriction to subsheaves $\mathcal{F}_{U}$ induced by subspaces $U\subseteq V$ in the last statement of the criterion but instead use any subsheave
$\mathcal{F}^\prime$ with the induced subpace $V^\prime:=V\cap H^0(F_{\mathcal{E}}(\mathcal{F}^\prime)(m))$.
\end{proof}

\begin{theorem}\label{git=pair1}
Assume  $\mathrm{deg}\,\delta<\mathrm{deg}\,P$ and the equality \eqref{ratio1} holds. For $l$ sufficiently large,  a point $([a],[q])\in\mathcal{Z}^\prime$ is GIT-(semi)stable with respect to $\mathcal{O}_{\mathcal{Z}^\prime}(n_{1},n_{2})$ if and only if  the corresponding pair $(\mathcal{F},\varphi)$ is $\delta$-(semi)stable and the map  $H^0(\widetilde{q}):V\to H^0(F_{\mathcal{E}}(\mathcal{F})(m))$ induced by $q$ is an isomorphism.
\end{theorem}
\begin{proof}
Let $([a],[q])\in\mathcal{Z}^\prime$ be a  GIT-semistable point.
	Let $U\subseteq V$ be the kernel of $H^0(\widetilde{q}):V\to H^0(F_{\mathcal{E}}(\mathcal{F})(m))$, then it is shown in the proof of [\cite{Nir1}, Lemma 6.16] that
	\ben
	\mathcal{F}_{U}:=q(U\otimes\mathcal{E}\otimes\pi^*\mathcal{O}_{X}(-m))=0.
	\een
	In fact, if $\mathcal{F}_{U}\neq0$, we have $F_{\mathcal{E}}(\mathcal{F}_{U})\neq0$. Then $U\subseteq V\cap H^0(F_{\mathcal{E}}(\mathcal{F}_{U})(m))$ and the cartesian diagram 
	\ben
	\xymatrix{
		V\cap H^0(F_{\mathcal{E}}(\mathcal{F}_{U})(m))\ar[r]\ar[d] &H^0(F_{\mathcal{E}}(\mathcal{F}_{U})(m))\ar[d]	\\
		V \ar[r]^-{H^0(\widetilde{q})} &   H^0(F_{\mathcal{E}}(\mathcal{F})(m))
	}
	\een
	shows that  the map $U\to H^0(F_{\mathcal{E}}(\mathcal{F}_{U})(m))$ is zero. Since $\mathcal{F}_{U}\neq0$, we have a nonzero morphism $U\otimes\mathcal{O}_{\mathcal{X}}\to \mathcal{F}_{U}\otimes\mathcal{E}^\vee\otimes\pi^*\mathcal{O}_{X}(m)$ and hence $U\otimes O_{X}\to F_{\mathcal{E}}(\mathcal{F}_{U})(m)$ is nonzero by applying the exact functor $\pi_{*}$. This produces a contradiction.
	Thus, for $l$ sufficiently large,  $q^\prime(U\otimes W)=0$. By the inequality \eqref{GIT1} and $P(l)>\frac{n_{1}}{n_{2}}$, one has $U=0$. Then
	$H^0(\widetilde{q}):V\to H^0(F_{\mathcal{E}}(\mathcal{F})(m))$ is injective.
	
The proof is completed by following the similar argument in the proofs of [\cite{Wan}, Theorem 4.5 and Theorem 4.7] together with the corresponding Lemma \ref{bound2},  Lemma \ref{deg-pure} and Proposition \ref{Criterion1}.
\end{proof}
Next is the case when $\mathrm{deg}\,\delta\geq\mathrm{deg}\,P$ where as in [\cite{Lin18}, Section 4] we  set 
\be\label{ratio2}
\frac{n_{1}}{n_{2}}=\frac{P(l)}{2r}.
\ee
Applying the proof of Proposition 4.7 with the equality \eqref{ratio2}, we have
\begin{proposition}\label{Criterion2}
Assume  the equality \eqref{ratio2} holds.
Let $([a],[q])\in\mathcal{Z}^\prime$ be a point with the corresponding pair $(\mathcal{F},\varphi)$. For $l$ sufficiently large, 
$([a],[q])$ is GIT-(semi)stable with respect to $\mathcal{O}_{\mathcal{Z}^\prime}(n_{1},n_{2})$ if and only if for any nonzero proper subsheaf $\mathcal{F}^\prime$ and the induced subspace $V^\prime:=V\cap H^0(F_{\mathcal{E}}(\mathcal{F}^\prime)(m))$, we have an inequality of polynomials in $l$:
\ben
P\cdot((2r-1)\dim V^\prime+P(m)\epsilon(\varphi^\prime))(\leq)P_{\mathcal{E}}(\mathcal{F}^\prime)\cdot(2r\cdot P(m))
\een
where $(\mathcal{F}^\prime,\varphi^\prime)$ is the induced subpair of $(\mathcal{F},\varphi)$.
\end{proposition}

\begin{theorem}\label{git=pair2}
Assume $\mathrm{deg}\,\delta\geq\mathrm{deg}\,P$ and the equality \eqref{ratio2} holds. For $l$ sufficiently large,  a point $([a],[q])\in\mathcal{Z}^\prime$ is GIT-(semi)stable with respect to $\mathcal{O}_{\mathcal{Z}^\prime}(n_{1},n_{2})$ if and only if  the corresponding pair $(\mathcal{F},\varphi)$ is $\delta$-(semi)stable and the map  $H^0(\widetilde{q}):V\to H^0(F_{\mathcal{E}}(\mathcal{F})(m))$ induced by $q$ is an isomorphism.
\end{theorem}
\begin{proof}
The proof follows from the similar argument in Theorem \ref{git=pair1}, but together with Lemma \ref{bound3} and Proposition \ref{Criterion2}.
\end{proof}
\begin{remark}
By Lemma \ref{no-str-semi}, the statement in Theorem \ref{git=pair2}  actually prove that a GIT-stable point corresponds to a $\delta$-stable pair.
\end{remark}
\subsection{Moduli spaces of $\delta$-(semi)stable pairs}
In this subsection, we  introduce a moduli functor of $\delta$-(semi)stable pairs, 
and then prove that there is a moduli space for this functor. These results  generalize the ones in [\cite{Wan}, Theorem 3.8] and [\cite{Lin18}, Theorem 1.1] to the case of projective Deligne-Mumford stacks. We will adopt the techniques and ideas  in [\cite{BS}, Section 4], most of which can be traced back to the theory for the case of projective schemes in [\cite{HL3}, Chapter 4] (see also [\cite{HL1,HL2}]). Different from the definition of a flat family of pairs  [\cite{BS}, Definition 3.16], we take  Definition \ref{fam-pair} for our families. With this notion of families, as in [\cite{HL3}, Section 4.1], we give
\begin{definition}\label{moduli-functor}
	Define a functor
	\ben
	\mathcal{M}^{ss}_{\mathcal{X}/k}(\mathcal{F}_{0},P,\delta): (\mathrm{Sch}/k)^\circ\to(\mathrm{Sets})
	\een
	as follows. If $S$ is a $k$-scheme of finite type, let $\mathcal{M}^{ss}_{\mathcal{X}/k}(\mathcal{F}_{0},P,\delta)(S)$ be the set of isomorphism classes of flat families of $\delta$-semistable pairs  $(\mathcal{F},\varphi)$ with  
	modified Hilbert polynomial $P$ parametrized by  $S$, that is, such a flat family $(\mathcal{F},\varphi)$ satisfies that for each point $s\in S$, the pair $(\mathcal{F}_{s},\varphi|_{(\pi_{\mathcal{X}}^*\mathcal{F}_{0})_{s}})$ is a $\delta$-semistable pair with  modified Hilbert polynomial $P_{\check{\pi}_{s}^*\mathcal{E}}(\mathcal{F}_{s})=P$ where $\pi_{\mathcal{X}}:\mathcal{X}\times S\to \mathcal{X}$ and  $\check{\pi}_{s}:\mathcal{X}\times\mathrm{Spec}(k(s))\to\mathcal{X}$ are the projections. And for every morphism of $k$-schemes $f: S^\prime\to S$, we obtain a map
	\ben
	\mathcal{M}^{ss}_{\mathcal{X}/k}(\mathcal{F}_{0},P,\delta)(f):\mathcal{M}^{ss}_{\mathcal{X}/k}(\mathcal{F}_{0},P,\delta)(S)\to\mathcal{M}^{ss}_{\mathcal{X}/k}(\mathcal{F}_{0},P,\delta)(S^\prime)
	\een
	via pulling back $\mathcal{F}$ and $\varphi$ by $\mathrm{id}_{\mathcal{X}}\times f$. If we take families of $\delta$-stable pairs, we denote the corresponding subfunctor by $\mathcal{M}^s_{\mathcal{X}/k}(\mathcal{F}_{0},P,\delta)$. 
\end{definition}

\begin{remark}
By [\cite{Nir1}, Proposition 1.5] and Theorem \ref{base-chag-proj}, we have  $F_{\pi_{\mathcal{X}}^*\mathcal{E}}(\mathcal{F})_{s}=F_{\check{\pi}_{s}^*\mathcal{E}}(\mathcal{F}_{s})$. Then one has $P(F_{\pi_{\mathcal{X}}^*\mathcal{E}}(\mathcal{F})_{s}(m))=P_{\check{\pi}_{s}^*\mathcal{E}}(\mathcal{F}_{s})(m)=P(m)$ for each point $s\in S$ in the above definition.
\end{remark}

Recall that a scheme $M$ is called a (coarse) moduli space for the functor $\mathcal{M}$ if it correpresents the functor $\mathcal{M}$ and called a fine moduli space for $\mathcal{M}$ if it represents $\mathcal{M}$. One can refer to [\cite{HL3}, Section 2.2, 4.1 and 4.6] for more details. In order to apply GIT to  prove the existence of moduli spaces for functors  $\mathcal{M}^{(s)s}_{\mathcal{X}/k}(\mathcal{F}_{0},P,\delta)$, we will consider the GIT-(semi)stable points. 

Define $\mathcal{R}^{(s)s}\subseteq\mathcal{Z}^\prime\subseteq\mathcal{Z}$ to be 
the subset of GIT-(semi)stable points $([a],[q])$ corresponding to  $\delta$-(semi)stable pairs by Theorem \ref{git=pair1} and Theorem \ref{git=pair2}.
As in [\cite{HL3}, Section 4.2],  the sets $\mathcal{R}^{s}$ and $\mathcal{R}^{ss}$ are  open  $\mathrm{SL}(V)$-invariant subset of $\mathcal{Z}^\prime$. In particular, $\mathcal{R}^s$ is an open subset of $\mathcal{R}^{ss}$.  In Remark \ref{univ-fam}, we have obtained  a flat family $(\check{\mathcal{F}}_{\mathcal{X}\times\mathcal{Z}},\check{\varphi}_{\mathcal{X}\times\mathcal{Z}})$ of pairs parametrized by $\mathcal{Z}$, where 
\ben
\check{\mathcal{F}}_{\mathcal{X}\times\mathcal{Z}}:=(\mathrm{id}_{\mathcal{X}}\times\widetilde{i})^*p_{\mathcal{X}\times\mathbb{{P}}\times\widetilde{\mathcal{Q}},\mathcal{X}\times\widetilde{\mathcal{Q}}}^*\widetilde{\mathcal{F}}\otimes(\mathrm{id}_{\mathcal{X}}\times\widetilde{i})^* p_{\mathcal{X}\times\mathbb{P}\times\widetilde{\mathcal{Q}},\mathbb{P}}^*\mathcal{O}_{\mathbb{P}}(1)
\een
and 
$
\check{\varphi}_{\mathcal{X}\times\mathcal{Z}}:p_{\mathcal{X}\times\mathcal{Z},\mathcal{X}}^*\mathcal{F}_{0}\to \check{\mathcal{F}}_{\mathcal{X}\times\mathcal{Z}}
$.  As in [\cite{BS}, Proposition 4.2], the scheme $\mathcal{Z}^\prime$ has the universal property due to the ones of $\mathbb{P}$ and $\widetilde{\mathcal{Q}}$. By pulling back the flat family $(\check{\mathcal{F}}_{\mathcal{X}\times\mathcal{Z}},\check{\varphi}_{\mathcal{X}\times\mathcal{Z}})$ to $\mathcal{R}^{(s)s}$, we get the universal family  $(\check{\mathcal{F}}_{\mathcal{X}\times\mathcal{R}^{(s)s}},\check{\varphi}_{\mathcal{X}\times\mathcal{R}^{(s)s}})$ of $\delta$-semistable pairs with  
modified Hilbert polynomial $P$ parameterized by $\mathcal{R}^{(s)s}$ where  the morphism $\check{\varphi}_{\mathcal{X}\times\mathcal{R}^{(s)s}}$ is 
\ben
\check{\varphi}_{\mathcal{X}\times\mathcal{R}^{(s)s}}: p_{\mathcal{X}\times\mathcal{R}^{(s)s},\mathcal{X}}^*\mathcal{F}_{0}\to \check{\mathcal{F}}_{\mathcal{X}\times\mathcal{R}^{(s)s}}:=(\mathrm{id}_{\mathcal{X}}\times\widetilde{i}_{(s)s})^*p_{\mathcal{X}\times\mathbb{{P}}\times\widetilde{\mathcal{Q}},\mathcal{X}\times\widetilde{\mathcal{Q}}}^*\widetilde{\mathcal{F}}\otimes(\mathrm{id}_{\mathcal{X}}\times\widetilde{i}_{(s)s})^* p_{\mathcal{X}\times\mathbb{P}\times\widetilde{\mathcal{Q}},\mathbb{P}}^*\mathcal{O}_{\mathbb{P}}(1)
\een
and $\widetilde{i}_{(s)s}: \mathcal{R}^{(s)s}\to\mathbb{P}\times\widetilde{\mathcal{Q}}$ are the inclusions.

Since $\mathrm{SL}(V)$-linearized ample line bundles $\mathcal{O}_{\mathcal{Z}^\prime}(n_{1},n_{2})$ depend on choices of two positive integers $n_{1}$ and  $n_{2}$, we make the convention that they are chosen to satisfy the equality \eqref{ratio1}  when $\mathrm{deg}\,\delta<\mathrm{deg}\,P$ and the equality \eqref{ratio2} if $\mathrm{deg}\,\delta\geq\mathrm{deg}\,P$. By [\cite{HL3}, Theorem 4.2.10], we have
\begin{theorem}\label{GIT-quo1}
There is a projective scheme $M^{ss}:=M^{ss}_{\mathcal{X}/k}(\mathcal{F}_{0},P,\delta)$ and a morphism $\Theta: \mathcal{R}^{ss}\to M^{ss}$ such that $\Theta$ is a universal good quotient for the $\mathrm{SL}(V)$-action on $\mathcal{R}^{ss}$. And there is an open subscheme $M^s:=M^s_{\mathcal{X}/k}(\mathcal{F}_{0},P,\delta)\subseteq M^{ss}$ such that $\mathcal{R}^s=\Theta^{-1}(M^s)$ and $\Theta:\mathcal{R}^s\to M^s$ is a universal geometric quotient. Moreover, there is  a positive integer $\hat{l}$ and a very ample line bundle $\mathbb{M}$ on $M^{ss}$ such that $\mathcal{O}_{\mathcal{Z}^\prime}(n_{1},n_{2})^{\otimes\hat{l}}|_{\mathcal{R}^{ss}}\cong\Theta^{*}(\mathbb{M})$.

\end{theorem}

To obtain the similar results in [\cite{HL3}, Theorem 4.3.3], we need the following semicontinuous result for  the Hom group of pairs. The proof of this result is mainly based on the argument in that of [\cite{HL2}, Lemma 3.4]. 
We will adopt the slightly modified proof of [\cite{BS}, Proposition A.2] as follows.
\begin{lemma}\label{semi-conti}
Let $\mathcal{X}$ be a projective Deligne-Mumford stack over $k$ with a polarization $(\mathcal{E},\mathcal{O}_{X}(1))$.
Let $(\mathcal{F},\varphi)$ and $(\mathcal{G},\psi)$ be  two flat families of pairs over 
$\mathcal{X}_{T}:=\mathcal{X}\times T$ parametrized by a scheme $T$ of finite type over $k$. Then the function 
\ben
t\to\dim_{k(t)}\mathrm{Hom}_{\mathcal{X}\times \{t\}}((\mathcal{F}_{t},\varphi_{t}),(\mathcal{G}_{t},\psi_{t}))
\een
is an upper semicontinuous function on $T$.
\end{lemma}
\begin{proof}
It suffices to prove the case when $T=\mathrm{Spec}\,A$ where $A$ is a $k$-algebra of finite type. Since  $p_{T}: \mathcal{X}\times T\to T$ is a family of projective stacks, we have the following locally free  resolutions (see also the argument in Section 5.1):
\ben
&&p_{\mathcal{X}}^*(\mathcal{E})^{\oplus N_{1}}\otimes p_{\mathcal{X}}^*\circ\pi^*\mathcal{O}_{X}(-m_{1})\to p_{\mathcal{X}}^*(\mathcal{E})^{\oplus N_{2}}\otimes p_{\mathcal{X}}^*\circ\pi^*\mathcal{O}_{X}(-m_{2})\to\mathcal{F}\to0,\\
&&p_{\mathcal{X}}^*(\mathcal{E})^{\oplus \widetilde{N}_{1}}\otimes p_{\mathcal{X}}^*\circ\pi^*\mathcal{O}_{X}(-\widetilde{m}_{1})\to p_{\mathcal{X}}^*(\mathcal{E})^{\oplus \widetilde{N}_{2}}\otimes p_{\mathcal{X}}^*\circ\pi^*\mathcal{O}_{X}(-\widetilde{m}_{2})\to p_{\mathcal{X}}^*\mathcal{F}_{0}\to0,
\een
where $p_{\mathcal{X}}:\mathcal{X}\times T\to\mathcal{X}$ is the natural projection, $N_{i}, \widetilde{N}_{i}$ are positive integers, and $m_{i}, \widetilde{m}_{i}$ are positive  integers large enough for $i=1, 2$.
Then we have the exact sequences for any $A$-module $M$
\ben
&&0\to\mathrm{Hom}(\mathcal{F},\mathcal{G}\otimes_{A}M)\to M_{\mathcal{F}}^2\otimes_{A}M\to M_{\mathcal{F}}^1\otimes_{A}M,\\
&&0\to\mathrm{Hom}(p_{\mathcal{X}}^*\mathcal{F}_{0},\mathcal{G}\otimes_{A}M)\to M_{\mathcal{F}_{0}}^2\otimes_{A}M\to M_{\mathcal{F}_{0}}^1\otimes_{A}M,
\een
where 
\ben
&&M_{\mathcal{F}}^i:=H^0(X\times T, F_{p_{\mathcal{X}}^*(\mathcal{E})}(\mathcal{G})^{\oplus N_{i}}(m_{i})),\\
&&M_{\mathcal{F}_{0}}^i:=H^0(X\times T, F_{p_{\mathcal{X}}^*(\mathcal{E})}(\mathcal{G})^{\oplus \widetilde{N}_{i}}(\widetilde{m}_{i})),
\een
 and $A$-module $M_{\mathcal{F}}^i, M_{\mathcal{F}_{0}}^i$ are free for $i=1, 2$ by the similar argument  in the proof of [\cite{Nir1}, Lemma 6.18].
Define the follwoing complexes of free $A$-modules concentrated in degree 0 and 1:
\ben
&&M_{\mathcal{F}}^\bullet: 0\to M_{\mathcal{F}}^2\to M_{\mathcal{F}}^1\to 0,\\
&&M_{\mathcal{F}_{0}}^\bullet: 0\to M_{\mathcal{F}_{0}}^2\to M_{\mathcal{F}_{0}}^1\to 0.
\een
Then we have for any $A$-module $M$
\ben
&&\mathrm{Hom}(\mathcal{F},\mathcal{G}\otimes_{A}M)\cong H^0(M_{\mathcal{F}}^\bullet\otimes_{A}M),\\
&&\mathrm{Hom}(p_{\mathcal{X}}^*\mathcal{F}_{0},\mathcal{G}\otimes_{A}M)\cong
H^0(M_{\mathcal{F}_{0}}^\bullet\otimes_{A}M).
\een
The morphism $\varphi:p_{\mathcal{X}}^*\mathcal{F}_{0}\to\mathcal{F}$ induces a morphism of complexes, denoted again by $\varphi:M_{\mathcal{F}}^\bullet\to M_{\mathcal{F}_{0}}^\bullet$. Denote by $A^\bullet$ the complex such that $A^0=A$ and $A^{i}=0$ for $i\neq0$, then the  morphism $\psi:p_{\mathcal{X}}^*\mathcal{F}_{0}\to\mathcal{G}$ induces a morphism $\psi:A^\bullet\to M_{\mathcal{F}_{0}}^\bullet$. Therefore, we have a morphism
$\Psi:=(\varphi,-\psi):M_{\mathcal{F}}^\bullet\oplus A^\bullet\to M_{\mathcal{F}_{0}}^\bullet$. Let $\mathrm{Cone}(\Psi)^\bullet$ be the mapping cone of $\Psi$, we have a short exact sequence
\ben
0\to M_{\mathcal{F}_{0}}^\bullet\to\mathrm{Cone}(\Psi)^\bullet\to M_{\mathcal{F}}^\bullet\oplus A^\bullet[1]\to0.
\een
Then we obtain the following long exact sequence
\ben
0\to h^{-1}(\mathrm{Cone}(\Psi)^\bullet\otimes_{A}M)\to\mathrm{Hom}(\mathcal{F},\mathcal{G}\otimes_{A}M)\oplus M\to\mathrm{Hom}(p_{\mathcal{X}}^*\mathcal{F}_{0},\mathcal{G}\otimes_{A}M)\to\cdots.
\een
For any $t\in T$ and $M=k(t)$, we have the following cartesian diagram
\ben
\xymatrix{
	h^{-1}(\mathrm{Cone}(\Psi)^\bullet\otimes_{A}k(t))\ar[d] \ar[r] &   k(t)\ar[d]^{\cdot\psi_{t}}\\
	\mathrm{Hom}(\mathcal{F}_{t},\mathcal{G}_{t})\ar[r]_-{\circ\varphi_{t}} &\mathrm{Hom}(\mathcal{F}_{0},\mathcal{G}_{t})
}
\een
By Remark \ref{Hom-set},  we have
\ben
\dim_{k(t)}\mathrm{Hom}_{\mathcal{X}\times \{t\}}((\mathcal{F}_{t},\varphi_{t}),(\mathcal{G}_{t},\psi_{t}))=\dim_{k(t)} h^{-1}(\mathrm{Cone}(\Psi)^\bullet\otimes_{A}k(t))-1+\epsilon(\psi_{t}).
\een
The proof is completed by using the fact that  $\epsilon(\psi_{t})$ is either zero  for all $t\in T$ or never zero by assumption of flatness and the function $t\to\dim	h^{-1}(\mathrm{Cone}(\Psi)^\bullet\otimes_{A}k(t))$ is upper semicontinuous by [\cite{BS}, Lemma A.1].
\end{proof}
As in [\cite{Lin18}, Section 4], we call a $\delta$-semistable pair  $\delta$-polystable if it is isomorphic to the direct sum of  $\delta$-stable pair with the same reduced Hilbert polynomial. Notice that the assumption of nondegerateness is not imposed on every summand of a nondegenerate $\delta$-polystable pair.   Using the similar argument in the proof of [\cite{HL2}, Proposition 3.3] and  Lemma \ref{semi-conti}, we have
\begin{lemma}
Two points $([a_{1}],[q_{1}])$ and $([a_{2}],[q_{2}])$ are mapped to the same point in $M^{ss}$ if and only if their corresponding $\delta$-semistable pairs $(\mathcal{F}_{1},\varphi_{1})$ and $(\mathcal{F}_{2},\varphi_{2})$ are $S$-equivalent. The orbit of a point $([a],[q])$ is closed in $\mathcal{R}^{ss}$ if and only if the corresponding pair $(\mathcal{F},\varphi)$ is $\delta$-polystable.
\end{lemma}
\begin{remark}
Since $\Theta: \mathcal{R}^{ss}\to M^{ss}$  is a  good quotient for the $\mathrm{SL}(V)$-action on $\mathcal{R}^{ss}$,
the statement that two points $([a_{1}],[q_{1}])$ and $([a_{2}],[q_{2}])$ are mapped to the same point in $M^{ss}$ is equivalent to the one that the closure of  orbits of two points
$([a_{1}],[q_{1}])$ and $([a_{2}],[q_{2}])$ in $\mathcal{R}^{ss}$ intersect. 
\end{remark}

Next, we shall show that the projective scheme $M^{ss}$ is a moduli space for the functor $\mathcal{M}^{ss}_{\mathcal{X}/k}(\mathcal{F}_{0},P,\delta)$ and the quasi-projective scheme $M^{s}$ is a fine moduli space for the functor $\mathcal{M}^{s}_{\mathcal{X}/k}(\mathcal{F}_{0},P,\delta)$.  Following the method in [\cite{BS}, Section 4.2], we first recall  the definition of good moduli space [\cite{Alper}] as follows.

\begin{definition}([\cite{Alper}, Section 2])
Let $S$ be a scheme and $(\mathrm{Sch}/S)_{\mathrm{Et}}$ a category of schemes over $S$ with the global $\acute{e}$tale topology. An algebraic space over $S$ is a sheaf of sets $X$
on $(\mathrm{Sch}/S)_{\mathrm{Et}}$ such that 

$(i)$ $\Delta_{X/S}:X\to X\times_{S}X$ is representable by schemes and quasi-compact.

$(ii)$ There exists an $\acute{e}$tale, surjective map $U\to X$, where $U$ is a scheme.\\
An Artin stack over $S$ is a stack $\mathcal{X}$ over $(\mathrm{Sch}/S)_{\mathrm{Et}}$ such that

$(i)$ $\Delta_{\mathcal{X}/S}:\mathcal{X}\to \mathcal{X}\times_{S}\mathcal{X}$ is representable, separated and quasi-compact.

$(ii)$ There exists a smooth, surjective map $X\to\mathcal{X}$, where $X$ is an algebraic space.
\end{definition}
\begin{remark}
An Artin stack $\mathcal{X}$ here is referred as an algebraic stack in [\cite{LMB}, Definition 4.1]. It is a Deligne-Mumford stack over $S$ when there exists an $\acute{e}$tale, surjective map  from an algebraic space.
\end{remark}
\begin{definition}([\cite{Alper}, Definition 3.1])
A morphism $f:\mathcal{X}\to\mathcal{Y}$ of Artin stacks over $S$ is  cohomologically affine if $f$ is quasi-compact and the functor $f_{*}:\mathrm{QCoh}(\mathcal{X})\to\mathrm{QCoh}(\mathcal{Y})$ is exact.
\end{definition}
\begin{definition}([\cite{Alper}, Definition 4.1])
Let $\mathcal{X}$ be an Artin stack and $Y$ be an algebraic space over $S$. We say that a morphism $f:\mathcal{X}\to Y$ is a good moduli space if the following conditions are satisfied:

$(i)$ $f$ is cohomologically affine.

$(ii)$ The natural morphism $\mathcal{O}_{Y}\to f_{*}\mathcal{O}_{\mathcal{X}}$ is an isomorphism.
\end{definition}
\begin{definition}([\cite{Alper}, Definiton 7.1])
We call $f:\mathcal{X}\to Y$ a tame moduli space if 

$(i)$ $f$ is a good moduli space,

$(ii)$ for all geometric points $\mathrm{Spec}\,k\to S$, the map $[\mathcal{X}(k)]\to Y(k)$ is a bijection of sets, where $[\mathcal{X}(k)]$ denotes the set of isomorphism classes of objects of $\mathcal{X}(k).$ 
\end{definition}
We introduce the following Artin stacks of finite type: 
\ben
\mathfrak{SR}^{(s)s}:=[\mathcal{R}^{s(s)}/\mathrm{SL}(V)], \;\;\;\mathfrak{R}^{(s)s}:=[\mathcal{R}^{s(s)}/\mathrm{GL}(V)],\;\;\; \mathfrak{PR}^{(s)s}:=[\mathcal{R}^{s(s)}/\mathrm{PGL}(V)],
\een
where $\mathfrak{PR}^{(s)s}$ is well-defined due to Lemma \ref{sta-group}.
Using the rigidification of a stack in [\cite{ACV}, Section 5], it is easy to show  that $\mathfrak{R}^{(s)s}$ is $\mathbb{G}_{m}$-gerbe on 
$\mathfrak{PR}^{(s)s}$ and $\mathfrak{SR}^{(s)s}$ is $\mu(V)$-gerbe on 
$\mathfrak{PR}^{(s)s}$ where $\mu(V)\subseteq\mathrm{SL}(V)$ denotes the group of the product of $\dim(V)$-roots of unity and the identity matrix. 
Since $\mathcal{O}_{\mathcal{Z}^\prime}(n_{1},n_{2})|_{\mathcal{R}^{ss}}$ is a $\mathrm{SL}(V)$-linearization on $\mathcal{R}^{ss}$, we denote by $\mathcal{O}(n_{1},n_{2})$ the corresponding line bundle on $\mathfrak{SR}^{ss}$.
Combing with Theorem \ref{GIT-quo1},  an analogue of GIT  in [\cite{Alper}, Theorem 13.6] shows  
\begin{theorem}\label{good-moduli}
There is a good moduli space $\Theta_{\mathfrak{S}}: \mathfrak{SR}^{ss}\to M^{ss}$ with   $\mathcal{O}(n_{1},n_{2})^{\otimes\hat{l}}=\Theta_{\mathfrak{S}}^*(\mathbb{M})$ and the morphism $\Theta_{\mathfrak{S}}: \mathfrak{SR}^{s}\to M^{s}$ is a tame moduli space.
\end{theorem}

As in [\cite{BS}, Section 4.2],  the morphism $\Theta_{\mathfrak{S}}: \mathfrak{SR}^{ss}\to M^{ss}$ induces the commutative diagram 
\ben
\xymatrixcolsep{5pc}\xymatrix{
	\mathfrak{SR}^{(s)s}\ar[dr]_{\Theta_{\mathfrak{S}}} \ar[r] &   \mathfrak{PR}^{(s)s}\ar[d]^{\Theta_{\mathfrak{P}}} &\mathfrak{R}^{(s)s}\ar[l]\ar[dl]^{\Theta_{\mathfrak{G}}}\\
	 &M^{(s)s} &
}
\een
and  morphisms $\Theta_{\mathfrak{G}}$ and $\Theta_{\mathfrak{P}}$ also satisfy assertions stated in Theorem \ref{good-moduli}.
Following [\cite{BS}, Section 4.2], we define $[\mathfrak{R}^{(s)s}]$ to be the contravariant functor such that for any scheme $S$ of finite type, $[\mathfrak{R}^{(s)s}](S)$ is the set of isomorphism classes of objects of $\mathfrak{R}^{(s)s}(S)$. An object in $[\mathfrak{R}^{(s)s}](S)$ is an isomorphism class $[(\overline{q}: \overline{P}\to S, \phi:\overline{P}\to\mathcal{R}^{(s)s})]$, where $\overline{q}: \overline{P}\to S$ is a $\mathrm{GL}(V)$-torsor over $S$ and $\phi$ is a $\mathrm{GL}(V)$-equivariant morphism. Here, $(\overline{q}: \overline{P}\to S, \phi:\overline{P}\to\mathcal{R}^{(s)s})$ and $(\overline{q}^\prime: \overline{P}^\prime\to S, \phi^\prime:\overline{P}^\prime\to\mathcal{R}^{(s)s})$ are called isomorphic objects if we have the following commutative diagram
\ben
\xymatrix{
	\overline{P}\ar[d]_{\overline{q}} \ar[r]^{\varsigma} &   \overline{P}^\prime\ar[d]^{\overline{q}^\prime}\\
	S\ar[r]_-{\mathrm{id}_{S}} &S
}
\een
where $\varsigma$ is an isomorphism of $\mathrm{GL}(V)$-torsors compatible with $\phi$ and $\phi^\prime$. Now, we have
\begin{theorem}\label{iso-functor}
The functor $\mathcal{M}^{(s)s}_{\mathcal{X}/k}(\mathcal{F}_{0},P,\delta)$ is isomorphic to
$[\mathfrak{R}^{(s)s}]$.
\end{theorem}
\begin{proof}
Use the similar argument in the proof of [\cite{BS}, Theorem 4.12]  with our notion of flat families of pairs in Definition \ref{moduli-functor} (see also Definition \ref{fam-pair}) and the  universal family $(\check{\mathcal{F}}_{\mathcal{X}\times\mathcal{R}^{(s)s}},\check{\varphi}_{\mathcal{X}\times\mathcal{R}^{(s)s}})$ of $\delta$-(semi)stable pairs with  
modified Hilbert polynomial $P$ parameterized by $\mathcal{R}^{(s)s}$.
\end{proof}
Since the morphism $\Theta_{\mathfrak{G}}:\mathfrak{R}^{(s)s}\to M^{(s)s}$ factors through $\mathfrak{R}^{(s)s}\to[\mathfrak{R}^{(s)s}]$, we have the morphism $\mathcal{M}^{(s)s}_{\mathcal{X}/k}(\mathcal{F}_{0},P,\delta)\to  M^{(s)s}$ by Theorem \ref{iso-functor}.
By the Yoneda Lemma, we may take the scheme $M^{(s)s}$ as a functor $\mathrm{Mor}(-,M^{(s)s})$ and the mophism $\mathcal{M}^{(s)s}_{\mathcal{X}/k}(\mathcal{F}_{0},P,\delta)\to  M^{(s)s}$ as a natural transformation $\mathcal{M}^{(s)s}_{\mathcal{X}/k}(\mathcal{F}_{0},P,\delta)\to \mathrm{Mor}(-,M^{(s)s})$. We need the following lemmas which are useful for  proving the existence of a fine moduli space for $\mathcal{M}^{s}_{\mathcal{X}/k}(\mathcal{F}_{0},P,\delta)$. First, as in [\cite{BS}, Lemma 4.5], we have
\begin{lemma}\label{sta-group}
Let $([a],[q])\in\mathcal{Z}\subseteq \mathbb{P}\times\widetilde{\mathcal{Q}}$ be a closed point corresponding to a pair $(\mathcal{F},\varphi)$ such that $F_{\mathcal{E}}(\mathcal{F})(m)$ is globally generated and $H^0(\widetilde{q}):V\to H^0(F_{\mathcal{E}}(\mathcal{F})(m))$ induced by $q:V\otimes\mathcal{E}\otimes\pi^*\mathcal{O}_{X}(-m)\to\mathcal{F}$ is an isomorphism. Then there exists a natural injective  homomorphism $\mathrm{Aut}(\mathcal{F},\varphi)\to\mathrm{GL}(V)$ whose image is precisely the stabilizer group $\mathrm{GL}(V)_{([a],[q])}$ of the point $([a],[q])$.
\end{lemma}
\begin{proof}
Notice morphisms of pairs in Definition \ref{def-pair} and use the similar argument in the proof of [\cite{HL3}, Lemma 4.3.2].
\end{proof}
Corresponding to [\cite{BS}, Proposition 4.14], we have
\begin{lemma}\label{uni-fam-inv}
Let $(\check{\mathcal{F}}_{\mathcal{X}\times\mathcal{R}^{s}},\check{\varphi}_{\mathcal{X}\times\mathcal{R}^{s}})$
be the universal family of $\delta$-semistable pairs with  
modified Hilbert polynomial $P$ parameterized by $\mathcal{R}^{s}$. Then $\check{\mathcal{F}}_{\mathcal{X}\times\mathcal{R}^{s}}$ is invariant with respect to the action of the center $\mathbb{G}_{m}$ of $\mathrm{GL}(V)$.
\end{lemma}
\begin{proof} 
Since the assumption of Lemma \ref{sta-group} is satisfied on $\mathcal{R}^{s}$ by Theorem \ref{git=pair1} and Theorem \ref{git=pair2},  the center $\mathbb{G}_{m}$  of $\mathrm{GL}(V)$ acts trivially  on $\mathcal{R}^s$. This implies that the restriction of $\eta:\mathbb{P}\times\widetilde{\mathcal{Q}}\times \mathrm{GL}(V)\to\mathbb{P}\times\widetilde{\mathcal{Q}}$ to $\mathcal{R}^{s}\times\mathbb{G}_{m}$ is a trivial action where $\eta$ is induced by $\mathrm{GL}(V)$-actions $\eta_{1}:\widetilde{\mathcal{Q}}\times \mathrm{GL}(V)\to\widetilde{\mathcal{Q}}$ and $\eta_{2}:\mathbb{P}\times\mathrm{GL}(V)\to\mathbb{P}$. Denote by $i_{\mathbb{G}_{m}}:\mathbb{G}_{m}\to \mathrm{GL}(V)$ the inclusion. Then we have the following identities:
\ben
&&p_{\mathcal{X}\times\mathbb{{P}}\times\widetilde{\mathcal{Q}},\mathcal{X}\times\widetilde{\mathcal{Q}}}\circ(\mathrm{id}_{\mathcal{X}}\times\widetilde{i}_{s})\circ p_{\mathcal{X}\times\mathcal{R}^s\times\mathbb{G}_{m},\mathcal{X}\times\mathcal{R}^s}\\
&=&(\mathrm{id}_{\mathcal{X}}\times\eta_{1})\circ
p_{\mathcal{X}\times\mathbb{{P}}\times\widetilde{\mathcal{Q}}\times\mathrm{GL}(V),\mathcal{X}\times\widetilde{\mathcal{Q}}\times\mathrm{GL}(V)}\circ(\mathrm{id}_{\mathcal{X}}\times\widetilde{i}_{s}\times i_{\mathbb{G}_{m}})\\
&=&p_{\mathcal{X}\times\widetilde{\mathcal{Q}}\times\mathrm{GL}(V),\mathcal{X}\times\widetilde{\mathcal{Q}}}\circ
p_{\mathcal{X}\times\mathbb{{P}}\times\widetilde{\mathcal{Q}}\times\mathrm{GL}(V),\mathcal{X}\times\widetilde{\mathcal{Q}}\times\mathrm{GL}(V)}\circ(\mathrm{id}_{\mathcal{X}}\times\widetilde{i}_{s}\times i_{\mathbb{G}_{m}})
\een
and 
\ben
&&p_{\mathcal{X}\times\mathbb{P}\times\widetilde{\mathcal{Q}},\mathbb{P}}\circ(\mathrm{id}_{\mathcal{X}}\times\widetilde{i}_{s})\circ p_{\mathcal{X}\times\mathcal{R}^s\times\mathbb{G}_{m},\mathcal{X}\times\mathcal{R}^s}\\
&=&\eta_{2}\circ p_{\mathbb{{P}}\times\widetilde{\mathcal{Q}}\times\mathrm{GL}(V),\mathbb{{P}}\times\mathrm{GL}(V)}\circ(\widetilde{i}_{s}\times i_{\mathbb{G}_{m}})\circ p_{\mathcal{X}\times\mathcal{R}^s\times\mathbb{G}_{m},\mathcal{R}^s\times\mathbb{G}_{m}}\\
&=&p_{\mathbb{P}\times\mathrm{GL}(V),\mathbb{P}}\circ p_{\mathbb{{P}}\times\widetilde{\mathcal{Q}}\times\mathrm{GL}(V),\mathbb{{P}}\times\mathrm{GL}(V)}\circ(\widetilde{i}_{s}\times i_{\mathbb{G}_{m}})\circ p_{\mathcal{X}\times\mathcal{R}^s\times\mathbb{G}_{m},\mathcal{R}^s\times\mathbb{G}_{m}}.
\een
Then two isomorphisms $\Lambda_{1}:(\mathrm{id}_{\mathcal{X}}\times\eta_{1})^*\widetilde{\mathcal{F}}\to p_{\mathcal{X}\times\widetilde{\mathcal{Q}}\times\mathrm{GL}(V),\mathcal{X}\times\widetilde{\mathcal{Q}}}^*\widetilde{\mathcal{F}}$  and $\Lambda_{2}:\eta_{2}^*\mathcal{O}_{\mathbb{P}}(1)\to p_{\mathbb{P}\times\mathrm{GL}(V),\mathbb{P}}^*\mathcal{O}_{\mathbb{P}}(1)$ induces the following  isomorphisms
\ben
&\widetilde{\Lambda}_{1}:&p_{\mathcal{X}\times\mathcal{R}^s\times\mathbb{G}_{m},\mathcal{X}\times\mathcal{R}^s}^*(\mathrm{id}_{\mathcal{X}}\times\widetilde{i}_{s})^*p_{\mathcal{X}\times\mathbb{{P}}\times\widetilde{\mathcal{Q}},\mathcal{X}\times\widetilde{\mathcal{Q}}}^*\widetilde{\mathcal{F}}\\
&=&(\mathrm{id}_{\mathcal{X}}\times\widetilde{i}_{s}\times i_{\mathbb{G}_{m}})^*p_{\mathcal{X}\times\mathbb{{P}}\times\widetilde{\mathcal{Q}}\times\mathrm{GL}(V),\mathcal{X}\times\widetilde{\mathcal{Q}}\times\mathrm{GL}(V)}^*(\mathrm{id}_{\mathcal{X}}\times\eta_{1})^*\widetilde{\mathcal{F}}\\
&\xrightarrow{\cong}&(\mathrm{id}_{\mathcal{X}}\times\widetilde{i}_{s}\times i_{\mathbb{G}_{m}})^*p_{\mathcal{X}\times\mathbb{{P}}\times\widetilde{\mathcal{Q}}\times\mathrm{GL}(V),\mathcal{X}\times\widetilde{\mathcal{Q}}\times\mathrm{GL}(V)}^*p_{\mathcal{X}\times\widetilde{\mathcal{Q}}\times\mathrm{GL}(V),\mathcal{X}\times\widetilde{\mathcal{Q}}}^*\widetilde{\mathcal{F}}\\
&=&p_{\mathcal{X}\times\mathcal{R}^s\times\mathbb{G}_{m},\mathcal{X}\times\mathcal{R}^s}^*(\mathrm{id}_{\mathcal{X}}\times\widetilde{i}_{s})^*p_{\mathcal{X}\times\mathbb{{P}}\times\widetilde{\mathcal{Q}},\mathcal{X}\times\widetilde{\mathcal{Q}}}^*\widetilde{\mathcal{F}}
\een
 and 
\ben
&\widetilde{\Lambda}_{2}:&p_{\mathcal{X}\times\mathcal{R}^s\times\mathbb{G}_{m},\mathcal{X}\times\mathcal{R}^s}^*(\mathrm{id}_{\mathcal{X}}\times\widetilde{i}_{s})^* p_{\mathcal{X}\times\mathbb{P}\times\widetilde{\mathcal{Q}},\mathbb{P}}^*\mathcal{O}_{\mathbb{P}}(1)\\
&=&p_{\mathcal{X}\times\mathcal{R}^s\times\mathbb{G}_{m},\mathcal{R}^s\times\mathbb{G}_{m}}^*(\widetilde{i}_{s}\times i_{\mathbb{G}_{m}})^*p_{\mathbb{{P}}\times\widetilde{\mathcal{Q}}\times\mathrm{GL}(V),\mathbb{{P}}\times\mathrm{GL}(V)}^*\eta_{2}^*\mathcal{O}_{\mathbb{P}}(1)\\
&\xrightarrow{\cong}&p_{\mathcal{X}\times\mathcal{R}^s\times\mathbb{G}_{m},\mathcal{R}^s\times\mathbb{G}_{m}}^* (\widetilde{i}_{s}\times i_{\mathbb{G}_{m}})^*p_{\mathbb{{P}}\times\widetilde{\mathcal{Q}}\times\mathrm{GL}(V),\mathbb{{P}}\times\mathrm{GL}(V)}^*p_{\mathbb{P}\times\mathrm{GL}(V),\mathbb{P}}^*\mathcal{O}_{\mathbb{P}}(1)\\
&=&p_{\mathcal{X}\times\mathcal{R}^s\times\mathbb{G}_{m},\mathcal{X}\times\mathcal{R}^s}^*(\mathrm{id}_{\mathcal{X}}\times\widetilde{i}_{s})^* p_{\mathcal{X}\times\mathbb{P}\times\widetilde{\mathcal{Q}},\mathbb{P}}^*\mathcal{O}_{\mathbb{P}}(1)
\een
such that the following diagram commute

\ben
\xymatrixcolsep{5pc}\xymatrix{
 p_{\mathcal{X}\times\mathcal{R}^s\times\mathbb{G}_{m},\mathcal{X}}^*\mathcal{F}_{0}\ar[d]_{p_{\mathcal{X}\times\mathcal{R}^s\times\mathbb{G}_{m},\mathcal{X}\times\mathcal{R}^s}^*\check{\varphi}_{\mathcal{X}\times\mathcal{R}^{s}}}   \ar@{=}[r]& p_{\mathcal{X}\times\mathcal{R}^s\times\mathbb{G}_{m},\mathcal{X}}^*\mathcal{F}_{0} \ar[d]^{p_{\mathcal{X}\times\mathcal{R}^s\times\mathbb{G}_{m},\mathcal{X}\times\mathcal{R}^s}^*\check{\varphi}_{\mathcal{X}\times\mathcal{R}^{s}}} \\
	p_{\mathcal{X}\times\mathcal{R}^s\times\mathbb{G}_{m},\mathcal{X}\times\mathcal{R}^s}^*\check{\mathcal{F}}_{\mathcal{X}\times\mathcal{R}^{s}} \ar[r]_-{\widetilde{\Lambda}_{1}\otimes\widetilde{\Lambda}_{2}}        &	p_{\mathcal{X}\times\mathcal{R}^s\times\mathbb{G}_{m},\mathcal{X}\times\mathcal{R}^s}^*\check{\mathcal{F}}_{\mathcal{X}\times\mathcal{R}^{s}}
}
\een	
where 
\ben
\check{\varphi}_{\mathcal{X}\times\mathcal{R}^{s}}: p_{\mathcal{X}\times\mathcal{R}^{s},\mathcal{X}}^*\mathcal{F}_{0}\to \check{\mathcal{F}}_{\mathcal{X}\times\mathcal{R}^{s}}:=(\mathrm{id}_{\mathcal{X}}\times\widetilde{i}_{s})^*p_{\mathcal{X}\times\mathbb{{P}}\times\widetilde{\mathcal{Q}},\mathcal{X}\times\widetilde{\mathcal{Q}}}^*\widetilde{\mathcal{F}}\otimes(\mathrm{id}_{\mathcal{X}}\times\widetilde{i}_{s})^* p_{\mathcal{X}\times\mathbb{P}\times\widetilde{\mathcal{Q}},\mathbb{P}}^*\mathcal{O}_{\mathbb{P}}(1).
\een
Using the similar argument in the proof of [\cite{BS}, Proposition 4.14], one can show that both  $\widetilde{\Lambda}_{1}$ and $\widetilde{\Lambda}_{2}$	are the identity morphisms.
This implies $\check{\mathcal{F}}_{\mathcal{X}\times\mathcal{R}^{s}}$ is invariant with respect to the action of the center $\mathbb{G}_{m}$ of $\mathrm{GL}(V)$.
\end{proof}
\begin{remark}\label{no-asump}
In the proof of [\cite{BS}, Proposition 4.14], the assumption of  irreducibility of $\mathcal{X}$ is used to assert that  their universal family of framed sheaves parameterized by $\mathcal{R}^s$ can be $\mathrm{PGL}(V)$-linearized in order to show that moduli spaces of $\delta$-stable framed sheaves are fine.  While in our setting, the universal family $(\check{\mathcal{F}}_{\mathcal{X}\times\mathcal{R}^{s}},\check{\varphi}_{\mathcal{X}\times\mathcal{R}^{s}})$ of $\delta$-stable pairs is $\mathrm{PGL}(V)$-linearized without this  assumption since $\check{\mathcal{F}}_{\mathcal{X}\times\mathcal{R}^{s}}$ is invariant with respect to the action of the center $\mathbb{G}_{m}$ of $\mathrm{GL}(V)$ by Lemma \ref{uni-fam-inv}.
\end{remark}
Now, we have
\begin{theorem}\label{main-result}
Let $\mathcal{X}$ be a projective Deligne-Mumford stack of dimension $d$ over $k$ with a moduli scheme $\pi:\mathcal{X}\to X$ and a polarization $(\mathcal{E},\mathcal{O}_{X}(1))$. Let $\mathcal{F}_{0}$ be any fixed coherent sheaf  on $\mathcal{X}$, and let $\delta$ be any given stability parameter  which is a rational polynomial with positive leading coefficient and $P$  any given polynomial of degree $\deg\,P\leq d$. Then
the projective scheme $M^{ss}:=M^{ss}_{\mathcal{X}/k}(\mathcal{F}_{0},P,\delta)$ is a moduli space for the moduli functor $\mathcal{M}^{ss}_{\mathcal{X}/k}(\mathcal{F}_{0},P,\delta)$ and  the quasi-projective scheme $M^s:=M^s_{\mathcal{X}/k}(\mathcal{F}_{0},P,\delta)$ is a fine moduli space for the moduli functor $\mathcal{M}^s_{\mathcal{X}/k}(\mathcal{F}_{0},P,\delta)$. Moreover, the Artin stack $\mathfrak{R}^s$ is a $\mathbb{G}_{m}$-gerbe over its moduli scheme $M^s$.
\end{theorem}
\begin{proof}
The proof is completed by following the similar argument in [\cite{BS}, Theorem 4.15] but with our notion of universal families $(\check{\mathcal{F}}_{\mathcal{X}\times\mathcal{R}^{(s)s}},\check{\varphi}_{\mathcal{X}\times\mathcal{R}^{(s)s}})$. The last statement is true since $\mathcal{R}^s\to M^s$ can be shown to be a $\mathrm{PGL}(V)$-torsor.
\end{proof}
\begin{remark}
One may consider the  construction of relative moduli spaces of $\delta$-(semi)stable pairs  by using [\cite{Alper,Ses}] for GIT construction in relative case and  obtaining the relative version of [\cite{Nir1}, Proposition 4.24], Lemma \ref{Gro-lem}  and  Lemma \ref{glo-sec-bound} for the boundedness results. We will investigate this issue elsewhere which is essential to the study of relative orbifold PT theory and its degeneration formula.
\end{remark}

\subsection{Variation of moduli spaces} 
For a given polynomial $P$ and a fixed coherent sheaf $\mathcal{F}_{0}$ on $\mathcal{X}$, as in [\cite{Wan}, Section 5], we will investigate the variation of  moduli spaces $M^{ss}_{\mathcal{X}/k}(\mathcal{F}_{0},P,\delta)$ when one changes the stability parameter $\delta\in\mathbb{Q}[m]$.
To indicate the dependence on $\delta$, we denote by $\mathcal{R}^{(s)s}(\delta)$  the subset of GIT-(semi)stable points $([a],[q])\in\mathcal{Z}^\prime$ corresponding to  $\delta$-(semi)stable pairs.
We generalize [\cite{Wan}, Theorem 5.5] to the case of projective Deligne-Mumford stacks as follows.
\begin{theorem}\label{Chamber}
	Let $\mathcal{X}$ be a projective Deligne-Mumford stack over $k$ with a moduli scheme $\pi:\mathcal{X}\to X$ and a polarization $(\mathcal{E},\mathcal{O}_{X}(1))$.
	There are finitely many critical values $\delta^1,\cdots,\delta^t\in\mathbb{Q}[m]$ satisfying 
	\ben
	\delta^0:=0<\delta^1<\cdots<\delta^t<\delta^{t+1}:=+\infty
	\een
	such that we have the following properties:\\
	$(i)$ For $i=0,\cdots,t$ and $\delta,\delta^\prime\in(\delta^{i},\delta^{i+1})$, one has 
	$\mathcal{R}^{(s)s}(\delta)=\mathcal{R}^{(s)s}(\delta^\prime)$.\\
	$(ii)$ For $i=0,\cdots,t$ and $\delta\in(\delta^{i},\delta^{i+1})$, we have
	\ben
	\mathcal{R}^{ss}(\delta)\subseteq\mathcal{R}^{ss}(\delta^i)\cap\mathcal{R}^{ss}(\delta^{i+1}),\;\;\;\;\mathcal{R}^s(\delta)\supseteq \mathcal{R}^s(\delta^i)\cup\mathcal{R}^s(\delta^{i+1}).
	\een
	$(iii)$ For $i=0,\cdots,t$ and $\delta\in(\delta^{i},\delta^{i+1})$, one has 
	\ben
	M^{ss}_{\mathcal{X}/k}(\mathcal{F}_{0},P,\delta)=M^{s}_{\mathcal{X}/k}(\mathcal{F}_{0},P,\delta).
	\een
	Moreover, we have the chamber structure of the stability parameter as follows:
	\ben
	\xymatrixcolsep{1pc}\xymatrix{
		 M^{0}_{\mathcal{X}/k}(\mathcal{F}_{0},P,\delta)\ar[d] \ar[rd] & \cdots &    M^{t}_{\mathcal{X}/k}(\mathcal{F}_{0},P,\delta)\ar[d] \ar[rd] & &\\
		M^{ss}_{\mathcal{X}/k}(\mathcal{F}_{0},P,\delta^0)  &M^{ss}_{\mathcal{X}/k}(\mathcal{F}_{0},P,\delta^1) \cdots& \cdots M^{ss}_{\mathcal{X}/k}(\mathcal{F}_{0},P,\delta^t) &M^{ss}_{\mathcal{X}/k}(\mathcal{F}_{0},P,\delta^{t+1})
	}
	\een
	where $M^{i}_{\mathcal{X}/k}(\mathcal{F}_{0},P,\delta):=M^{ss}_{\mathcal{X}/k}(\mathcal{F}_{0},P,\delta)$ for some $\delta\in(\delta^i,\delta^{i+1})$, $i=0,\cdots,t$.
\end{theorem}
\begin{proof}
With the notation in Remark \ref{univ-fam}, we have the following universal morphism
\ben
p_{\mathcal{X}\times\mathbb{P}\times\widetilde{\mathcal{Q}},\mathcal{X}}^*(H^0(F_{\mathcal{E}}(\mathcal{F}_{0})(m))\otimes\mathcal{E}\otimes\pi^*\mathcal{O}_{X}(-m))\otimes p_{\mathcal{X}\times\mathbb{P}\times\widetilde{\mathcal{Q}},\mathbb{P}}^*\mathcal{O}_{\mathbb{P}}\to p_{\mathcal{X}\times\mathbb{{P}}\times\widetilde{\mathcal{Q}},\mathcal{X}\times\widetilde{\mathcal{Q}}}^*\widetilde{\mathcal{F}}\otimes p_{\mathcal{X}\times\mathbb{P}\times\widetilde{\mathcal{Q}},\mathbb{P}}^*\mathcal{O}_{\mathbb{P}}(1).
\een
As in [\cite{Wan}, Lemma 5.1], one can show	that the set-theoretic family of subsheaves $\mathrm{im}\varphi$ from $\delta$-semistable pairs $(\mathcal{F},\varphi)$ for all $\delta$ is bounded. Using this result together with Grothendieck's Lemma \ref{Gro-lem} and  Kleiman criterion for stacks, and following the similar argument in the proof of  [\cite{Wan}, Proposition 5.2], one has the following result:

There exists a rational polynomial $\delta_{\max}$ of degree $(\deg\,P-1)$ such that for each $\delta>\delta_{\max}$ and each pair $(\mathcal{F},\varphi)$, we have
\ben
\mbox{ $(\mathcal{F},\varphi)$ is $\delta$-semistable}\Longleftrightarrow\mbox{$\mathcal{F}$ is pure and $\mathrm{dim}\, \mathrm{coker}\varphi<\dim \mathcal{F}$}.
\een
It is easy to verify that those $\delta$ satisfying $\deg\delta\geq\dim\mathcal{X}  \;(\mbox{or more generally } \deg\,\delta\geq\deg\,P)$ are also allowed in the inequality $\delta>\delta_{\max}$.

The proof  of [\cite{Wan}, Lemma 5.3 and Corollary 5.4]  also holds in our stacky version. We complete the proof by using  the results above, Theorem \ref{git=pair1} and Theorem \ref{git=pair2}.
\end{proof}
\begin{remark}
Notice that a critical value  is defined to be a value such that when $\delta$ crosses it, the moduli space $M^{ss}_{\mathcal{X}/k}(\mathcal{F}_{0},P,\delta)$ changes. As  in [\cite{Wan}], the maps in the chamber structure above are determined by the property $(ii)$ in Theorem \ref{Chamber} and the universal properties of the universal good quotient in Theorem \ref{GIT-quo1}. In addition, we have $\delta_{\max}\geq\delta^t$. One can choose $\delta_{\max}$ to be $\delta^t$. By the properties $(i), (iii)$ and Theorem \ref{GIT-quo1}, 
 for any two parameters $\delta^\prime, \delta^{\prime\prime}>\delta^{t}$,  one has $M^{s}_{\mathcal{X}/k}(\mathcal{F}_{0},P,\delta^\prime)=M^{ss}_{\mathcal{X}/k}(\mathcal{F}_{0},P,\delta^\prime)\cong M^{ss}_{\mathcal{X}/k}(\mathcal{F}_{0},P,\delta^{\prime\prime})=M^{s}_{\mathcal{X}/k}(\mathcal{F}_{0},P,\delta^{\prime\prime})$. This implies that as is pointed out  in the introduction of  [\cite{Lin18}],  one can choose  a stability parameter $\hat{\delta}$ of degree $(\deg\,P-1)$ with $\hat{\delta}>\delta_{\max}$ such that the construction of moduli space $M^{ss}_{\mathcal{X}/k}(\mathcal{F}_{0},P,\hat{\delta})$ implies the existence of $M^{ss}_{\mathcal{X}/k}(\mathcal{F}_{0},P,\delta)$ for any $\delta$ satisfying $\deg\,\delta\geq\deg\,P$. Furtherly, it is interesting to study the explicit wall-crossing behavior for variation of moduli spaces.
\end{remark}

\section{Deformation-obstruction theories and virtual fundamental classes}
We generalize the deformation and obstruction theory of $\delta$-stable pairs for  smooth projective varieties  in [\cite{Lin18}, Theorem 1.2] and the existence of virtual fundamental classes of moduli spaces of $\delta$-stable pairs for a smooth projective surface in [\cite{Lin18}, Theorem 1.3] to the case of smooth projective Deligne-Mumford stacks. We  will generalize the deformation and obstruction theory developed in [\cite{Inaba}] to our case as an alternative approach, which is useful for proving the existence of virtual fundamental classes as in [\cite{PT1,HT09}] for the case of dimension three. Finally, we give a definition of a stacky version of  Pandharipande-Thomas invariants.

\subsection{Deformation and obstruction theories}
It is proved in [\cite{Kre}, Section 5] that the resolution property holds for a   projective Deligne-Mumford stack $\mathcal{X}$, that is, any coherent sheaf admits a surjective morphism from a locally free coherent sheaf of finite rank. By  inductive use of this property, every coherent sheaf on $\mathcal{X}$ has a locally free resolution. It is interesting to know  whether the locally free resolution  is of finite length or not. It is shown in [\cite{BS}, Appendix B] that a smooth projective Deligne-Mumford stack is of the form $[Z/G]$ where $Z$ is a smooth quasi-projective variety and $G$ is a linear algebraic group,
and hence any coherent sheaf on a smooth projective stack  $\mathcal{X}$  admits a finite resolution by locally free sheaves of finite rank by using  [\cite{CG}, Proposition 5.1.28 and Theorem 5.1.30] and the fact that the category of coherent sheaves on $\mathcal{X}=[Z/G]$ is equivalent to the category of coherent  $G$-equivariant sheaves on $Z$ (see [\cite{BS}, Remark 2.17]).  In this section, we will consider smooth projective Deligne-Mumford stacks.

Let $\mathcal{X}$ be a smooth projective Deligne-Mumford stack over $k$ with a moduli scheme $\pi:\mathcal{X}\to X$ and a polarization $(\mathcal{E},\mathcal{O}_{X}(1))$.
Since $\mathcal{E}$ is a generating sheaf for $\mathcal{X}$, the morphism
\ben
\pi^*(\pi_{*}\mathcal{H}om_{\mathcal{O}_{\mathcal{X}}}(\mathcal{E},\mathcal{G}))\otimes_{\mathcal{O}_{\mathcal{X}}}\mathcal{E}\to\mathcal{G}
\een
is surjective for any coherent sheaf $\mathcal{G}$. As $\pi_{*}\mathcal{H}om_{\mathcal{O}_{\mathcal{X}}}(\mathcal{E},\mathcal{G})$ is a coherent sheaf on the projective scheme $X$, then we can take a positive integer $n_{1}\geq m$ large enough such that $\pi_{*}\mathcal{H}om_{\mathcal{O}_{\mathcal{X}}}(\mathcal{E},\mathcal{G})(n_{1})$ is generated by global sections where $m$ is the integer chosen  in the Section 4. Then we have the following surjective morphism
\ben
H^0(\pi_{*}\mathcal{H}om_{\mathcal{O}_{\mathcal{X}}}(\mathcal{E},\mathcal{G})(n_{1}))\otimes\mathcal{O}_{X}(-n_{1})\to\pi_{*}\mathcal{H}om_{\mathcal{O}_{\mathcal{X}}}(\mathcal{E},\mathcal{G})\to0.
\een 
Thus, we have the following surjective morphism
\ben
H^0(\pi_{*}\mathcal{H}om_{\mathcal{O}_{\mathcal{X}}}(\mathcal{E},\mathcal{G})(n_{1}))\otimes\pi^*\mathcal{O}_{X}(-n_{1})\otimes\mathcal{E}\to\pi^*(\pi_{*}\mathcal{H}om_{\mathcal{O}_{\mathcal{X}}}(\mathcal{E},\mathcal{G}))\otimes\mathcal{E}\to\mathcal{G}\to0.
\een
Denote by $W_{1}=H^0(\pi_{*}\mathcal{H}om_{\mathcal{O}_{\mathcal{X}}}(\mathcal{E},\mathcal{G})(n_{1}))$, we have a  surjection from a locally free sheaf $W_{1}\otimes\mathcal{E}\otimes\pi^*\mathcal{O}_{X}(-n_{1})\to\mathcal{G}\to0$. Let $\mathcal{G}_{1}=\ker(W_{1}\otimes\mathcal{E}\otimes\pi^*\mathcal{O}_{X}(-n_{1})\to\mathcal{G})$. Applying the same process above to $\mathcal{G}_{1}$, one can get
\ben
W_{2}\otimes\mathcal{E}\otimes\pi^*\mathcal{O}_{X}(-n_{2})\twoheadrightarrow\mathcal{G}_{1}\hookrightarrow W_{1}\otimes\mathcal{E}\otimes\pi^*\mathcal{O}_{X}(-n_{1})\twoheadrightarrow\mathcal{G}
\een
or a locally free resolution
\ben
W_{2}\otimes\mathcal{E}\otimes\pi^*\mathcal{O}_{X}(-n_{2})\to W_{1}\otimes\mathcal{E}\otimes\pi^*\mathcal{O}_{X}(-n_{1})\to\mathcal{G}\to0
\een
where $n_{1},n_{2}\geq m$ are  positive integers large enough and $W_{1}, W_{2}$ are vector spaces.
The next step is again applying the same process to $\mathcal{G}_{2}=\ker(W_{2}\otimes\mathcal{E}\otimes\pi^*\mathcal{O}_{X}(-n_{2})\to W_{1}\otimes\mathcal{E}\otimes\pi^*\mathcal{O}_{X}(-n_{1}))$. By induction and the  finite resolution property in [\cite{BS}, Lemma B.3], we  construct a finite locally free resolution of $\mathcal{G}$ (see also Lemma \ref{finite-res}):
\ben
\cdots\to\mathcal{W}_{i+1}\to\mathcal{W}_{i}\to\cdots\to\mathcal{W}_{1}\to\mathcal{G}\to0
\een
where $\mathcal{W}_{i}=W_{i}\otimes\mathcal{E}\otimes\pi^*\mathcal{O}_{X}(-n_{i})$
with $n_{i}\geq m$ large enough and $W_{i}$ is a vector space for any $i\geq1$.

Let $\mathcal{A}rt_{k}$ be the category of Artinian local $k$-algebras with residue field $k$.
For $A, A^\prime\in\mathrm{Ob}\mathcal{A}rt_{k}$ and  let the short exact sequence 
\ben
0\to I \to A^\prime\to A\to0
\een
be a small extension, that is, $\mathtt{m}_{A^\prime}I=0$. Let $\mathcal{F}_{0}$ be any fixed coherent sheaf  on $\mathcal{X}$, and let $\delta$ be any  rational polynomial with positive leading coefficient and $P$  any given polynomial of degree $\deg\,P\leq d$.   Let $[(\mathcal{F},\varphi)]$ be a point in the moduli space $M^s_{\mathcal{X}/k}(\mathcal{F}_{0},P,\delta)$. Suppose $\check{\varphi}_{A}:\mathcal{F}_{0}\otimes A\to \mathcal{F}_{A}$ over $\mathcal{X}_{A}=\mathcal{X}\times \mathrm{Spec}\,A$ is a flat extension of $(\mathcal{F},\varphi)$ where $\mathcal{F}_{A}$ is flat over $A$. Let $\mathbf{I}^{\bullet}:=\{\mathcal{F}_{0} \xrightarrow{\varphi} \mathcal{F}\}$ and  $\mathbf{I}_{A}^{\bullet}:=\{\mathcal{F}_{0}\otimes A\xrightarrow{\check{\varphi}_{A}} \mathcal{F}_{A}\}$ be the complexes concentrated in degree 0 and 1.  Since $\mathcal{X}_{A}\to\mathrm{Spec}\,A$ is a family of projective stacks by Theorem \ref{base-chag-proj}, the morphism 
\ben
(\pi\times\mathrm{id}_{\mathrm{Spec}\,A})^*\left((\pi\times\mathrm{id}_{\mathrm{Spec}\,A})_*\mathcal{H}om(p_{\mathcal{X}_{A},\mathcal{X}}^*\mathcal{E},\mathcal{F}_{A})\right)\otimes p_{\mathcal{X}_{A},\mathcal{X}}^*\mathcal{E}\to\mathcal{F}_{A}
\een
is surjective. Since any coherent sheaf on  $\mathcal{X}_{A}$ has a locally free resolution,
together with using the finite resolution property  for any fiber of $\mathcal{X}_{A}\to\mathrm{Spec}\,A$, the similar argument as above shows that
there exists a finite locally free resolution of $\mathcal{F}_{A}$: 
\ben
\cdots\to\mathcal{W}_{A}^{i-1}\to\mathcal{W}_{A}^{i}\to\cdots \to\mathcal{W}_{A}^{0}\to\mathcal{F}_{A}\to0
\een
where $\mathcal{W}_{A}^i=W^i\otimes p_{\mathcal{X}_{A},\mathcal{X}}^*\mathcal{E}\otimes p_{\mathcal{X}_{A},\mathcal{X}}^*\pi^*\mathcal{O}_{X}(-n_{i})$ with $n_{i}\geq m$ large enough and $W^i$ is a vector space for any $i\leq0$. Again, we take a finite locally free resolution of $\mathcal{F}_{0}$ as follows:
\ben
\cdots\to\mathcal{V}^{i-1}\to\mathcal{V}^{i}\to\cdots\to\mathcal{V}^{0}\to\mathcal{F}_{0}\to0
\een
where $\mathcal{V}^{i}=V^{i}\otimes\mathcal{E}\otimes\pi^*\mathcal{O}_{X}(-m_{i})$
with $m_{i}\geq m$ large enough and $V^{i}$ is a vector space for any $i\leq0$. As in [\cite{Lin18}, Section 5A], lifting $\check{\varphi}_{A}:\mathcal{F}_{0}\otimes A\to \mathcal{F}_{A}$ to a morphism of complexes $\check{\varphi}^\bullet_{A}:\mathcal{V}^{\bullet}\otimes A\to\mathcal{W}_{A}^\bullet$, we have the  commutative diagram as follows
\ben
\xymatrixcolsep{4pc}\xymatrix{
 \cdots\ar[r]^{d_{\mathcal{V}}^{-2}\otimes A} &\mathcal{V}^{-1}\otimes A\ar[d]^{\check{\varphi}_{A}^{-1}} \ar[r]^{d_{\mathcal{V}}^{-1}\otimes A} &    \mathcal{V}^0\otimes A\ar[d]^{\check{\varphi}_{A}^0} \ar[r]& \mathcal{F}_{0}\otimes A\ar[d]^{\check{\varphi}_{A}}\ar[r]& 0\\
\cdots\ar[r]^{d_{\mathcal{W}_{A}}^{-2}} &	\mathcal{W}_{A}^{-1}\ar[r]^{d_{\mathcal{W}_{A}}^{-1}} &\mathcal{W}_{A}^0 \ar[r]& \mathcal{F}_{A}\ar[r] &0.
}
\een
Notice that  $n_{i}, m_{i}\geq m$, one has
\ben
&&H^i(\mathcal{X},\mathcal{F}\otimes\mathcal{E}^\vee\otimes\pi^*\mathcal{O}_{X}(n_{j}))=H^i(X,F_{\mathcal{E}}(\mathcal{F})(n_{j}))=0,\\
&&H^i(\mathcal{X},\mathcal{F}\otimes\mathcal{E}^\vee\otimes\pi^*\mathcal{O}_{X}(m_{j}))=H^i(X,F_{\mathcal{E}}(\mathcal{F})(m_{j}))=0,
\een
for all $i>0$ and $j\leq0$. 

With these preparations, by the similar argument in  [\cite{Lin18}, Section 5], we have the following straightforward generalization of [\cite{Lin18}, Theorem 1.2].
\begin{theorem}\label{def-ob1}
Let $[(\mathcal{F},\varphi)]$ be a point in the moduli space $M^s_{\mathcal{X}/k}(\mathcal{F}_{0},P,\delta)$. Let $\check{\varphi}_{A}:\mathcal{F}_{0}\otimes_{k} A\to \mathcal{F}_{A}$ be a morphism over $\mathcal{X}_{A}=\mathcal{X}\times_{\mathrm{Spec}\,k} \mathrm{Spec}\,A$ extending  $\varphi$, where $\mathcal{F}_{A}$ is a coherent sheaf flat over $A$. Then  for a given small extension $0\to I \to A^\prime\xrightarrow{\sigma} A\to0$, there is a class
\ben
\mathrm{ob}(\check{\varphi}_{A},\sigma)\in\mathrm{Ext}^1(\mathbf{I}^\bullet,\mathcal{F}\otimes I)
\een
such that there exists a flat extension of $\check{\varphi}_{A}$ over $\mathcal{X}_{A^\prime}$ if and only if $\mathrm{ob}(\check{\varphi}_{A},\sigma)=0$. If $\mathrm{ob}(\check{\varphi}_{A},\sigma)=0$, the space of extensions  is a torsor under $\mathrm{Hom}(\mathbf{I}^\bullet,\mathcal{F}\otimes I)$.
\end{theorem}

\begin{remark}
The deformation and obstruction theory in [\cite{Lin18}, Theorem 1.2] is analogous to the one in  [\cite{She}, Theorem 4.2] (see also  [\cite{Illu}, IV 3.2.12])   for any  small  extension  $0\to I \to A^\prime\to A\to0$ when we view the 2-term complex as a mapping cone. Actually Theorem \ref{def-ob1} provides a deformation and obstruction theory for the stacky version of higher rank Pandharipande-Thomas stable pairs [\cite{She}].
\end{remark}

The deformation and obstruction theory in Theorem \ref{def-ob1} is suitable for showing the existence of virtual fundamental classes for the case of dimension two in the next subsection, but it seems difficult to deal with 3-dimensional case (see Lemma \ref{def-ob-cmg1} and Remark \ref{3-term}). Now, we take another approach to give a stacky version of  the deformation and obstruction theory in [\cite{Inaba}, Section 2]. We start with 
\begin{lemma}\label{qua-iso1}
	Suppose $B\in\mathrm{Ob}\mathcal{A}rt_{k}$, and let $\mathcal{F}^\bullet$ be a bounded complex of coherent sheaves over  $\mathcal{X}_{B}:=\mathcal{X}\times\mathrm{Spec}\,B$ where each $\mathcal{F}^j$ is flat over $B$. Then there are a complex $V^\bullet=(V^i, d_{V^\bullet}^i)$ where $V^i=L_{i}\otimes p_{\mathcal{X}_{B},\mathcal{X}}^*\mathcal{E}\otimes p_{\mathcal{X}_{B},\mathcal{X}}^*\pi^*\mathcal{O}_{X}(-m_{i})$   and a quasi-isomorphism $\phi^{\bullet}:V^\bullet\to\mathcal{F}^\bullet$ such that $(V^\bullet\to\mathcal{F}^\bullet,V^\bullet)$ satisfies the following two conditions:
	
	$(i)$  $H^c(\mathcal{X}_{B},\mathcal{F}^j\otimes p_{\mathcal{X}_{B},\mathcal{X}}^*\mathcal{E}^\vee\otimes p_{\mathcal{X}_{B},\mathcal{X}}^*\pi^*\mathcal{O}_{X}(m_{i}))=0$ for any $i, j$ and any $c>0$.
	
	$(ii)$ Let $U^\bullet:=\mathcal{F}^\bullet\oplus V^\bullet[1]$ be the mapping cone of $\phi^\bullet$ and set $W^i:=\ker(U^i\to U^{i+1})$. Then
	the map $H^0(\mathcal{X}_{B},U^{j-1}\otimes p_{\mathcal{X}_{B},\mathcal{X}}^*\mathcal{E}^\vee\otimes p_{\mathcal{X}_{B},\mathcal{X}}^*\pi^*\mathcal{O}_{X}(m_{j}))\to H^0(\mathcal{X}_{B},W^{j}\otimes p_{\mathcal{X}_{B},\mathcal{X}}^*\mathcal{E}^\vee\otimes p_{\mathcal{X}_{B},\mathcal{X}}^*\pi^*\mathcal{O}_{X}(m_{j}))$ is surjective for any $j$ and $H^c(\mathcal{X}_{B},W^{j}\otimes p_{\mathcal{X}_{B},\mathcal{X}}^*\mathcal{E}^\vee\otimes p_{\mathcal{X}_{B},\mathcal{X}}^*\pi^*\mathcal{O}_{X}(m_{i}))=0$ for $i\leq j$ and $c>0$.\\
Here, $L_{i}$ is a vector space and $m_{i}$ is a sufficiently large integer  for any $i$.
\end{lemma}
\begin{proof}
We follow the similar argument in the poofs of [\cite{Inaba}, Proposition 1.1 and Remark 2.2]. 
Choose  an integer $l$ such that $\mathcal{F}^j=0$ if $j>l$. By Theorem \ref{base-chag-proj}, the sheaf $p_{\mathcal{X}_{B},\mathcal{X}}^*\mathcal{E}$ is a generating sheaf for $\mathcal{X}_{B}$. Then we have the surjective morphism
	\ben
	(\pi\times\mathrm{id}_{\mathrm{Spec}B})^*(F_{p_{\mathcal{X}_{B},\mathcal{X}}^*\mathcal{E}}(\mathcal{F}^l))\otimes p_{\mathcal{X}_{B},\mathcal{X}}^*\mathcal{E}\to\mathcal{F}^l.
	\een
Choose a positive integer $m_{l}$ sufficiently large such that $F_{p_{\mathcal{X}_{B},\mathcal{X}}^*\mathcal{E}}(\mathcal{F}^l)\otimes p_{X_{B},X}^*\mathcal{O}_{X}(m_{l})$ is generated by global sections, and then we have
	the following surjective morphism
	\ben
	V^l:=H^0(X_{B},F_{p_{\mathcal{X}_{B},\mathcal{X}}^*\mathcal{E}}(\mathcal{F}^l)\otimes p_{X_{B},X}^*\mathcal{O}_{X}(m_{l}))\otimes p_{\mathcal{X}_{B},\mathcal{X}}^*\mathcal{E}\otimes p_{\mathcal{X}_{B},\mathcal{X}}^*\pi^*\mathcal{O}_{X}(-m_{l})\to\mathcal{F}^{l}
	\een
	where  $X_{B}:=X\times\mathrm{Spec}\,B$. Set $L_{l}:=H^0(X_{B},F_{p_{\mathcal{X}_{B},\mathcal{X}}^*\mathcal{E}}(\mathcal{F}^l)\otimes p_{X_{B},X}^*\mathcal{O}_{X}(m_{l}))$, $\widetilde{W}^l:=\mathcal{F}^l$ and $\widetilde{W}^{l-1}:=\ker(\mathcal{F}^{l-1}\oplus V^l \to\widetilde{W}^l)$. Inductively define $V^{i}:=L_{i}\otimes p_{\mathcal{X}_{B},\mathcal{X}}^*\mathcal{E}\otimes p_{\mathcal{X}_{B},\mathcal{X}}^*\pi^*\mathcal{O}_{X}(-m_{i})$ and $\widetilde{W}^{i-1}:=\ker(\mathcal{F}^{i-1}\oplus V^i\to\widetilde{W}^{i})$ where 
	$
	L_{i}:=H^0(X_{B},F_{p_{\mathcal{X}_{B},\mathcal{X}}^*\mathcal{E}}(\widetilde{W}^i)\otimes p_{X_{B},X}^*\mathcal{O}_{X}(m_{i}))
	$ and $m_{i}$ is a positive integer sufficiently large. Notice that for any coherent sheaf $\mathcal{G}$ on $\mathcal{X}$ and any $i,j$, we have
	\ben
	H^j(\mathcal{X}_{B},\mathcal{G}\otimes p_{\mathcal{X}_{B},\mathcal{X}}^*\mathcal{E}^\vee\otimes p_{\mathcal{X}_{B},\mathcal{X}}^*\pi^*\mathcal{O}_{X}(m_{i}))=H^j(X_{B},F_{p_{\mathcal{X}_{B},\mathcal{X}}^*\mathcal{E}}(\mathcal{G})\otimes p_{X_{B},X}^*\mathcal{O}_{X}(m_{i})).
	\een
	Actually one can  inductively choose $m_{i}$ sufficiently large to satisfy the condition $(i)$  and
	\ben
	H^c(\mathcal{X}_{B},\widetilde{W}^{j}\otimes p_{\mathcal{X}_{B},\mathcal{X}}^*\mathcal{E}^\vee\otimes p_{\mathcal{X}_{B},\mathcal{X}}^*\pi^*\mathcal{O}_{X}(m_{i}))=0, \;\;\;\mbox{for $i\leq j$ and $c>0$}.
	\een
	The similar argument in the proof of [\cite{Inaba}, Proposition 1.1] shows  one can define the complex $V^\bullet$  with $V^i=0$ for $i>l$ and $d_{V^\bullet}^i:V^i\to \widetilde{W}^i\to V^{i+1}$ such that there is a quasi-isomorphism $\phi^{\bullet}:V^\bullet\to\mathcal{F}^\bullet$. It is easy to verify that 
	$\widetilde{W}^i=W^i:=\ker(U^i\to U^{i+1})$. Then we have the surjection $U^{j-1}=\mathcal{F}^{j-1}\oplus V^j\to W^j$ for any $j$ which implies the surjectivity in the condition $(ii)$ by the definition of $V^\bullet$.
\end{proof}

\begin{remark}\label{qua-iso2}
	Using the similar argument in the proof of [\cite{Inaba}, Lemma 2.1],
	the condition $(i)$ in Lemma \ref{qua-iso1} corresponding to the condition $(*)$ in [\cite{Inaba}, Lemma 2.1]  implies that there are bijective canonical homorphisms for any $c>0$ as follows
	\ben
	H^c(\mathrm{Hom}^\bullet(V^\bullet,\mathcal{F}^\bullet))\to\mathrm{Ext}^c_{\mathcal{X}_{B}}(V^\bullet,\mathcal{F}^\bullet).
	\een
	And the  condition $(ii)$ in Lemma \ref{qua-iso1} corresponding to  the condition $(L_{0})$ in [\cite{Inaba}, Lemma 2.1]  shows that the canonical homomorphisms
	\ben
	H^c(\mathrm{Hom}^\bullet(V^\bullet,V^\bullet))\to H^c(\mathrm{Hom}^\bullet(V^\bullet,\mathcal{F}^\bullet))
	\een
	are surjective for $c\geq0$ and bijective for $c>0$.
\end{remark}
Then, we have
\begin{theorem}\label{def-ob2}
	Let $[(\mathcal{F},\varphi)]$ be a point in the moduli space $M^s_{\mathcal{X}/k}(\mathcal{F}_{0},P,\delta)$. Let $\check{\varphi}_{A}:\mathcal{F}_{0}\otimes_{k} A\to \mathcal{F}_{A}$ be a morphism over $\mathcal{X}_{A}=\mathcal{X}\times_{\mathrm{Spec}\,k} \mathrm{Spec}\,A$ extending  $\varphi$, where $\mathcal{F}_{A}$ is a coherent sheaf flat over $A$. Let $\mathbf{I}^{\bullet}=\{\mathcal{F}_{0} \xrightarrow{\varphi} \mathcal{F}\}$ and  $\mathbf{I}_{A}^{\bullet}=\{\mathcal{F}_{0}\otimes_{k} A\xrightarrow{\check{\varphi}_{A}} \mathcal{F}_{A}\}$ be the complexes concentrated in degree 0 and 1. Then  for a given small extension $0\to I \to A^\prime\xrightarrow{\sigma} A\to0$, there is an element
	\ben
	\omega(\mathbf{I}_{A}^{\bullet})\in\mathrm{Ext}^2(\mathbf{I}^\bullet,\mathbf{I}^\bullet\otimes I)
	\een
	such that there exists a flat extension of $\check{\varphi}_{A}$ over $\mathcal{X}_{A^\prime}$ if and only if $\omega(\mathbf{I}_{A}^{\bullet})=0$. If $\omega(\mathbf{I}_{A}^{\bullet})=0$, then  the space of extensions form a torsor under $\mathrm{Ext}^1(\mathbf{I}^\bullet,\mathbf{I}^\bullet\otimes I)$.
\end{theorem}
\begin{proof}
	Notice that  the bounded complex we  consider here is a  two-term complex  $\mathbf{I}_{A}^{\bullet}$ with $\mathbf{I}^{\bullet}=\mathbf{I}_{A}^{\bullet}\otimes k$ from a stable pair $(\mathcal{F}_{A}, \check{\varphi}_{A})$ and the isomorphic relations in Definition \ref{fam-pair} is contained in the equivalent relations in [\cite{Inaba}, Definition 0.1].  Following the similar argument in the proof of [\cite{Inaba}, Proposition 2.3]   and using  Lemma \ref{qua-iso1} and Remark \ref{qua-iso2} to show that the space of extensions is pseudo-torsor under $\mathrm{Ext}^1(\mathbf{I}^\bullet,\mathbf{I}^\bullet\otimes I)$. The proof is completed by  using the similar argument in the proof of [\cite{HT09}, Corollary 3.4] or [\cite{Lieb}, Corollary 3.2.12] to show it is actually a torsor under $\mathrm{Ext}^1(\mathbf{I}^\bullet,\mathbf{I}^\bullet\otimes I)$.
\end{proof}

Actually, the above result can be generalized to the case of a square zero extension as follows.

\begin{theorem}\label{def-ob3}
	Let $[(\mathcal{F},\varphi)]$ be a point in the moduli space $M^s_{\mathcal{X}/k}(\mathcal{F}_{0},P,\delta)$. Let $\check{\varphi}_{A}:\mathcal{F}_{0}\otimes_{k} A\to \mathcal{F}_{A}$ be a morphism over $\mathcal{X}_{A}=\mathcal{X}\times_{\mathrm{Spec}\,k} \mathrm{Spec}\,A$ extending  $\varphi$, where $\mathcal{F}_{A}$ is a coherent sheaf flat over $A$. Let  $\mathbf{I}_{A}^{\bullet}=\{\mathcal{F}_{0}\otimes_{k} A\xrightarrow{\check{\varphi}_{A}} \mathcal{F}_{A}\}$ be the complex concentrated on degree 0 and 1. Then  for a given square zero extension $0\to I \to A^\prime\to A\to0$, there is an element
	\ben
	\omega(\mathbf{I}_{A}^{\bullet})\in\mathrm{Ext}^2(\mathbf{I}^\bullet_{A},\mathbf{I}^\bullet_{A}\otimes I)
	\een
	such that there exists a flat extension of $\check{\varphi}_{A}$ over $\mathcal{X}_{A^\prime}$ if and only if $\omega(\mathbf{I}_{A}^{\bullet})=0$. If $\omega(\mathbf{I}_{A}^{\bullet})=0$, then  the space of extensions form a torsor under $\mathrm{Ext}^1(\mathbf{I}^\bullet_{A},\mathbf{I}^\bullet_{A}\otimes I)$.
\end{theorem}
\begin{proof}
Notice that since $I^2=0$,  we have a quasi-isomorphism $V^\bullet\to\mathbf{I}_{A}^{\bullet}$ such that $(V^\bullet\otimes I\to\mathbf{I}_{A}^{\bullet}\otimes I, V^\bullet)$ satisfies  conditions $(i)$ and $(ii)$  in Lemma \ref{qua-iso1} (with $B$ replaced by $A$), which will be used to replace the case of the restriction to the closed fiber in the proof of [\cite{Inaba}, Proposition 2.3] in the following argument. By Remark \ref{qua-iso2}, we have 
$H^i(\mathrm{Hom}^\bullet(V^\bullet, V^\bullet\otimes I))\cong \mathrm{Ext}^i(\mathbf{I}^\bullet_{A},\mathbf{I}^\bullet_{A}\otimes I)$ for $i\geq1$. Let 
$\tilde{d}^i_{V^\bullet}: L_{i}\otimes p_{\mathcal{X}_{A^\prime},\mathcal{X}}^*\mathcal{E}\otimes p_{\mathcal{X}_{A^\prime},\mathcal{X}}^*\pi^*\mathcal{O}_{X}(-m_{i})\to L_{i+1}\otimes p_{\mathcal{X}_{A^\prime},\mathcal{X}}^*\mathcal{E}\otimes p_{\mathcal{X}_{A^\prime},\mathcal{X}}^*\pi^*\mathcal{O}_{X}(-m_{i+1})$ be a lift of morphism $d^i_{V^\bullet}: L_{i}\otimes p_{\mathcal{X}_{A},\mathcal{X}}^*\mathcal{E}\otimes p_{\mathcal{X}_{A},\mathcal{X}}^*\pi^*\mathcal{O}_{X}(-m_{i})\to L_{i+1}\otimes p_{\mathcal{X}_{A},\mathcal{X}}^*\mathcal{E}\otimes p_{\mathcal{X}_{A},\mathcal{X}}^*\pi^*\mathcal{O}_{X}(-m_{i+1})$. Since $V^\bullet=(V^\bullet,d^i_{V^\bullet})$ is a complex and $I^2=0$, the image of $\varrho^i:=\tilde{d}^{i+1}_{V^\bullet}\circ\tilde{d}^i_{V^\bullet}$ is in $(L_{i+2}\otimes p_{\mathcal{X}_{A^\prime},\mathcal{X}}^*\mathcal{E}\otimes p_{\mathcal{X}_{A^\prime},\mathcal{X}}^*\pi^*\mathcal{O}_{X}(-m_{i+2}))\otimes I$. This yields an element $\omega(\mathbf{I}_{A}^{\bullet})\in H^2(\mathrm{Hom}^\bullet(V^\bullet, V^\bullet\otimes I))\cong \mathrm{Ext}^2(\mathbf{I}^\bullet_{A},\mathbf{I}^\bullet_{A}\otimes I)$. Again, the proof is completed by following the similar argument in the proof of [\cite{Inaba}, Proposition 2.3] for the pseudo-torsor result and then using the one of [\cite{HT09}, Corollary 3.4] or [\cite{Lieb}, Corollary 3.2.12] for the torsor result. 
\end{proof}

\begin{remark}\label{def-ob-compare}
The first order or infinitesimal deformation theory of the complex $\mathbf{I}^\bullet$ is governed by $\mathrm{Hom}(\mathbf{I}^\bullet,\mathcal{F})$ and $\mathrm{Ext}^1(\mathbf{I}^\bullet,\mathcal{F})$ by Theorem \ref{def-ob1},
 or
$\mathrm{Ext}^1(\mathbf{I}^\bullet,\mathbf{I}^\bullet)$ and $\mathrm{Ext}^2(\mathbf{I}^\bullet,\mathbf{I}^\bullet)$ by Theorem \ref{def-ob2}.
As in [\cite{PT1}, Section 2.1], there is a map for $i=0,1$
\ben
\mathrm{Ext}^i(\mathbf{I}^\bullet,\mathcal{F})\to \mathrm{Ext}^{i+1}(\mathbf{I}^\bullet,\mathbf{I}^\bullet)
\een
obtained by applying the functor $\mathrm{Hom}(\mathbf{I}^\bullet, \cdot)$ to the following distinguished triangle
\ben
\mathcal{F}[-1]\to \mathbf{I}^\bullet\to \mathcal{F}_{0}\to\mathcal{F}.
\een
\end{remark}

\subsection{Virtual fundamental classes}
In this subsection,  assume that  $\mathcal{F}_{0}$ is torsion free and polynomials $\delta$, $P$   satisfy $\deg\,\delta\geq\deg \,P=1$.  Assume that $k=\mathbb{C}$. In order to prove the existence of virtual fundamental classes, it suffices to provide a perfect obstruction theory in the sense of  [\cite{BF,LT}]. We will consider some cases of dimension two and three. We begin with the following result.

\begin{lemma}\label{def-ob-cmg}
Let $\mathcal{X}$ be a smooth projective Deligne-Mumford stack of dimension 2 over $\mathbb{C}$. Let $[(\mathcal{F},\varphi)]$ be a point in the moduli space $M^s_{\mathcal{X}/\mathbb{C}}(\mathcal{F}_{0},P,\delta)$ and  $\mathbf{I}^{\bullet}=\{\mathcal{F}_{0} \xrightarrow{\varphi} \mathcal{F}\}$  be the complex concentrated in degree 0 and 1. Then we have
\ben
\mathrm{Ext}^i(\mathbf{I}^{\bullet},\mathcal{F})=0,\;\;\;\; \mbox{if } i\neq0,1.
\een
\end{lemma}
\begin{proof}
Follow the similar argument in the proof of [\cite{Lin18}, Lemma 6.1] and apply Serre duality in [\cite{BS}, Theorem B.7].
\end{proof}
Let  $\hat{M}^s:=M^s_{\mathcal{X}/\mathbb{C}}(\mathcal{F}_{0},P,\delta)$ be the moduli space of $\delta$-stable pairs. By Theorem \ref{main-result}, there is a universal $\delta$-stable pair which determines a universal complex
\ben
\mathbb{I}^\bullet=\{\mathcal{O}_{\mathcal{X}\times \hat{M}^s}\to\mathbb{F}\}\in\mathrm{D}^b(\mathcal{X}\times \hat{M}^s)
\een
where $\mathbb{F}$ is flat over $\hat{M}^s$. Let $\pi_{\hat{M}^s}: \mathcal{X}\times \hat{M}^s\to\hat{M}^s$ and $\hat{\pi}_{\mathcal{X}}: \mathcal{X}\times \hat{M}^s\to\mathcal{X}$ be the natural projections. As in [\cite{Lin18}, Section 6], due to Theorem \ref{def-ob1}, we will consider computing the following complex 
\ben
R\pi_{\hat{M}^s*}R\mathcal{H}om(\mathbb{I}^\bullet,\mathbb{F})
\een
to obtain the deformation sheaf $\mathcal{D}ef$ and the obstruction sheaf $\mathcal{O}bs$. To resolve $\mathbb{I}^\bullet$ by a finite complex of  locally free sheaves, we need the following generalization of [\cite{BS}, Lemma B.3].
\begin{lemma}\label{finite-res}
	Let $\mathcal{X}$ be a smooth projective Deligne-Mumford stack of dimension d over $\mathbb{C}$ with a moduli scheme $\pi:\mathcal{X}\to X$ and a polarization $(\mathcal{E},\mathcal{O}_{X}(1))$. For any bounded complex  $\mathcal{N}^\bullet$ of coherent sheaves on $\mathcal{X}$, there  is a bounded complex $W^\bullet$ of locally free sheaves of finite rank, which is of the form $W^i=W_{i}\otimes\mathcal{E}\otimes\pi^*\mathcal{O}_{X}(-m_{i})$, such that $W^\bullet\to\mathcal{N}^\bullet$ is a quasi-isomorphism where $W_{i}$ is a vector space and $m_{i}$ is an integer  sufficiently large for any $i\in\mathbb{Z}$.
\end{lemma}
\begin{proof}
	We combine the argument in the proof of  [\cite{Inaba}, Proposition 1.1] with the finite resolution property in [\cite{BS}, Lemma B.3]. Let $l_{1}$ and $l_{2}$ be integers such that $\mathcal{N}^i=0$ for $i>l_{1}$ and $i<l_{2}$ where $l_{1}\geq l_{2}$. Choose an integer $m_{l_{1}}$ sufficiently large such that  $\widehat{W}^{l_{1}}:=\widehat{W}_{l_{1}}\otimes\mathcal{E}\otimes\pi^*\mathcal{O}_{X}(-m_{l_{1}})\to\mathcal{N}^{l_{1}}$ is surjective where $\widehat{W}_{l_{1}}:=H^0(X,\pi_{*}\mathcal{H}om_{\mathcal{O}_{\mathcal{X}}}(\mathcal{E},\mathcal{N}^{l_{1}})(m_{l_{1}}))$. Set $K^{l_{1}}:=\mathcal{N}^{l_{1}}$ and  $K^{l_{1}-1}:=\ker(\mathcal{N}^{l_{1}-1}\oplus \widehat{W}^{l_{1}}\to\mathcal{N}^{l_{1}})$, then we have the following quasi-isomorphism of complexes
	\ben
	\{\mathcal{N}^{l_{2}}\to\cdots\to \mathcal{N}^{l_{1}-1}\to K^{l_{1}-1}\to \widehat{W}^{l_{1}}\}\xrightarrow{\sim}\{\mathcal{N}^{l_{2}}\to\cdots\to \mathcal{N}^{l_{1}-1}\to \mathcal{N}^{l_{1}-1}\to \mathcal{N}^{l_{1}}\}
	\een
	Suppose that $K^i$ and $\widehat{W}^{i+1}$ are defined  for $i\geq c$. For $K^c$, we choose a sufficiently large integer $m_{c}$ such that $\widehat{W}^{c}:=\widehat{W}_{c}\otimes\mathcal{E}\otimes\pi^*\mathcal{O}_{X}(-m_{c})\to K^{c}$ is surjective where $\widehat{W}_{c}:=H^0(X,\pi_{*}\mathcal{H}om_{\mathcal{O}_{\mathcal{X}}}(\mathcal{E},K^c)(m_{c}))$. Set $K^{c-1}:=\ker(\mathcal{N}^{c-1}\oplus \widehat{W}^{c}\to K^c)$. By induction, we have the complex $\widehat{W}^\bullet=(\widehat{W}^i,d_{\widehat{W}^\bullet}^i)$ defined by $\widehat{W}^i=\widehat{W}_{i}\otimes\mathcal{E}\otimes\pi^*\mathcal{O}_{X}(-m_{i})$ for $i\leq l_{1}$ and $\widehat{W}^i=0$ for $i> l_{1}$ with
	$d_{\widehat{W}^\bullet}^i: \widehat{W}^{i}\to K^i\to \widehat{W}^{i+1}$. Then we have a quasi-isomorphism $\widehat{W}^\bullet\to\mathcal{N}^\bullet$. Since  $\mathcal{N}^i=0$ for $i<l_{2}$, we have $K^{i}=\ker(\widehat{W}^{i+1}\to K^{i+1})$ and then there is a locally free resolution for $K^{l_{2}-1}$ as follows
	\ben
	\cdots\to \widehat{W}^{l_{2}-2}\to \widehat{W}^{l_{2}-1}\to K^{l_{2}-1}\to0.
	\een
	Since $K^{l_{2}-1}$ is a coherent sheaf on the smooth projective Deligne-Mumford stack $\mathcal{X}$, 
	then there exists an integer $l_{3}\leq l_{2}-1$ such that $K^{l_{3}-1}:=\ker(\widehat{W}^{l_{3}}\xrightarrow{d_{\widehat{W}^\bullet}^{l_{3}}} \widehat{W}^{l_{3}+1})$ is locally free of finite rank by [\cite{BS}, Lemma B.3] or  [\cite{CG},  Theorem 5.1.30]. Then we have a quasi-isomorphism
	\ben
	\{\cdots0\to K^{l_{3}-1}\to \widehat{W}^{l_{3}}\to\cdots\to \widehat{W}^{l_{1}}\to0\cdots\}\xrightarrow{\sim}\mathcal{N}^\bullet
	\een
	Since the morphism $K^{l_{3}-1}\otimes\mathcal{E}^\vee\otimes\pi^*\mathcal{O}_{X}(m_{l_{3}})\to\widehat{W}_{l_{3}}\otimes\mathcal{O}_{\mathcal{X}}$ of locally free sheaves  is injective, one can choose a vector space $W_{l_{3}-1}$ such that $K^{l_{3}-1}\cong W_{l_{3}-1}\otimes\mathcal{E}\otimes\pi^*\mathcal{O}_{X}(-m_{l_{3}})$. The proof is  completed by setting
	\ben
	W^{l_{3}-1}:=W_{l_{3}-1}\otimes\mathcal{E}\otimes\pi^*\mathcal{O}_{X}(-m_{l_{3}});\;\; W^i:=\widehat{W}^i, \;\;\mbox{if } \;l_{3}\leq i\leq l_{1};\;\;W^i:=0, \;\mbox{otherwise}.
    \een
\end{proof}
The locally free resolution for any bounded complex of coherent sheaves in Lemma \ref{finite-res} is called very negative if each $m_{i}$ is chosen very large. We show the following results generalizing [\cite{BS}, Lemmas B.4, B.5 and Proposition B.6] and Serre duality in [\cite{BS}, Theorem B.7], which are useful in our arguments.
\begin{lemma}\label{hom-alg}
	Let $\mathcal{X}$ be a smooth projective Deligne-Mumford stack of dimension d over $\mathbb{C}$. Let $\mathcal{M}^\bullet$, $\mathcal{N}^\bullet$ and $\mathcal{H}^\bullet$ be bounded complexes of coherent sheaves on $\mathcal{X}$. Then we have the following functorial isomorphisms
	\ben
	&&R\mathcal{H}om^\bullet_{\mathcal{X}}(\mathcal{M}^\bullet,\mathcal{N}^\bullet)\mathop{\otimes}\limits^{\mathbf{L}}\mathcal{H}^\bullet\cong R\mathcal{H}om^\bullet_{\mathcal{X}}(\mathcal{M}^\bullet,\mathcal{N}^\bullet\mathop{\otimes}\limits^{\mathbf{L}}\mathcal{H}^\bullet) \\
	&&R\mathcal{H}om^\bullet_{\mathcal{X}}(\mathcal{M}^\bullet\mathop{\otimes}\limits^{\mathbf{L}}\mathcal{H}^\bullet,\mathcal{N}^\bullet)\cong R\mathcal{H}om^\bullet_{\mathcal{X}}(\mathcal{M}^\bullet,R\mathcal{H}om^\bullet_{\mathcal{X}}(\mathcal{H}^\bullet,\mathcal{N}^\bullet))\\
	&&\mathrm{Hom}_{\mathrm{D}^b(\mathcal{X})}(\mathcal{M}^\bullet\mathop{\otimes}\limits^{\mathbf{L}}(\mathcal{H}^{\bullet})^\vee,\mathcal{N}^\bullet)\cong\mathrm{Hom}_{\mathrm{D}^b(\mathcal{X})}(\mathcal{M}^\bullet,\mathcal{N}^\bullet\otimes\mathcal{H}^\bullet)\\
	&&\mathrm{Hom}_{\mathrm{D}^b(\mathcal{X})}(\mathcal{M}^\bullet\mathop{\otimes}\limits^{\mathbf{L}}\mathcal{H}^\bullet,\mathcal{N}^\bullet)\cong\mathrm{Hom}_{\mathrm{D}^b(\mathcal{X})}(\mathcal{M}^\bullet,\mathcal{N}^\bullet\otimes(\mathcal{H}^{\bullet})^\vee)
	\een
	where $(\mathcal{H}^{\bullet})^\vee=R\mathcal{H}om^\bullet_{\mathcal{X}}(\mathcal{H}^\bullet,\mathcal{O}_{\mathcal{X}})$ satisfying $(\mathcal{H}^{\bullet})^{\vee\vee}\cong\mathcal{H}^{\bullet}$.
\end{lemma}
\begin{proof}
	 The proof follows from Lemma \ref{finite-res} and the techniques used in the proofs of [\cite{BBR}, Propositions A.86, A.87, A.88].
\end{proof}

\begin{lemma}\label{Serre-duality}
	Let $p:\mathcal{X}\to\mathrm{Spec}\,\mathbb{C}$ be a smooth projective Delinge-Mumford stack of dimension d. Let $\mathcal{M}^\bullet$ and $\mathcal{N}^\bullet$ be bounded complexes of coherent sheaves on $\mathcal{X}$. Then we have 
	\ben
	\mathrm{Ext}^i(\mathcal{M}^\bullet,\mathcal{N}^\bullet)\cong\mathrm{Ext}^{d-i}(\mathcal{N}^\bullet,\mathcal{M}^\bullet\otimes\omega_{\mathcal{X}})^\vee.
	\een
	where $\omega_{\mathcal{X}}$ is the canonical line bundle of $\mathcal{X}$.
\end{lemma}
\begin{proof}
	It follows from the similar argument in the proof of  [\cite{BS}, Theorem B.7] by using Serre duality for Deligne-Mumford stacks in [\cite{Nir1}, Corollary 2.10 and Theorem 2.22] and Lemma \ref{hom-alg}.
\end{proof}

Since $ \hat{M}^s=M^{ss}_{\mathcal{X}/\mathbb{C}}(\mathcal{F}_{0},P,\delta)$ is a projective scheme by Theorem \ref{main-result} and Lemma \ref{no-str-semi}, one can  resolve   $\mathbb{F}$ and  $\mathbb{I}^\bullet$ by   finite complexes $P^\bullet$  and  $J^\bullet$ of locally free sheaves respectively  by Lemma \ref{finite-res}, and hence  $G^\bullet:=R\mathcal{H}om(R\mathcal{H}om(\mathbb{I}^\bullet,\mathbb{F}),\mathcal{O}_{\mathcal{X}\times \hat{M}^s})$  is a bounded complex of coherent sheaves. Again we take a  finite complex $Q^\bullet$ of very negative locally free sheaves resolving $G^\bullet$  by Lemma \ref{finite-res}, then by Lemma \ref{hom-alg} we have
\ben
R\pi_{\hat{M}^s*}R\mathcal{H}om(\mathbb{I}^\bullet,\mathbb{F})\cong R\pi_{\hat{M}^s*}(Q^\bullet)^\vee\cong \pi_{\hat{M}^s*}(Q^\bullet)^\vee
\een
which is a finite complex of locally free sheaves since $Q^\bullet$ is very negative. Denote this complex by $D^\bullet$. Together with Lemma \ref{def-ob-cmg}, following the similar argument in [\cite{Lin18}, Section 6], one has a short exact sequence on $\hat{M}^s$
\ben
0\to\mathcal{D}ef\to \widetilde{D^0}\to \widetilde{D^1}\to\mathcal{O}bs\to0.
\een
where $\widetilde{D^0}$ and $\widetilde{D^1}$ are locally free sheaves. This implies  the obstruction theory is perfect in the sense of [\cite{LT,BF}]. Then we have the following stacky version of [\cite{Lin18}, Theorem 1.3].
\begin{theorem}\label{vir-exi-2}
Let $\mathcal{X}$ be a smooth projective Deligne-Mumford stack of dimension 2 over $\mathbb{C}$. Let $\mathcal{F}_{0}$ be a torsion free sheaf and $\deg\,\delta\geq\deg \,P=1$. Then the moduli space $M^s_{\mathcal{X}/\mathbb{C}}(\mathcal{F}_{0},P,\delta)$ of $\delta$-stable pairs admits a virtual fundamental class.
\end{theorem}

Next, we will concentrate on investigating the existence of virtual fundamental classes for some special  case of dimension three.  We start with the following

\begin{lemma}\label{def-ob-cmg1}
Let $\mathcal{X}$ be a smooth projective Deligne-Mumford stack of dimension 3 over $\mathbb{C}$. Let $[(\mathcal{F},\varphi)]$ be a point in the moduli space $M^s_{\mathcal{X}/\mathbb{C}}(\mathcal{F}_{0},P,\delta)$ and  $\mathbf{I}^{\bullet}=\{\mathcal{F}_{0} \xrightarrow{\varphi} \mathcal{F}\}$  be the complex concentrated in degree 0 and 1. Then we have
\ben
\mathrm{Ext}^i(\mathbf{I}^{\bullet},\mathcal{F})=0,\;\;\;\; \mbox{if } i\neq0,1,2.
\een
\end{lemma}
\begin{proof}
Follow the similar argument in the proof of  Lemma \ref{def-ob-cmg}.	
\end{proof}
\begin{remark}\label{3-term}
The Ext group $\mathrm{Ext}^2(\mathbf{I}^{\bullet},\mathcal{F})$ fits into the following short exact sequence	
\ben
\cdots\to\mathrm{Ext}^1(\ker\varphi,\mathcal{F})\to\mathrm{Ext}^3(\mathrm{coker}\varphi,\mathcal{F})\to\mathrm{Ext}^2(\mathbf{I}^\bullet,\mathcal{F})\to\mathrm{Ext}^2(\ker\varphi,\mathcal{F})\to\mathrm{Ext}^4(\mathrm{coker}\varphi,\mathcal{F})=0.
\een
It seems difficult to prove the vanishness of $\mathrm{Ext}^2(\mathbf{I}^{\bullet},\mathcal{F})$ even when $\mathcal{F}_{0}=\mathcal{O}_{\mathcal{X}}$.
\end{remark}

Due to Remark \ref{3-term},  in order to define a perfect (two-term) obstruction theory for the moduli space of $\delta$-stable pairs,  we will take the  deformation and obstruction theories obtained in Theorem \ref{def-ob2} and Theorem \ref{def-ob3} as an alternative approach.  
Assume in addition  $\mathcal{F}_{0}=\mathcal{O}_{\mathcal{X}}$. In this case, the moduli space of $\delta$-stable pairs parametrizes the stacky version of PT stable pairs by Remark \ref{stacky-PT}.
\begin{lemma}\label{def-ob-cmg2}
	Let $\mathcal{X}$ be a 3-dimensional smooth projective Deligne-Mumford stack over $\mathbb{C}$. Assume that $[(\mathcal{F},\varphi)]$ is a point in the moduli space $\overline{M}^s:=M^s_{\mathcal{X}/\mathbb{C}}(\mathcal{O}_{\mathcal{X}},P,\delta)$ and   $\bar{\mathbf{I}}^{\bullet}:=\{\mathcal{O}_{\mathcal{X}} \xrightarrow{\varphi} \mathcal{F}\}$  is a complex concentrated in degree 0 and 1. Then we have
	\ben
	\mathcal{E}xt^{\leq-1}(\bar{\mathbf{I}}^{\bullet},\bar{\mathbf{I}}^{\bullet})=0=\mathcal{E}xt^{\leq-1}(\bar{\mathbf{I}}^{\bullet},\mathcal{O}_{\mathcal{X}})\;\;\;\; \mbox{and}\;\;\;\; \mathcal{H}om(\bar{\mathbf{I}}^{\bullet},\bar{\mathbf{I}}^{\bullet})=\mathcal{O}_{\mathcal{X}}=\mathcal{H}om(\bar{\mathbf{I}}^{\bullet},\mathcal{O}_{\mathcal{X}}).
	\een
\end{lemma}
\begin{proof}
	Notice that 
	for 1-dimensional pure sheaf $\mathcal{F}$, we have
	\ben
	\mathcal{H}om(\mathcal{F},\mathcal{O}_{\mathcal{X}})\otimes\omega_{\mathcal{X}}=\mathcal{H}om(\mathcal{F},\omega_{\mathcal{X}})=0\;\;\; \mbox{and}\;\;\; \mathcal{E}xt^{1}(\mathcal{F},\mathcal{O}_{\mathcal{X}})\otimes\omega_{\mathcal{X}}=\mathcal{E}xt^{1}(\mathcal{F},\omega_{\mathcal{X}})=0
	\een
	by [\cite{BS}, Proposition C.1  and Lemma B.4].
	Then  $\mathcal{H}om(\mathcal{F},\mathcal{O}_{\mathcal{X}})=0=\mathcal{E}xt^{1}(\mathcal{F},\mathcal{O}_{\mathcal{X}})$. And the proof is completed by following the similar argument in the proofs of [\cite{PT1}, Lemma 1.15 and Lemma 1.20].
\end{proof}
The similar argument for the proof of [\cite{PT1}, Proposition 1.21]  shows that a point $[(\mathcal{F},\varphi)]$  in  $\overline{M}^s$ can be taken as an object $\bar{\mathbf{I}}^{\bullet}=\{\mathcal{O}_{\mathcal{X}} \xrightarrow{\varphi} \mathcal{F}\}$ in $D^b(\mathcal{X})$. It follows from Lemma \ref{def-ob-cmg2} and the local-to-global spectral sequence that 
\ben
\mathrm{Ext}^{\leq-1}(\bar{\mathbf{I}}^{\bullet},\bar{\mathbf{I}}^{\bullet})=0\;\;\;\mbox{and}\;\;\;\mathrm{Hom}(\bar{\mathbf{I}}^{\bullet},\bar{\mathbf{I}}^{\bullet})=\mathbb{C}.
\een

By Remark \ref{def-ob-compare}, the first order or infinitesimal deformation theory of the complex $\bar{\mathbf{I}}^\bullet$ is governed by $\mathrm{Ext}^1(\bar{\mathbf{I}}^\bullet,\bar{\mathbf{I}}^\bullet)$ and $\mathrm{Ext}^2(\bar{\mathbf{I}}^\bullet,\bar{\mathbf{I}}^\bullet)$. Now, we have the following result on traceless $\mathrm{Ext}$ groups $\mathrm{Ext}^\bullet(\bar{\mathbf{I}}^\bullet,\bar{\mathbf{I}}^\bullet)_{0}$ (see [\cite{Tho}, Section 3] for the definition of the trace map).

\begin{lemma}\label{def-ob-cmg3}
Let $\mathcal{X}$ be a 3-dimensional smooth projective Deligne-Mumford stack over $\mathbb{C}$. Assume that $[(\mathcal{F},\varphi)]$ is a point in the moduli space $\overline{M}^s$ and   $\bar{\mathbf{I}}^{\bullet}=\{\mathcal{O}_{\mathcal{X}} \xrightarrow{\varphi} \mathcal{F}\}$  is a complex concentrated in degree 0 and 1. Then we have
\ben
\mathrm{Ext}^{i}(\bar{\mathbf{I}}^{\bullet},\bar{\mathbf{I}}^{\bullet})_{0}=0 \;\;\; \mbox{if } i\neq1,2.
\een
\end{lemma}
\begin{proof}
It remains to show that $\mathrm{Ext}^3(\bar{\mathbf{I}}^{\bullet},\bar{\mathbf{I}}^{\bullet})_{0}=0$. By Lemma \ref{def-ob-cmg2} and Lemma \ref{hom-alg}, we have
\ben
\mathcal{E}xt^{\leq-1}(\bar{\mathbf{I}}^{\bullet},\bar{\mathbf{I}}^{\bullet}\otimes\omega_{\mathcal{X}})=0;\;\;\;\mathcal{H}om(\bar{\mathbf{I}}^{\bullet},\bar{\mathbf{I}}^{\bullet}\otimes\omega_{\mathcal{X}})\cong\omega_{\mathcal{X}}
\een
Using the local-to-global spectral sequence, we have
\ben
\mathrm{Hom}(\bar{\mathbf{I}}^{\bullet},\bar{\mathbf{I}}^{\bullet}\otimes\omega_{\mathcal{X}})\cong H^0(\omega_{\mathcal{X}})
\een
By Lemma \ref{Serre-duality}, we have
\ben
\mathrm{Ext}^3(\bar{\mathbf{I}}^{\bullet},\bar{\mathbf{I}}^{\bullet})\cong H^3(\mathcal{O}_{\mathcal{X}})
\een
which implies  $\mathrm{Ext}^3(\bar{\mathbf{I}}^{\bullet},\bar{\mathbf{I}}^{\bullet})_{0}=0$.
\end{proof}	
Next, we  follow the similar argument in [\cite{PT1,HT09,GT,Zhou1}] to  construct a perfect obstruction theory with fixed determinant for $\overline{M}^s$.
Since $\overline{M}^s$ is a fine moduli space, there is a universal  complex
\ben
\bar{\mathbb{I}}^\bullet=\{\mathcal{O}_{\mathcal{X}\times \overline{M}^s}\to\mathbb{F}\}\in\mathrm{D}^b(\mathcal{X}\times \overline{M}^s).
\een
where $\mathbb{F}$ is flat over $\overline{M}^s$.
Let $\overline{\pi}_{\overline{M}^s}:\mathcal{X}\times \overline{M}^s\to \overline{M}^s$ and $\overline{\pi}_{\mathcal{X}}:\mathcal{X}\times \overline{M}^s\to\mathcal{X}$ be the projections. Since $\overline{M}^s$ is a projective scheme  by Theorem \ref{main-result} and Lemma \ref{no-str-semi},    we take a finite complex of locally free sheaves $A^\bullet$ resolving $\bar{\mathbb{I}}^\bullet$ such that
\ben
R\mathcal{H}om(\bar{\mathbb{I}}^{\bullet},\bar{\mathbb{I}}^{\bullet})\cong (A^\bullet)^\vee\otimes A^\bullet\cong\mathcal{O}_{\mathcal{X}\times \overline{M}^s}\oplus ((A^\bullet)^\vee\otimes A^\bullet)_{0}.
\een
As in [\cite{PT1}, Section 2.3] or [\cite{HT09}, Section 4.2],
we have $R\mathcal{H}om(\bar{\mathbb{I}}^{\bullet},\bar{\mathbb{I}}^{\bullet})\cong\mathcal{O}_{\mathcal{X}\times \overline{M}^s}\oplus R\mathcal{H}om(\bar{\mathbb{I}}^{\bullet},\bar{\mathbb{I}}^{\bullet})_{0}$  and $R\mathcal{H}om(\bar{\mathbb{I}}^{\bullet},\bar{\mathbb{I}}^{\bullet})_{0}\cong((A^\bullet)^\vee\otimes A^\bullet)_{0}$.
Let $\mathbb{L}_{\mathcal{X}\times \overline{M}^s}$ be the truncated cotangent complex (two-term complex quasi-isomorphic to the truncation $\tau^{\geq-1}L^\bullet_{\mathcal{X}\times \overline{M}^s}$ of Illusie's cotangent complex $L^\bullet_{\mathcal{X}\times \overline{M}^s}$, see also [\cite{HT09}, Definition 2.1]), we have the truncated Atiyah class of $\bar{\mathbb{I}}^{\bullet}$
\ben
\mathrm{At}(\bar{\mathbb{I}}^{\bullet})\in\mathrm{Ext}^1(\bar{\mathbb{I}}^{\bullet},\bar{\mathbb{I}}^{\bullet}\otimes\mathbb{L}_{\mathcal{X}\times \overline{M}^s})
\een
Composing the map $\bar{\mathbb{I}}^{\bullet}\to\bar{\mathbb{I}}^{\bullet}\otimes\mathbb{L}_{\mathcal{X}\times \overline{M}^s}[1]$ with the projection $\mathbb{L}_{\mathcal{X}\times \overline{M}^s}\to \overline{\pi}_{\overline{M}^s}^*\mathbb{L}_{\overline{M}^s}$ and restricting to the traceless part of $R\mathcal{H}om(\bar{\mathbb{I}}^{\bullet},\bar{\mathbb{I}}^{\bullet})$, we have a class in
\ben
\mathrm{Ext}^1(R\mathcal{H}om(\bar{\mathbb{I}}^{\bullet},\bar{\mathbb{I}}^{\bullet})_{0},\overline{\pi}_{\overline{M}^s}^*\mathbb{L}_{\overline{M}^s}).
\een 
Tensoring the map $R\mathcal{H}om(\bar{\mathbb{I}}^{\bullet},\bar{\mathbb{I}}^{\bullet})_{0}\to\overline{\pi}_{\overline{M}^s}^*\mathbb{L}_{\overline{M}^s}[1]$ with $\overline{\pi}_{\mathcal{X}}^*\omega_{\mathcal{X}}$, together with smooth Serre duality in [\cite{Nir1}, Theorem 2.22], we have
\ben
R\mathcal{H}om(\bar{\mathbb{I}}^{\bullet},\bar{\mathbb{I}}^{\bullet})_{0}\otimes\overline{\pi}_{\mathcal{X}}^*\omega_{\mathcal{X}}\to\overline{\pi}_{\overline{M}^s}^*\mathbb{L}_{\overline{M}^s}\otimes\overline{\pi}_{\mathcal{X}}^*\omega_{\mathcal{X}}[1]=\overline{\pi}_{\overline{M}^{s}}^!\mathbb{L}_{\overline{M}^s}[-2].
\een
By Serre duality for Deligne-Mumford stacks in [\cite{Nir1}, Corollary 2.10], we have
\ben
\Phi: E^\bullet:=R\overline{\pi}_{\overline{M}^{s}*}(R\mathcal{H}om(\bar{\mathbb{I}}^{\bullet},\bar{\mathbb{I}}^{\bullet})_{0}\otimes\overline{\pi}_{\mathcal{X}}^*\omega_{\mathcal{X}})[2]\to\mathbb{L}_{\overline{M}^s}.
\een

\begin{theorem}\label{perf-ob1}
In the notation of [\cite{BF}], the map $\Phi$ is a perfect obstruction theory for $\overline{M}^s$.
\end{theorem}
\begin{proof}
We combine the argument in the proof of [\cite{HT09}, Theorem 4.1] and the one in [\cite{Zhou1}, Theorem 6.2].
Let $T\to\overline{T}$ be an extension of $k$-schemes with ideal sheaf $J$ such that $J^2=0$ and $g: T\to \overline{M}^s$ be a morphism of $k$-schemes. An element $\omega(g)\in\mathrm{Ext}^1(g^*\mathbb{L}_{\overline{M}^s},J)$ is given by composing the natural map $g^*\mathbb{L}_{\overline{M}^s}\to\mathbb{L}_{T}$ with the truncated Kodaira-Spencer class of $T\subset \overline{T}$
\ben
\kappa(T/\overline{T})\in\mathrm{Ext}^1(\mathbb{L}_{T},J).
\een
The composition of $\omega(g)$ with $g^*\Phi$ gives an element 
\ben
\Phi^*\omega(g)\in\mathrm{Ext}^1(g^*E^\bullet,J).
\een
By the argument in the proof of [\cite{BF}, Theorem 4.5],  to show that $\Phi$ is an obstruction theory, we should prove that the obstruction $\Phi^*\omega(g)$ vanishes if and only if an extension $\overline{g}$ of $g$ to $\overline{T}$ exists, and when $\Phi^*\omega(g)=0$ the space of extensions form  a torsor under $\mathrm{Ext}^0(g^*E^\bullet,J)$.

Set $\hat{g}:=\mathrm{id}_{\mathcal{X}}\times g:\mathcal{X}\times T\to\mathcal{X}\times \overline{M}^s$ and let $p:\mathcal{X}\times T\to\mathcal{X}$ and $q:\mathcal{X}\times T\to T$ be the natural projections. Then we have 
\ben
g^*(R\overline{\pi}_{\overline{M}^{s}*}(R\mathcal{H}om(\bar{\mathbb{I}}^{\bullet},\bar{\mathbb{I}}^{\bullet})_{0}\otimes\overline{\pi}_{\mathcal{X}}^*\omega_{\mathcal{X}}))\cong Rq_{*}(R\mathcal{H}om(\hat{g}^*\bar{\mathbb{I}}^{\bullet},\hat{g}^*\bar{\mathbb{I}}^{\bullet})_{0}\otimes p^*\omega_{\mathcal{X}}).
\een
Using Serre duality in [\cite{Nir1}, Corollary 2.10] for the map $q$, we have
\ben
\Phi^*\omega(g)\in\mathrm{Ext}^1(Rq_{*}(R\mathcal{H}om(\hat{g}^*\bar{\mathbb{I}}^{\bullet},\hat{g}^*\bar{\mathbb{I}}^{\bullet})_{0}\otimes p^*\omega_{\mathcal{X}})[2],J)\cong\mathrm{Ext}^2(\hat{g}^*\bar{\mathbb{I}}^{\bullet},\hat{g}^*\bar{\mathbb{I}}^{\bullet}\otimes q^*J)_{0}
\een
and 
\ben
\mathrm{Ext}^0(g^*E^\bullet,J)\cong\mathrm{Ext}^0(Rq_{*}(R\mathcal{H}om(\hat{g}^*\bar{\mathbb{I}}^{\bullet},\hat{g}^*\bar{\mathbb{I}}^{\bullet})_{0}\otimes p^*\omega_{\mathcal{X}})[2],J)\cong\mathrm{Ext}^1(\hat{g}^*\bar{\mathbb{I}}^{\bullet},\hat{g}^*\bar{\mathbb{I}}^{\bullet}\otimes q^*J)_{0}.
\een
Then the obstruction $\Phi^*\omega(g)$ is the traceless part of the obstruction class which is a product of a truncated Atiyah class  $\overline{\mathrm{At}}(\hat{g}^*\bar{\mathbb{I}}^{\bullet})$ in $\mathrm{Ext}^1(\hat{g}^*\bar{\mathbb{I}}^{\bullet},\hat{g}^*\bar{\mathbb{I}}^{\bullet}\otimes q^*\mathbb{L}_{T})$ and a truncated Kodaira-Spencer class $q^*\kappa(T/\overline{T})$ in $\mathrm{Ext}^1(q^*\mathbb{L}_{T},q^*J)$, that is, 
\ben
\Phi^*\omega(g)=\left((\mathrm{id}_{\hat{g}^*\bar{\mathbb{I}}^{\bullet}}\otimes q^*\kappa(T/\overline{T}))\circ\overline{\mathrm{At}}(\hat{g}^*\bar{\mathbb{I}}^{\bullet})\right)_{0}
\een
As in the argument of the proof in [\cite{BF}, Theorem 4.5], to prove that $\Phi$ is an obstruction theory we only need to consider the case when both $T$ and $\overline{T}$ are affine. The similar argument in [\cite{HT09}, Section 3] shows that the obstruction class to extending $\hat{g}^*\bar{\mathbb{I}}^{\bullet}$ from $\mathcal{X}\times T$ to $\mathcal{X}\times\overline{T}$ defined as a square of the differential in the proof of Theorem \ref{def-ob3} (see also [\cite{Lieb}, Section 3.3]) up to a sign is exactly the obstruction class which is the product of a truncated Atiyah class and a truncated Kodaira-Spencer class. Using Lemma \ref{finite-res} to resolve the extension of $\hat{g}^*\bar{\mathbb{I}}^{\bullet}$ on $\mathcal{X}\times \overline{T}$ (flat over $\overline{T}$) and its restriction $\hat{g}^*\bar{\mathbb{I}}^{\bullet}$ on $\mathcal{X}\times T$, the similar argument in the proof of [\cite{Tho}, Theorem 3.23]  shows that the traceless part is the deformation and obstruction with  fixed determinant $\det(\hat{g}^*\bar{\mathbb{I}}^{\bullet})$. By Theorem \ref{def-ob3}, $\Phi^*\omega(g)=0$ if and only if there is an extension of the complex $\hat{g}^*\bar{\mathbb{I}}^{\bullet}$ from $\mathcal{X}\times T$ to $\mathcal{X}\times\overline{T}$, and in that case the space of extensions with fixed determinant forms a torsor under $\mathrm{Ext}^1(\hat{g}^*\bar{\mathbb{I}}^{\bullet},\hat{g}^*\bar{\mathbb{I}}^{\bullet}\otimes q^*J)_{0}$. Since $\overline{M}^s$ is a fine moduli space and $\bar{\mathbb{I}}^{\bullet}$ is a universal complex, then an extension $\overline{g}$
 of $g$ to $\overline{T}$  exists if and only if $\hat{g}^*\bar{\mathbb{I}}^{\bullet}$ extends from $\mathcal{X}\times T$ to $\mathcal{X}\times\overline{T}$. Thus $\Phi$ is an obstruction theory by  [\cite{BF}, Theorem 4.5].

Next, we will show that the complex $R\overline{\pi}_{\overline{M}^{s}*}R\mathcal{H}om(\bar{\mathbb{I}}^{\bullet},\bar{\mathbb{I}}^{\bullet})_{0}$ is quasi-isomorphic to a perfect $2$-term complex of locally free sheaves with amplitude contained in $[1,2]$. Since $R\mathcal{H}om(\bar{\mathbb{I}}^{\bullet},\bar{\mathbb{I}}^{\bullet})_{0}\cong((A^\bullet)^\vee\otimes A^\bullet)_{0}$ is a bounded complex of coherent sheaves, as in the $2$-dimensional case, one can take a finite very negative locally free resolution $B^\bullet$ of the bounded complex $R\mathcal{H}om(R\mathcal{H}om(\bar{\mathbb{I}}^{\bullet},\bar{\mathbb{I}}^{\bullet})_{0},\mathcal{O}_{\mathcal{X}\times \overline{M}^s})$. Then
\ben
R\overline{\pi}_{\overline{M}^{s}*}R\mathcal{H}om(\bar{\mathbb{I}}^{\bullet},\bar{\mathbb{I}}^{\bullet})_{0}\cong R\overline{\pi}_{\overline{M}^{s}*}(B^\bullet)^\vee\cong\overline{\pi}_{\overline{M}^{s}*}(B^\bullet)^\vee
\een
 is a finite complex of locally free sheaves. Denoted this complex by $N^\bullet$.
 By cohomology and base change theorem (cf. [\cite{Hall}])
and Lemma \ref{def-ob-cmg3}, the complex $N^\bullet$ has cohomology only in degree 1 and 2. By the standard argument in the proof of [\cite{PT1}, Lemma 2.1] or [\cite{HT09}, Lemma 4.2] for triming the  complex $N^\bullet$, one can show that $R\overline{\pi}_{\overline{M}^{s}*}R\mathcal{H}om(\bar{\mathbb{I}}^{\bullet},\bar{\mathbb{I}}^{\bullet})_{0}$ is quasi-isomorphic to 2-term complex of locally free sheaves concentrated only in degree 1 and 2. This implies that $R\overline{\pi}_{\overline{M}^{s}*}(R\mathcal{H}om(\bar{\mathbb{I}}^{\bullet},\bar{\mathbb{I}}^{\bullet})_{0}\otimes\overline{\pi}_{\mathcal{X}}^*\omega_{\mathcal{X}})[2]$ is perfect of perfect amplitude contained in $[-1,0]$ in the notation of [\cite{BF}]. Then $\Phi$ is perfect.
\end{proof}
\begin{corollary}\label{vir-exi-3}
Let $\mathcal{X}$ be a 3-dimensional smooth projective Deligne-Mumford stack  over $\mathbb{C}$. Assume that  polynomials $\delta$ and  $P$ satisfy $\deg\,\delta\geq\deg \,P=1$.
Then there exists a virtual fundamental class $[\overline{M}^s]^{\mathrm{vir}}\in A_{\mathrm{vdim}}(\overline{M}^s)$ of virtual dimension $\mathrm{vdim}=\mathrm{rk}(E^\bullet)$ where $\overline{M}^s:=M^s_{\mathcal{X}/\mathbb{C}}(\mathcal{O}_{\mathcal{X}},P,\delta)$ and $E^\bullet:=R\overline{\pi}_{\overline{M}^{s}*}(R\mathcal{H}om(\bar{\mathbb{I}}^{\bullet},\bar{\mathbb{I}}^{\bullet})_{0}\otimes\overline{\pi}_{\mathcal{X}}^*\omega_{\mathcal{X}})[2]$.
\end{corollary}
\begin{remark}
Since $\overline{M}^s$ is a projective scheme, then it is proper and  $[\overline{M}^s]^{\mathrm{vir}}$ can be integrated. By Remark \ref{stacky-PT}, the intersection theory on  $[\overline{M}^s]^{\mathrm{vir}}$ actually produces a stacky version of Pandharipande-Thomas theory.
\end{remark}

\subsection{Orbifold Pandharipande-Thomas theory}

In [\cite{BCR}], the authors give a definition of orbifold Pandharipande-Thomas invariants for  Calabi-Yau 3-orbifolds, which are smooth projective Deligne-Mumford stacks $\mathcal{Y}$ with generically trivial stabilizer groups satisfying $\omega_{\mathcal{Y}}\cong\mathcal{O}_{\mathcal{Y}}$ and $H^1(\mathcal{Y},\mathcal{O}_{\mathcal{Y}})=0$. Compared with their definition empolying Behrend's weighted Euler characteristic [\cite{Beh}], our apparoach here is to take integrations against virtual fundamental classes for more general cases. Let $\mathcal{X}$ be a smooth projective Deligne-Mumford stack of dimension 3 over $\mathbb{C}$ with a moduli scheme $\pi:\mathcal{X}\to X$ and a polarization $(\mathcal{E},\mathcal{O}_{X}(1))$. We first recall some notation in [\cite{BCR}, Section 2]. Let $\mathrm{Perf}(\mathcal{X})$ be the subcategory of $\mathrm{D}^b(\mathcal{X})$ with objects being perfect complexes, that is, those locally isomorphic to a bounded 
complex of locally free sheaves. By Lemma \ref{finite-res}, we have $\mathrm{Perf}(\mathcal{X})=\mathrm{D}^b(\mathcal{X})$. Define the Euler pairing
\ben
\chi(F^\bullet,G^\bullet)=\sum_{i}(-1)^i\dim\mathrm{Hom}(F^\bullet,G^\bullet[i])
\een
for any $F^\bullet\in\mathrm{Perf}(\mathcal{X})$ and $G^\bullet\in\mathrm{D}^b(\mathcal{X})$. The complex $G^\bullet$ is called numerically trivial if $\chi(F^\bullet,G^\bullet)=0$ for all $F^\bullet\in\mathrm{Perf}(\mathcal{X})$. Denote by $K(\mathcal{X})=K(D^b(\mathcal{X}))=K(\mathrm{Coh}(\mathcal{X}))$  the Grothendieck group of $\mathcal{X}$. The numerical Grothendieck group  $N(\mathcal{X})$ is define to be the quotient of $K(\mathcal{X})$ by the subgroup generated by numerically trivial complexes. Let $\mathrm{Coh}_{\leq d}(\mathcal{X})\subset\mathrm{Coh}(\mathcal{X})$ be the subcategory of sheaves supported in dimension at most $d$. Define $N_{\leq d}(\mathcal{X})\subset N(\mathcal{X})$ as the subgroup generated by classes of sheaves in $\mathrm{Coh}_{\leq d}(\mathcal{X})$. Set $N_d(\mathcal{X})=N_{\leq d}(\mathcal{X})/N_{\leq d-1}(\mathcal{X})$. One can choose a splitting of $N_{\leq1}(\mathcal{X})$ as follows
\ben
N_{\leq1}(\mathcal{X})\cong N_{1}(\mathcal{X})\oplus N_{0}(\mathcal{X}).
\een
Given a class $\beta=(\beta_{1},\beta_{0})\in N_{\leq1}(\mathcal{X})$ with
a sheaf $\mathcal{F}\in \mathrm{Coh}_{\leq1}(\mathcal{X})$ satisfying $[\mathcal{F}]=\beta$ where $\beta_{i}\in N_{i}(\mathcal{X})$ for $i=0,1$, then the modified Hilbert polynomial of  $\mathcal{F}$ is
\ben
P_{\mathcal{E}}(\mathcal{F})(m)=\chi(\mathcal{X},\mathcal{F}\otimes\mathcal{E}^\vee\otimes\pi^*\mathcal{O}_{X}(m)):=l(\mathcal{F})\cdot m+\deg(\mathcal{F}).
\een
where $\deg\mathcal{F}=\chi(F_{\mathcal{E}}(\mathcal{F}))$.
By the definition of numerical Grothendieck group, $P_{\mathcal{E}}(\mathcal{F})(m)$ is independent of the choice of representative $\mathcal{F}$ in $[\mathcal{F}]$. Thus $P_{\mathcal{E}}(\beta)(m)$, $l(\beta)$ and $\deg(\beta)$ (or $\chi(F_{\mathcal{E}}(\beta))$) are well defined for any class $\beta\in N_{\leq1}(\mathcal{X})$  and we have $P_{\mathcal{E}}(\beta)(m)=l(\beta)\cdot m+\deg(\beta)$. 

Notice that the sheaf $\mathcal{F}$ underlying a $\delta$-stable pair $(\mathcal{F},\varphi)$ is pure of dimension one. Consider any fixed nonzero class $\beta\in N_{\leq1}(\mathcal{X})$, we have the corresponding degree one polynomial $P_{\mathcal{E}}(\beta)(m)=P(m)$. For such a polynomial $P$, and let $\delta$ be a rational polynomial with positive leading coefficient and $\deg\,\delta\geq1=\deg\,P$, we have  the subfunctor $\mathcal{M}^{(s)s}_{\mathcal{X}/\mathbb{C}}(\mathcal{O}_{\mathcal{X}},\beta,\delta)$ of  $\mathcal{M}^{(s)s}_{\mathcal{X}/\mathbb{C}}(\mathcal{O}_{\mathcal{X}},P,\delta)$ defined as follows. For any $\mathbb{C}$-scheme $S$ of finite type, $\mathcal{M}^{(s)s}_{\mathcal{X}/\mathbb{C}}(\mathcal{O}_{\mathcal{X}},\beta,\delta)(S)$ is the set of isomorphism classes of flat families of nondegenerate $\delta$-(semi)stable pairs $(\mathcal{F},\varphi)$ with the fixed numerical class $\beta$  parametrized by a scheme $S$, that is, such a flat family $(\mathcal{F},\varphi)$ satisfies that for each point $s\in S$, the pair $(\mathcal{F}_{s},\varphi|_{(\pi_{\mathcal{X}}^*\mathcal{O}_{\mathcal{X}})_{s}})$ is a $\delta$-(semi)stable pair with $[\mathcal{F}_{s}]=\beta$ (hence with the modified Hilbert polynomial $P_{\check{\pi}_{s}^*\mathcal{E}}(\mathcal{F}_{s})=P$ where $\check{\pi}_{s}:\mathcal{X}\times\mathrm{Spec}(k(s))\to\mathcal{X}$ is the projection). See Definition \ref{moduli-functor} for more details. 

Following the same argument in Section 3 and Section 4 for the functor $\mathcal{M}^{(s)s}_{\mathcal{X}/\mathbb{C}}(\mathcal{O}_{\mathcal{X}},\beta,\delta)$ and the corresponding $\delta$-(semi)stable pairs with fixed numerical class $\beta$,  and noticing  Lemma \ref{no-str-semi}, one has

\begin{theorem}\label{moduli-spaces2}
Let $\mathcal{X}$ be a 3-dimensional smooth projective Deligne-Mumford stack over $\mathbb{C}$ with a moduli scheme $\pi:\mathcal{X}\to X$ and a polarization $(\mathcal{E},\mathcal{O}_{X}(1))$. Assume $0\neq\beta\in N_{\leq1}(\mathcal{X})$ and the rational polynomial $\delta$ satisfy $\deg\,\delta\geq1$.
	Then there is a projective scheme $\overline{M}^{s}_{\beta}:=M^{s}_{\mathcal{X}/\mathbb{C}}(\mathcal{O}_{\mathcal{X}},\beta,\delta)$, which is a fine moduli space for the moduli functor $\mathcal{M}^{s}_{\mathcal{X}/\mathbb{C}}(\mathcal{O}_{\mathcal{X}},\beta,\delta)$.
\end{theorem}

 Let $\tilde{\pi}_{\overline{M}^s_{\beta}}:\mathcal{X}\times \overline{M}^s_{\beta}\to \overline{M}^s_{\beta}$ and $\tilde{\pi}_{\mathcal{X}}:\mathcal{X}\times \overline{M}^s_{\beta}\to\mathcal{X}$ be the projections. Denote by $\bar{\mathbb{I}}^\bullet_{\beta}:=\{\mathcal{O}_{\mathcal{X}\times \overline{M}^{s}_{\beta}}\to\mathbb{F}\}\in\mathrm{D}^b(\mathcal{X}\times \overline{M}^{s}_{\beta})$ the universal complex where $\mathbb{F}$ is the universal sheaf on $\mathcal{X}\times \overline{M}^{s}_{\beta}$. It follows from the  same argument in Section 5.1 and Section 5.2 for the moduli space $\overline{M}^{s}_{\beta}$ that
\begin{theorem}\label{perf-ob2}
	In the sense of [\cite{BF}], the map 
	\ben
	\Phi_{\beta}:E^\bullet_{\beta}:=R\tilde{\pi}_{\overline{M}^s_{\beta}*}(R\mathcal{H}om(\bar{\mathbb{I}}^{\bullet}_{\beta},\bar{\mathbb{I}}^{\bullet}_{\beta})_{0}\otimes\tilde{\pi}_{\mathcal{X}}^*\omega_{\mathcal{X}})[2]\to\mathbb{L}_{\overline{M}^s_{\beta}}.
	\een
	is a perfect obstruction theory for $ \overline{M}^{s}_{\beta}:=M^{s}_{\mathcal{X}/\mathbb{C}}(\mathcal{O}_{\mathcal{X}},\beta,\delta)$. And there exists a virtual fundamental class $[ \overline{M}^{s}_{\beta}]^{\mathrm{vir}}\in A_{\mathrm{vdim}}( \overline{M}^{s}_{\beta})$ of virtual dimension $\mathrm{vdim}=\mathrm{rk}(E^\bullet_{\beta})$.
\end{theorem}
Set
\ben
S_{P,\delta}=\{\beta=[\mathcal{F}]:[(\mathcal{F},\varphi)]\in M^s_{\mathcal{X}/\mathbb{C}}(\mathcal{O}_{\mathcal{X}},P,\delta)\}.
\een
Then we have the following decomposition as a disjoint union
\ben
M^s_{\mathcal{X}/\mathbb{C}}(\mathcal{O}_{\mathcal{X}},P,\delta)=\coprod_{\beta\in S_{P,\delta}} M^{s}_{\mathcal{X}/\mathbb{C}}(\mathcal{O}_{\mathcal{X}},\beta,\delta).
\een

Next,  we will follow the similar definition of Donaldson-Thomas invariants in [\cite{Zhou1}, Section 6] to give a stacky version of Pandharipande-Thomas invariants. We briefly recall some notation and definitions in [\cite{Tse}, Section 2 and Appendix A] to define certain Chern character. Let $I\mathcal{X}$ be the inertia stack, which is defined to be the fiber product $I\mathcal{X}:=\mathcal{X}\times_{\Delta,\mathcal{X}\times\mathcal{X},\Delta}\mathcal{X}$ where $\Delta:\mathcal{X}\to\mathcal{X}\times\mathcal{X}$ is the diagonal map. Its underlying category has objects  of the form 
$
\mathrm{Ob}(I\mathcal{X})=\{(x,g)|x\in\mathrm{Ob}(\mathcal{X}), g\in\mathrm{Aut}_{\mathcal{X}}(x)\}
$. Then there is a natural projection $\pi_{0}:I\mathcal{X}\to\mathcal{X}$ with the map $\pi_{0}((x,g))=x$ on the level of objects and we have a  decomposition of $I\mathcal{X}$ as a disjoint union $I\mathcal{X}:=\coprod_{i\in\mathcal{J}}\mathcal{X}_{i}$ for some index set $\mathcal{J}$. We also have a canonical involution $\iota:I\mathcal{X}\to I\mathcal{X}$ with the map $\iota((x,g))=(x,g^{-1})$ on objects. For any $(x,g)\in\mathcal{X}_{i}$, one has a decomposition of tangent space $T_{x}\mathcal{X}=\bigoplus_{0\leq t<r_{i}}U^{(t)}$ where $U^{(t)}$ is an  eigenspace with eigenvalue $\zeta_{r_{i}}^t$, $0\leq t<r_{i}$ and  $\zeta_{r_{i}}=\exp(\frac{2\pi i}{r_{i}})$. Define $\mathrm{age}_{i}:=\frac{1}{r_{i}}\sum_{0\leq t<r_{i}}t\cdot \dim_{\mathbb{C}}U^{(t)}$.
Similarly, for any vector bundle $W$ on $I\mathcal{X}$,  there exists  a decomposition of $W$ as follows \ben
W=\bigoplus_{\zeta}W^{(\zeta)}
\een
 where $W^{(\zeta)}$ is an eigenbundle with the eigenvalue $\zeta$.
\begin{definition}
The map $\rho: K(I\mathcal{X})\to K(I\mathcal{X})_{\mathbb{C}}$ is defined to be
\ben
\rho(W):=\sum_{\zeta}\zeta W^{(\zeta)}\in K(I\mathcal{X})_{\mathbb{C}}.
\een
Define $\widetilde{\mathrm{ch}}: K(\mathcal{X})\to A^*(I\mathcal{X})_{\mathbb{C}}$ to be 
\ben
\widetilde{\mathrm{ch}}(V):=\mathrm{ch}(\rho(\pi_{0}^*V))
\een
where  $\mathrm{ch}$ is the usual Chern character.
\end{definition}
The orbifold or Chen-Ruan cohomology of $\mathcal{X}$ (cf. [\cite{AGV}, Section 7.3]) is defined as
\ben
A^*_{\mathrm{orb}}(\mathcal{X}):=\bigoplus_{i}A^{*-\mathrm{age}_{i}}(\mathcal{X}_{i})
\een
where $\mathrm{age}_{i}$ is the degree shift number. Define orbifold Chern character $\widetilde{\mathrm{ch}}^{\mathrm{orb}}: K(\mathcal{X})\to A^*_{\mathrm{orb}}(\mathcal{X})$ as
\ben
\widetilde{\mathrm{ch}}^{\mathrm{orb}}_{k}\bigg|_{\mathcal{X}_{i}}:=\widetilde{\mathrm{ch}}_{k-\mathrm{age}_{i}}\bigg|_{\mathcal{X}_{i}}.
\een
For any $\gamma\in A^l_{\mathrm{orb}}(\mathcal{X})$, define the operators
\ben
\widetilde{\mathrm{ch}}_{k+2}^{\mathrm{orb}}(\gamma): A_{*}(\overline{M}^s_{\beta})\to A_{*-k+1-l}(\overline{M}^s_{\beta})
\een
to be
\ben
\widetilde{\mathrm{ch}}_{k+2}^{\mathrm{orb}}(\gamma)(\xi):=\pi_{2*}\left(\widetilde{\mathrm{ch}}_{k+2}^{\mathrm{orb}}(\mathbb{F})\cdot \iota^*\pi_{1}^*\gamma\cap \pi_{2}^*\xi\right)
\een
where $\mathbb{F}$ is the universal complex on $\mathcal{X}\times \overline{M}^s_{\beta}$ and  the maps $\pi_{1}$, $\pi_{2}$ are the natural projections from $I\mathcal{X}\times \overline{M}^s_{\beta}$ to the first and second factor respectively. The operator $\widetilde{\mathrm{ch}}_{k+2}^{\mathrm{orb}}(\gamma)$ has the  degree changed as above  due to the identity $\mathrm{age}_{i}+\mathrm{age}_{\iota(i)}=\dim_{\mathbb{C}}\mathcal{X}-\dim_{\mathbb{C}}\mathcal{X}_{i}$ (cf. [\cite{CR2}, Lemma 3.2.1]). To follow the similar definition of Pandharipande-Thomas invariants on nonsingular projective 3-folds in [\cite{PT2}, Section 0.5], we use the notation  $\overline{M}^s_{n,\beta}:=M^{s}_{\mathcal{X}/\mathbb{C}}(\mathcal{O}_{\mathcal{X}},n,\beta,\delta)$ to denote $\overline{M}^s_{\beta}$, which is the moduli space of  orbifold PT  stable pairs  with $[\mathcal{F}]=\beta=(\beta_{1},\beta_{0})\in N_{\leq1}(\mathcal{X})$ and $\chi(F_{\mathcal{E}}(\mathcal{F}))=n$  by Remark \ref{stacky-PT}.
\begin{definition}\label{descen-PT}
Given $\gamma_{i}\in A_{\mathrm{orb}}^*(\mathcal{X})$, $1\leq i\leq r$, define the Pandharipande-Thomas invariants with descendents as
\ben
\bigg\langle\prod_{i=1}^r\tau_{k_{i}}(\gamma_{i})\bigg\rangle_{n,\beta}^{\mathcal{X}}&:=&\int_{[M^{s}_{\mathcal{X}/\mathbb{C}}(\mathcal{O}_{\mathcal{X}},n,\beta,\delta)]^{\mathrm{vir}}}\prod_{i=1}^r\widetilde{\mathrm{ch}}_{k_{i}+2}^{\mathrm{orb}}(\gamma_{i})\\
&=&\int_{M^{s}_{\mathcal{X}/\mathbb{C}}(\mathcal{O}_{\mathcal{X}},n,\beta,\delta)}\prod_{i=1}^r\widetilde{\mathrm{ch}}_{k_{i}+2}^{\mathrm{orb}}(\gamma_{i})([M^{s}_{\mathcal{X}/\mathbb{C}}(\mathcal{O}_{\mathcal{X}},n,\beta,\delta)]^{\mathrm{vir}})
\een
The partition function is defined to be 
\ben
Z_{\mathrm{PT},\beta_{1}}\bigg(\prod_{i=1}^r\tau_{k_{i}}(\gamma_{i})\bigg):=\sum_{n}\bigg\langle\prod_{i=1}^r\tau_{k_{i}}(\gamma_{i})\bigg\rangle_{n,\beta}^{\mathcal{X}}q^n
\een
\end{definition}
\begin{remark} 
In the proof of Lemma \ref{case2-bound}, we have 
a short exact sequence $0\to F_{\mathcal{E}}(\mathrm{im}\varphi)\to F_{\mathcal{E}}(\mathcal{F})\to F_{\mathcal{E}}(\mathrm{coker}\,\varphi)\to0$,  where
$\mathrm{coker}\,\varphi$ and $F_{\mathcal{E}}(\mathrm{coker}\,\varphi)$ are 0-dimensional. Then  we have 
\ben
n=\chi(F_{\mathcal{E}}(\mathcal{F}))=\chi(F_{\mathcal{E}}(\mathrm{im}\varphi))+\chi(F_{\mathcal{E}}(\mathrm{coker}\,\varphi))\geq\chi(F_{\mathcal{E}}(\mathrm{im}\varphi)).
\een
For a fixed $\beta_{1}$, we have no idea whether $\overline{M}^s_{n,\beta}$ is empty for $n$ very negative. If one can show the boundedness of $\chi(F_{\mathcal{E}}(\mathrm{im}\varphi))$ from below, then $Z_{\mathrm{PT},\beta_{1}}\big(\prod_{i=1}^r\tau_{k_{i}}(\gamma_{i})\big)$ is a Laurent series in $q$ and hence one can further study its rationality   as the one of PT stable pair invariants (e.g., [\cite{Toda,Bri2,PP3,PP4,PP5}]). Alternatively, one can define the partition function as in  [\cite{BCR}].
\end{remark}
 To conclude this section, we consider some special cases as follows.
If $\mathcal{X}$ is  a  3-dimensional smooth projective Deligne-Mumford stack over $\mathbb{C}$ satisfying $\omega_{\mathcal{X}}\cong\mathcal{O}_{\mathcal{X}}$. In this case, using Serre duality for Deligne-Mumford stacks, we have two isomorphisms
\ben
E^\bullet\xrightarrow{\theta} E^{\bullet\vee}[1];\;\;\;\;E^\bullet_{\beta}\xrightarrow{\theta_{\beta}} E_{\beta}^{\bullet\vee}[1]
\een
satisfying  $\theta^\vee[1]=\theta$ and $\theta_{\beta}^\vee[1]=\theta_{\beta}$. This shows that two perfect obstruction theories are symmetric in the sense of [\cite{Beh}]. Then we have $\mathrm{rk}(E^\bullet)=\mathrm{rk}(E^\bullet_{\beta})=0$. Then $[M^s_{\mathcal{X}/\mathbb{C}}(\mathcal{O}_{\mathcal{X}},P,\delta)]^{\mathrm{vir}}$ and $[M^s_{\mathcal{X}/\mathbb{C}}(\mathcal{O}_{\mathcal{X}},\beta,\delta)]^{\mathrm{vir}}$ are  0-cycles. Let $\nu_{\overline{M}^s}$ and $\nu_{\overline{M}^s_{\beta}}$ be the Behrend's  constructible functions in [\cite{Beh}] on $M^s_{\mathcal{X}/\mathbb{C}}(\mathcal{O}_{\mathcal{X}},P,\delta)$ and $M^{s}_{\mathcal{X}/\mathbb{C}}(\mathcal{O}_{\mathcal{X}},\beta,\delta)$ respectively. Since $M^s_{\mathcal{X}/\mathbb{C}}(\mathcal{O}_{\mathcal{X}},P,\delta)$ and $M^{s}_{\mathcal{X}/\mathbb{C}}(\mathcal{O}_{\mathcal{X}},\beta,\delta)$ are proper, by [\cite{Beh}, Theorem 4.18], we have the following 
\begin{definition}\label{CY-PT1}
Let $\mathcal{X}$ be a smooth projective Deligne-Mumford stack of dimension 3 satisfying $\omega_{\mathcal{X}}\cong\mathcal{O}_{\mathcal{X}}$. We define Pandharipande-Thomas invariants of $\mathcal{X}$ corresponding to $P$ and $\beta$  as follows
\ben
&&\mathrm{PT}(\mathcal{O}_{\mathcal{X}},P,\delta):=\chi(M^s_{\mathcal{X}/\mathbb{C}}(\mathcal{O}_{\mathcal{X}},P,\delta),\nu_{\overline{M}^s})=\deg([M^s_{\mathcal{X}/\mathbb{C}}(\mathcal{O}_{\mathcal{X}},P,\delta)]^{\mathrm{vir}}),\\
&&\mathrm{PT}(\mathcal{O}_{\mathcal{X}},\beta,\delta):=\chi(M^s_{\mathcal{X}/\mathbb{C}}(\mathcal{O}_{\mathcal{X}},\beta,\delta),\nu_{\overline{M}^s_{\beta}})=\deg([M^s_{\mathcal{X}/\mathbb{C}}(\mathcal{O}_{\mathcal{X}},\beta,\delta)]^{\mathrm{vir}}).
\een
\end{definition}
\begin{remark}\label{CY-PT2}
When $\mathcal{X}$ is a 3-dimensional Calabi-Yau orbifold, the invariant $\mathrm{PT}(\mathcal{O}_{\mathcal{X}},\beta,\delta)$ in the above definition is corresponding to $\mathrm{PT}(\mathcal{X})_{\beta}$ defined in [\cite{BCR}, Section 1 (1.4)].
\end{remark}

\end{document}